\documentclass[review,3p]{elsarticle}
\usepackage{amsthm}					% Provides extended theorem environments
\newtheorem{theorem:definition}{Definition}
\usepackage{scrhack}
\rmfamily 							% Serif font

\usepackage[dvipsnames]{xcolor}

\usepackage{mdwlist}
\usepackage[T1]{fontenc}
\usepackage[utf8]{inputenc}
\usepackage[english]{babel}	% Spelling control
\usepackage{lipsum}
\usepackage{eso-pic}
\usepackage{graphicx} 				% Include Graphics (pdf, png - but avoid jpg)
\graphicspath{{./Images/}}          % Path to images
\usepackage{tikz}
\usetikzlibrary{shapes,arrows}
\usetikzlibrary{decorations.pathreplacing}

\usepackage{tabularx} 				% Better design of tables
\usepackage{longtable}
\usepackage{booktabs}
\usepackage{graphicx}
\usepackage{multirow}
\usepackage{makecell}
\usepackage[font=small]{caption}
\usepackage{lscape}					% Modifies margins and rotates page contents but not page number.
\usepackage[active]{srcltx}
\usepackage{listings}
\usepackage{setspace} 				% Line settings
\usepackage{amssymb}				% Provides various useful mathematical symbols
\usepackage{amsfonts}
\usepackage{amsmath}
\usepackage{mathtools}
\usepackage{cases}
\usepackage{comment}
\usepackage{siunitx}
\usepackage{lineno}		% Adds line numbers
\modulolinenumbers[5]
\usepackage[colorinlistoftodos]{todonotes}

\usepackage{floatrow}
\newfloatcommand{capbtabbox}{table}[][0.5\textwidth]

\usepackage{epstopdf}				%Convert EPS to PDF using Ghostscript
\usepackage[nolist]{acronym}		% Manage acronyms
\usepackage{multirow}				%Adds an optional argument to the enumerate environment
\usepackage{enumerate}
\usepackage{nicefrac}				%Typesets fractions “nicely”
\usepackage{xfrac}					%Nicefrac’s facilities are provided in a cleaner way		
\usepackage{algpseudocode}
\DeclareMathOperator*{\argmin}{arg\,min}

\usepackage{booktabs}				%Enhances the quality of tables in LaTeX
\usepackage{bigstrut}				%Produces a strut
\usepackage{cancel}
\usepackage{caption}
\usepackage{subcaption}

\usepackage{csquotes}
\usepackage{helvet}
\usepackage{listings}
\usepackage[toc,page]{appendix}
\usepackage{wrapfig}
\usepackage{subcaption}
\usepackage{rotating}

\allowdisplaybreaks

\usepackage{verbatim} % for multi-line comments

\newcommand{\uproman}[1]{\uppercase\expandafter{\romannumeral#1}}

\usepackage{color, colortbl}
\usepackage{url}
\usepackage[colorlinks = true, linkcolor=CITE, citecolor=CITE, urlcolor=CITE]{hyperref}
\usepackage{cleveref}
\definecolor{CITE}{HTML}{0571B0}
\AtBeginDocument{\hypersetup{citecolor=CITE, linkcolor=CITE, citecolor=CITE, urlcolor=CITE}}
\usepackage{lipsum}
\usepackage{svg}

\usepackage{tikz}
\usetikzlibrary{arrows,petri,topaths,quotes}
\tikzset{invisible/.style={minimum width=0mm,inner sep=0mm,outer sep=0mm}}
\usepackage{subcaption} 
\definecolor{A}{gray}{0.95}
\definecolor{B}{gray}{0.85}
\definecolor{C}{gray}{0.75}
\definecolor{D}{gray}{0.65}
\definecolor{E}{gray}{0.55}
\definecolor{F}{gray}{0.45}
\definecolor{G}{gray}{0.35}
\definecolor{H}{gray}{0.25}
\definecolor{I}{gray}{0.15}
\usepackage{svg}
\def\checkmark{\tikz\fill[scale=0.4](0,.35) -- (.25,0) -- (1,.7) -- (.25,.15) -- cycle;}

\theoremstyle{definition}

\newtheorem{corollary}{Corollary}
\newtheorem{lemma}{Lemma}
\newtheorem{proposition}{Proposition}
\newtheorem{definition}{Definition}

\newtheorem{problem}{Problem}

\usepackage{paralist}
\usepackage{amsmath}
\usepackage{amssymb}
\usepackage{amssymb}
\usepackage[ruled,vlined, linesnumbered,noresetcount]{algorithm2e}
\SetKwRepeat{Do}{do}{while}

\makeatletter
\def\ps@pprintTitle{%
  \let\@oddhead\@empty
  \let\@evenhead\@empty
  \def\@oddfoot{\reset@font\hfil\thepage\hfil}
  \let\@evenfoot\@oddfoot
}
\makeatother

\journal{Computers \& Operations Research}

\biboptions{authoryear}

\newcommand{\TabBenchmarkGonazelezUniformFull}[1]{
\begin{table}[]
\small
\resizebox{\textwidth}{!}{%
\centering
  
\begin{tabular}{rr@{\hspace{20pt}}rrrrrr@{\hspace{20pt}}rrrrrr@{\hspace{20pt}}rrrrrr}
\toprule
& & \multicolumn{6}{c@{\hspace{20pt}}}{Speed ratio 1}  & \multicolumn{6}{c@{\hspace{20pt}}}{Speed ratio 2} & \multicolumn{6}{c}{Speed ratio 3} \\
\multicolumn{2}{c@{\hspace{20pt}}}{Instance} 
& \multicolumn{2}{c}{MD-SPP-H} & \multicolumn{2}{c}{IGH} & \multicolumn{2}{c@{\hspace{20pt}}}{$\Delta$ [\%]} & \multicolumn{2}{c}{MD-SPP-H} & \multicolumn{2}{c}{IGH} & \multicolumn{2}{c@{\hspace{20pt}}}{$\Delta$ [\%]} & \multicolumn{2}{c}{MD-SPP-H} & \multicolumn{2}{c}{IGH} & \multicolumn{2}{c}{$\Delta$ [\%]} \\
\cmidrule(lr{0.6cm}){1-2}
\cmidrule(lr){3-4} \cmidrule(lr){5-6} \cmidrule(lr{0.6cm}){7-8}
\cmidrule(lr){9-10} \cmidrule(lr){11-12} \cmidrule(lrr{0.6cm}){13-14} 
\cmidrule(lr){15-16} \cmidrule(lr){17-18} \cmidrule(lr){19-20}
\multicolumn{1}{c}{ID} & \multicolumn{1}{c@{\hspace{20pt}}}{$n$} & \multicolumn{1}{c}{$z_\text{best}$} & \multicolumn{1}{c}{$z_\text{avg}$} & \multicolumn{1}{c}{$z_\text{best}$} & \multicolumn{1}{c}{$z_\text{avg}$} & \multicolumn{1}{c}{$\Delta_\text{best}$} & \multicolumn{1}{c@{\hspace{20pt}}}{$\Delta_\text{avg}$} & \multicolumn{1}{c}{$z_\text{best}$} & \multicolumn{1}{c}{$z_\text{avg}$} & \multicolumn{1}{c}{$z_\text{best}$} & \multicolumn{1}{c}{$z_\text{avg.}$} & \multicolumn{1}{c}{$\Delta_\text{best}$} & \multicolumn{1}{c@{\hspace{20pt}}}{$\Delta_\text{avg}$} & \multicolumn{1}{c}{$z_\text{best}$} & \multicolumn{1}{c}{$z_\text{avg}$} & \multicolumn{1}{c}{$z_\text{best}$} & \multicolumn{1}{c}{$z_\text{avg}$} & \multicolumn{1}{c}{$\Delta_\text{best}$} & \multicolumn{1}{c}{$\Delta_\text{avg}$} \\
\midrule

71    & 50    & 405.8 & 407.0 & 458.6 & 502.0 & 11.5  & 18.9  & 281.5 & 282.7 & 312.1 & 325.3 & 9.8   & 13.1  & 246.0 & 251.5 & 270.8 & 275.4 & 9.1   & 8.7 \\
72    & 50    & 409.7 & 417.8 & 495.9 & 521.9 & 17.4  & 19.9  & 296.0 & 302.5 & 316.0 & 316.0 & 6.3   & 4.3   & 249.2 & 254.6 & 268.4 & 268.7 & 7.1   & 5.2 \\
73    & 50    & 403.8 & 406.7 & 505.3 & 522.2 & 20.1  & 22.1  & 290.6 & 294.2 & 316.5 & 322.5 & 8.2   & 8.8   & 243.5 & 245.0 & 258.4 & 262.3 & 5.8   & 6.6 \\
74    & 50    & 415.3 & 419.8 & 474.9 & 495.0 & 12.6  & 15.2  & 287.8 & 289.0 & 291.0 & 291.7 & 1.1   & 0.9   & 251.7 & 254.1 & 232.8 & 233.5 & -8.1  & -8.8 \\
75    & 50    & 409.7 & 414.4 & 518.5 & 543.1 & 21.0  & 23.7  & 291.5 & 298.0 & 314.6 & 320.7 & 7.4   & 7.1   & 256.1 & 259.5 & 248.7 & 252.5 & -2.9  & -2.8 \\
76    & 50    & 389.8 & 392.5 & 475.1 & 506.4 & 17.9  & 22.5  & 278.5 & 285.0 & 296.6 & 306.6 & 6.1   & 7.0   & 249.4 & 250.5 & 238.4 & 243.4 & -4.6  & -2.9 \\
77    & 50    & 416.0 & 423.6 & 472.6 & 505.6 & 12.0  & 16.2  & 295.3 & 301.7 & 311.6 & 317.2 & 5.2   & 4.9   & 253.0 & 261.1 & 266.7 & 266.7 & 5.1   & 2.1 \\
78    & 50    & 400.7 & 402.0 & 466.4 & 491.4 & 14.1  & 18.2  & 275.9 & 278.4 & 290.9 & 291.3 & 5.2   & 4.4   & 232.2 & 235.7 & 238.3 & 239.2 & 2.6   & 1.5 \\
79    & 50    & 370.4 & 376.2 & 530.1 & 539.2 & 30.1  & 30.2  & 269.1 & 273.8 & 311.2 & 311.7 & 13.5  & 12.2  & 235.3 & 237.1 & 242.8 & 244.2 & 3.1   & 2.9 \\
80    & 50    & 379.5 & 385.1 & 437.6 & 467.5 & 13.3  & 17.6  & 269.3 & 271.8 & 281.9 & 286.5 & 4.5   & 5.1   & 228.1 & 228.1 & 223.1 & 223.1 & -2.3  & -2.3 \\\addlinespace
81    & 75    & 453.1 & 459.8 & 566.2 & 589.4 & 20.0  & 22.0  & 311.9 & 321.7 & 365.7 & 368.7 & 14.7  & 12.7  & 258.5 & 265.9 & 296.1 & 297.0 & 12.7  & 10.5 \\
82    & 75    & 416.8 & 429.9 & 553.6 & 567.7 & 24.7  & 24.3  & 291.9 & 303.6 & 317.7 & 321.8 & 8.1   & 5.7   & 251.4 & 254.4 & 253.7 & 253.7 & 0.9   & -0.3 \\
83    & 75    & 420.1 & 423.8 & 549.4 & 567.0 & 23.5  & 25.3  & 291.8 & 296.4 & 319.4 & 325.6 & 8.6   & 9.0   & 241.3 & 243.5 & 238.7 & 240.2 & -1.1  & -1.4 \\
84    & 75    & 446.0 & 454.5 & 611.3 & 666.3 & 27.1  & 31.8  & 316.8 & 321.8 & 381.5 & 388.4 & 16.9  & 17.1  & 256.9 & 261.6 & 310.1 & 319.2 & 17.1  & 18.0 \\
85    & 75    & 475.9 & 478.5 & 616.4 & 651.2 & 22.8  & 26.5  & 338.3 & 344.9 & 389.8 & 395.7 & 13.2  & 12.8  & 280.4 & 286.1 & 302.5 & 305.4 & 7.3   & 6.3 \\
86    & 75    & 452.3 & 468.5 & 601.6 & 616.9 & 24.8  & 24.1  & 314.2 & 326.0 & 368.2 & 371.2 & 14.7  & 12.2  & 263.3 & 269.6 & 305.5 & 309.5 & 13.8  & 12.9 \\
87    & 75    & 453.7 & 463.3 & 573.0 & 605.9 & 20.8  & 23.5  & 321.6 & 328.8 & 362.4 & 376.1 & 11.3  & 12.6  & 271.7 & 274.7 & 284.9 & 287.5 & 4.6   & 4.4 \\
88    & 75    & 451.5 & 458.2 & 627.3 & 644.8 & 28.0  & 28.9  & 320.4 & 330.5 & 351.4 & 360.7 & 8.8   & 8.4   & 261.7 & 268.2 & 283.9 & 286.5 & 7.8   & 6.4 \\
89    & 75    & 406.9 & 415.6 & 584.4 & 611.3 & 30.4  & 32.0  & 292.7 & 298.2 & 345.3 & 348.7 & 15.2  & 14.5  & 250.9 & 255.2 & 275.3 & 275.5 & 8.9   & 7.4 \\
90    & 75    & 460.6 & 469.3 & 558.6 & 582.3 & 17.5  & 19.4  & 319.8 & 322.5 & 325.7 & 336.1 & 1.8   & 4.1   & 272.5 & 276.5 & 240.9 & 244.5 & -13.1 & -13.1 \\\addlinespace
91    & 100   & 512.4 & 529.5 & 676.3 & 701.3 & 24.2  & 24.5  & 346.3 & 359.4 & 395.2 & 403.8 & 12.4  & 11.0  & 287.5 & 296.5 & 302.7 & 307.8 & 5.0   & 3.7 \\
92    & 100   & 484.9 & 493.1 & 652.3 & 674.1 & 25.7  & 26.9  & 337.8 & 345.4 & 396.8 & 409.8 & 14.9  & 15.7  & 276.2 & 279.5 & 304.6 & 311.0 & 9.3   & 10.1 \\
93    & 100   & 483.5 & 495.6 & 652.5 & 697.2 & 25.9  & 28.9  & 325.3 & 339.6 & 388.4 & 394.9 & 16.2  & 14.0  & 267.3 & 273.0 & 303.3 & 306.3 & 11.9  & 10.9 \\
94    & 100   & 480.6 & 496.4 & 639.7 & 667.2 & 24.9  & 25.6  & 335.8 & 344.9 & 375.3 & 382.0 & 10.5  & 9.7   & 283.5 & 285.0 & 307.9 & 308.1 & 7.9   & 7.5 \\
95    & 100   & 510.0 & 521.2 & 691.2 & 739.3 & 26.2  & 29.5  & 333.8 & 352.0 & 416.0 & 421.9 & 19.8  & 16.6  & 274.9 & 282.2 & 299.1 & 301.1 & 8.1   & 6.3 \\
96    & 100   & 507.7 & 511.0 & 726.7 & 755.2 & 30.1  & 32.3  & 347.5 & 353.3 & 417.7 & 423.5 & 16.8  & 16.6  & 291.3 & 292.6 & 340.9 & 343.6 & 14.5  & 14.8 \\
97    & 100   & 495.8 & 510.7 & 651.1 & 694.6 & 23.8  & 26.5  & 347.6 & 356.8 & 377.5 & 387.4 & 7.9   & 7.9   & 276.7 & 285.0 & 287.0 & 291.5 & 3.6   & 2.2 \\
98    & 100   & 502.8 & 509.6 & 686.7 & 699.9 & 26.8  & 27.2  & 345.2 & 353.0 & 415.8 & 419.4 & 17.0  & 15.8  & 278.9 & 285.4 & 354.4 & 354.6 & 21.3  & 19.5 \\
99    & 100   & 494.8 & 498.0 & 703.5 & 731.2 & 29.7  & 31.9  & 341.7 & 350.5 & 412.8 & 424.1 & 17.2  & 17.4  & 275.5 & 282.7 & 336.3 & 344.3 & 18.1  & 17.9 \\
100   & 100   & 490.6 & 502.7 & 668.6 & 689.2 & 26.6  & 27.1  & 337.9 & 356.5 & 391.8 & 397.4 & 13.8  & 10.3  & 276.0 & 279.7 & 321.7 & 321.9 & 14.2  & 13.1 \\\addlinespace
101   & 175   & 630.3 & 634.7 & 888.1 & 908.3 & 29.0  & 30.1  & 429.5 & 432.3 & 499.7 & 511.6 & 14.0  & 15.5  & 339.0 & 347.2 & 405.3 & 408.3 & 16.4  & 15.0 \\
102   & 175   & 623.7 & 626.3 & 842.7 & 912.5 & 26.0  & 31.4  & 422.5 & 425.7 & 520.3 & 530.7 & 18.8  & 19.8  & 327.6 & 331.0 & 384.3 & 390.1 & 14.8  & 15.1 \\
103   & 175   & 644.8 & 652.4 & 868.4 & 893.3 & 25.7  & 27.0  & 434.6 & 445.2 & 488.2 & 499.3 & 11.0  & 10.9  & 345.2 & 359.6 & 359.9 & 362.4 & 4.1   & 0.8 \\
104   & 175   & 630.2 & 638.7 & 868.4 & 895.0 & 27.4  & 28.6  & 436.1 & 436.2 & 493.1 & 500.4 & 11.6  & 12.8  & 341.5 & 343.9 & 375.0 & 376.5 & 8.9   & 8.6 \\
105   & 175   & 639.1 & 642.3 & 882.5 & 922.7 & 27.6  & 30.4  & 439.2 & 441.4 & 517.7 & 523.8 & 15.2  & 15.7  & 335.5 & 336.6 & 363.3 & 365.1 & 7.7   & 7.8 \\
106   & 175   & 651.2 & 669.4 & 870.1 & 891.7 & 25.2  & 24.9  & 439.6 & 446.8 & 473.9 & 480.7 & 7.2   & 7.1   & 349.1 & 357.3 & 353.9 & 354.3 & 1.4   & -0.8 \\
107   & 175   & 637.8 & 641.2 & 892.5 & 914.5 & 28.5  & 29.9  & 430.7 & 447.5 & 514.1 & 519.8 & 16.2  & 13.9  & 345.4 & 350.9 & 386.6 & 391.9 & 10.7  & 10.4 \\
108   & 175   & 658.0 & 659.3 & 923.8 & 949.7 & 28.8  & 30.6  & 450.1 & 458.5 & 522.5 & 534.7 & 13.8  & 14.2  & 341.6 & 358.9 & 370.8 & 373.8 & 7.9   & 4.0 \\
109   & 175   & 636.6 & 643.4 & 902.9 & 920.4 & 29.5  & 30.1  & 428.5 & 435.2 & 494.6 & 505.7 & 13.4  & 13.9  & 331.8 & 335.8 & 348.5 & 351.7 & 4.8   & 4.5 \\
110   & 175   & 613.8 & 616.7 & 908.1 & 932.0 & 32.4  & 33.8  & 414.4 & 418.4 & 488.4 & 508.6 & 15.2  & 17.7  & 323.5 & 327.7 & 385.3 & 387.8 & 16.0  & 15.5 \\\addlinespace
111   & 250   & 710.7 & 714.2 & 1052.4 & 1091.1 & 32.5  & 34.5  & 493.6 & 494.9 & 593.9 & 607.0 & 16.9  & 18.5  & 368.7 & 369.5 & 438.4 & 442.8 & 15.9  & 16.6 \\
112   & 250   & 732.0 & 733.9 & 1047.3 & 1080.6 & 30.1  & 32.1  & 492.1 & 494.4 & 584.1 & 600.3 & 15.7  & 17.6  & 376.8 & 382.1 & 418.8 & 421.9 & 10.0  & 9.4 \\
113   & 250   & 743.0 & 743.4 & 1096.0 & 1115.2 & 32.2  & 33.3  & 494.1 & 494.9 & 588.1 & 603.3 & 16.0  & 18.0  & 380.0 & 383.3 & 439.1 & 445.3 & 13.5  & 13.9 \\
114   & 250   & 734.0 & 735.4 & 1056.2 & 1085.2 & 30.5  & 32.2  & 499.5 & 502.5 & 573.7 & 596.8 & 12.9  & 15.8  & 384.7 & 387.2 & 430.2 & 432.8 & 10.6  & 10.5 \\
115   & 250   & 748.0 & 748.6 & 1080.4 & 1108.7 & 30.8  & 32.5  & 501.1 & 501.2 & 593.8 & 603.5 & 15.6  & 16.9  & 380.8 & 387.1 & 422.8 & 428.0 & 9.9   & 9.6 \\
116   & 250   & 715.7 & 725.8 & 1037.5 & 1078.1 & 31.0  & 32.7  & 491.4 & 494.5 & 580.4 & 589.8 & 15.3  & 16.2  & 396.3 & 397.1 & 409.7 & 414.4 & 3.3   & 4.2 \\
117   & 250   & 750.5 & 751.3 & 1120.8 & 1150.7 & 33.0  & 34.7  & 508.2 & 518.3 & 631.2 & 640.8 & 19.5  & 19.1  & 396.2 & 397.7 & 468.8 & 472.8 & 15.5  & 15.9 \\
118   & 250   & 692.8 & 696.2 & 1019.2 & 1044.8 & 32.0  & 33.4  & 483.7 & 484.4 & 570.5 & 584.1 & 15.2  & 17.1  & 377.2 & 381.4 & 425.0 & 425.9 & 11.2  & 10.4 \\
119   & 250   & 805.6 & 808.5 & 1086.9 & 1132.5 & 25.9  & 28.6  & 511.3 & 527.1 & 613.8 & 621.7 & 16.7  & 15.2  & 390.5 & 395.2 & 435.6 & 442.2 & 10.4  & 10.6 \\
120   & 250   & 732.8 & 732.9 & 1087.8 & 1111.1 & 32.6  & 34.0  & 493.2 & 493.5 & 605.5 & 611.6 & 18.5  & 19.3  & 382.0 & 382.8 & 459.6 & 462.3 & 16.9  & 17.2 \\
\cmidrule(lr{0.6cm}){7-8} \cmidrule(lr{0.6cm}){13-14} \cmidrule(lr){19-20}
& & & & & & 25.3 & 27.4 & & & & & 12.5 & 12.5 & & & & & 7.9 & 7.3 \\
\bottomrule
\end{tabular}
}
\caption{Results of the comparison of \ac{MD-SPP-H} and \ac{IGH} (based on the solution published by \citet{Gonzalez-R2020Truck-dronePlanning}), considering $n=50$ up to $250$ uniformly distributed customers.}
\label{tab:TabBenchmarkGonazelezUniformFull}
\end{table}%
}

\newcommand{\TabBenchmarkGonazelezSingleCenterFull}[1]{
\begin{table}[]
\small
\resizebox{\textwidth}{!}{%
\centering
\begin{tabular}{rr@{\hspace{20pt}}rrrrrr@{\hspace{20pt}}rrrrrr@{\hspace{20pt}}rrrrrr}
\toprule
& & \multicolumn{6}{c@{\hspace{20pt}}}{Speed ratio 1}  & \multicolumn{6}{c@{\hspace{20pt}}}{Speed ratio 2} & \multicolumn{6}{c}{Speed ratio 3} \\
\multicolumn{2}{c@{\hspace{20pt}}}{Instance} 
& \multicolumn{2}{c}{MD-SPP-H} & \multicolumn{2}{c}{IGH} & \multicolumn{2}{c@{\hspace{20pt}}}{$\Delta$ [\%]} & \multicolumn{2}{c}{MD-SPP-H} & \multicolumn{2}{c}{IGH} & \multicolumn{2}{c@{\hspace{20pt}}}{$\Delta$ [\%]} & \multicolumn{2}{c}{MD-SPP-H} & \multicolumn{2}{c}{IGH} & \multicolumn{2}{c}{$\Delta$ [\%]} \\
\cmidrule(lr{0.6cm}){1-2}
\cmidrule(lr){3-4} \cmidrule(lr){5-6} \cmidrule(lr{0.6cm}){7-8}
\cmidrule(lr){9-10} \cmidrule(lr){11-12} \cmidrule(lrr{0.6cm}){13-14} 
\cmidrule(lr){15-16} \cmidrule(lr){17-18} \cmidrule(lr){19-20}
\multicolumn{1}{c}{ID} & \multicolumn{1}{c@{\hspace{20pt}}}{$n$} & \multicolumn{1}{c}{$z_\text{best}$} & \multicolumn{1}{c}{$z_\text{avg}$} & \multicolumn{1}{c}{$z_\text{best}$} & \multicolumn{1}{c}{$z_\text{avg}$} & \multicolumn{1}{c}{$\Delta_\text{best}$} & \multicolumn{1}{c@{\hspace{20pt}}}{$\Delta_\text{avg}$} & \multicolumn{1}{c}{$z_\text{best}$} & \multicolumn{1}{c}{$z_\text{avg}$} & \multicolumn{1}{c}{$z_\text{best}$} & \multicolumn{1}{c}{$z_\text{avg.}$} & \multicolumn{1}{c}{$\Delta_\text{best}$} & \multicolumn{1}{c@{\hspace{20pt}}}{$\Delta_\text{avg}$} & \multicolumn{1}{c}{$z_\text{best}$} & \multicolumn{1}{c}{$z_\text{avg}$} & \multicolumn{1}{c}{$z_\text{best}$} & \multicolumn{1}{c}{$z_\text{avg}$} & \multicolumn{1}{c}{$\Delta_\text{best}$} & \multicolumn{1}{c}{$\Delta_\text{avg}$} \\\midrule
71    & 50    & 516.8 & 528.8 & 571.8 & 618.3 & 9.6   & 14.5  & 350.1 & 353.6 & 370.9 & 371.5 & 5.6   & 4.8   & 275.5 & 278.3 & 292.6 & 292.6 & 5.8   & 4.9 \\
72    & 50    & 608.3 & 619.2 & 744.8 & 793.5 & 18.3  & 22.0  & 413.4 & 416.9 & 501.0 & 505.6 & 17.5  & 17.5  & 375.0 & 378.2 & 445.5 & 449.8 & 15.8  & 15.9 \\
73    & 50    & 454.2 & 457.3 & 492.7 & 535.0 & 7.8   & 14.5  & 319.4 & 323.3 & 348.4 & 348.4 & 8.3   & 7.2   & 254.2 & 256.7 & 315.4 & 315.6 & 19.4  & 18.7 \\
74    & 50    & 626.5 & 631.1 & 758.3 & 792.1 & 17.4  & 20.3  & 428.1 & 432.0 & 490.5 & 490.5 & 12.7  & 11.9  & 357.8 & 360.1 & 446.0 & 448.1 & 19.8  & 19.6 \\
75    & 50    & 791.0 & 806.2 & 830.8 & 857.6 & 4.8   & 6.0   & 616.7 & 621.5 & 648.5 & 654.3 & 4.9   & 5.0   & 561.7 & 571.9 & 629.3 & 629.3 & 10.7  & 9.1 \\
76    & 50    & 663.5 & 672.7 & 749.9 & 812.6 & 11.5  & 17.2  & 442.0 & 443.9 & 527.0 & 528.3 & 16.1  & 16.0  & 376.5 & 378.1 & 439.1 & 447.2 & 14.2  & 15.5 \\
77    & 50    & 656.6 & 671.5 & 700.7 & 749.6 & 6.3   & 10.4  & 511.0 & 513.9 & 578.8 & 580.3 & 11.7  & 11.5  & 551.8 & 556.9 & 543.8 & 543.9 & -1.5  & -2.4 \\
78    & 50    & 733.1 & 737.2 & 791.1 & 807.6 & 7.3   & 8.7   & 504.2 & 506.2 & 529.0 & 536.6 & 4.7   & 5.7   & 498.4 & 502.0 & 466.6 & 478.4 & -6.8  & -4.9 \\
79    & 50    & 634.1 & 638.2 & 663.2 & 690.2 & 4.4   & 7.5   & 535.2 & 536.7 & 546.0 & 548.0 & 2.0   & 2.1   & 511.1 & 513.8 & 499.9 & 501.4 & -2.2  & -2.5 \\
80    & 50    & 731.6 & 746.0 & 839.8 & 889.4 & 12.9  & 16.1  & 558.1 & 563.4 & 657.9 & 667.8 & 15.2  & 15.6  & 476.5 & 478.3 & 619.0 & 621.0 & 23.0  & 23.0 \\\addlinespace
81    & 75    & 919.7 & 939.9 & 1103.6 & 1139.8 & 16.7  & 17.5  & 686.1 & 700.5 & 742.4 & 752.5 & 7.6   & 6.9   & 567.5 & 579.6 & 678.4 & 680.0 & 16.3  & 14.8 \\
82    & 75    & 686.1 & 692.2 & 870.3 & 905.5 & 21.2  & 23.6  & 470.2 & 483.4 & 556.8 & 569.9 & 15.6  & 15.2  & 380.1 & 386.7 & 444.2 & 453.8 & 14.4  & 14.8 \\
83    & 75    & 742.4 & 756.3 & 919.6 & 951.4 & 19.3  & 20.5  & 516.5 & 520.5 & 647.7 & 654.0 & 20.2  & 20.4  & 486.4 & 489.8 & 545.2 & 547.2 & 10.8  & 10.5 \\
84    & 75    & 762.5 & 780.8 & 924.8 & 967.6 & 17.5  & 19.3  & 540.3 & 551.2 & 602.3 & 604.1 & 10.3  & 8.8   & 477.7 & 487.8 & 493.3 & 497.5 & 3.2   & 1.9 \\
85    & 75    & 765.9 & 773.4 & 896.8 & 917.6 & 14.6  & 15.7  & 575.8 & 587.7 & 643.0 & 646.5 & 10.4  & 9.1   & 500.8 & 511.3 & 558.6 & 560.3 & 10.4  & 8.8 \\
86    & 75    & 939.2 & 957.5 & 1056.5 & 1089.0 & 11.1  & 12.1  & 732.2 & 745.7 & 753.7 & 759.1 & 2.8   & 1.8   & 666.0 & 672.4 & 652.6 & 652.8 & -2.0  & -3.0 \\
87    & 75    & 850.2 & 866.1 & 971.2 & 1012.2 & 12.5  & 14.4  & 651.5 & 667.4 & 632.2 & 637.8 & -3.0  & -4.6  & 593.9 & 600.0 & 538.2 & 538.2 & -10.4 & -11.5 \\
88    & 75    & 965.9 & 984.5 & 1037.7 & 1083.6 & 6.9   & 9.1   & 790.1 & 805.6 & 861.2 & 867.4 & 8.3   & 7.1   & 773.5 & 776.3 & 800.2 & 801.1 & 3.3   & 3.1 \\
89    & 75    & 746.5 & 749.2 & 915.5 & 940.4 & 18.5  & 20.3  & 519.6 & 528.9 & 623.0 & 628.6 & 16.6  & 15.9  & 469.2 & 474.4 & 567.5 & 568.6 & 17.3  & 16.6 \\
90    & 75    & 823.9 & 839.8 & 1034.6 & 1071.7 & 20.4  & 21.6  & 576.1 & 585.8 & 713.3 & 724.3 & 19.2  & 19.1  & 487.3 & 492.6 & 613.8 & 616.0 & 20.6  & 20.0 \\
91    & 100   & 989.9 & 1009.1 & 1175.6 & 1205.6 & 15.8  & 16.3  & 725.4 & 740.6 & 803.8 & 810.1 & 9.8   & 8.6   & 732.5 & 735.6 & 738.5 & 741.7 & 0.8   & 0.8 \\\addlinespace
92    & 100   & 1031.3 & 1038.0 & 1276.7 & 1318.0 & 19.2  & 21.2  & 715.1 & 739.0 & 862.0 & 870.1 & 17.0  & 15.1  & 590.3 & 595.1 & 764.2 & 764.2 & 22.8  & 22.1 \\
93    & 100   & 1056.5 & 1065.9 & 1258.6 & 1309.1 & 16.1  & 18.6  & 762.9 & 774.4 & 912.4 & 921.9 & 16.4  & 16.0  & 704.8 & 716.2 & 853.6 & 853.6 & 17.4  & 16.1 \\
94    & 100   & 1109.6 & 1126.2 & 1318.6 & 1385.8 & 15.9  & 18.7  & 776.7 & 793.5 & 880.4 & 883.7 & 11.8  & 10.2  & 795.1 & 807.5 & 756.4 & 758.5 & -5.1  & -6.5 \\
95    & 100   & 852.6 & 859.6 & 1025.2 & 1090.3 & 16.8  & 21.2  & 633.5 & 637.4 & 727.7 & 735.4 & 12.9  & 13.3  & 529.8 & 537.6 & 603.2 & 609.0 & 12.2  & 11.7 \\
96    & 100   & 982.7 & 998.1 & 1191.7 & 1254.1 & 17.5  & 20.4  & 687.8 & 704.3 & 822.8 & 828.6 & 16.4  & 15.0  & 607.0 & 615.7 & 719.0 & 720.6 & 15.6  & 14.6 \\
97    & 100   & 1094.4 & 1106.3 & 1240.1 & 1259.2 & 11.7  & 12.1  & 808.5 & 818.2 & 846.1 & 852.6 & 4.4   & 4.0   & 698.6 & 707.9 & 772.3 & 776.1 & 9.5   & 8.8 \\
98    & 100   & 1119.6 & 1140.2 & 1337.3 & 1357.5 & 16.3  & 16.0  & 947.0 & 957.9 & 1060.2 & 1064.8 & 10.7  & 10.0  & 876.3 & 886.6 & 1002.2 & 1016.6 & 12.6  & 12.8 \\
99    & 100   & 787.5 & 798.3 & 999.4 & 1038.9 & 21.2  & 23.2  & 551.3 & 565.2 & 632.2 & 643.3 & 12.8  & 12.1  & 481.4 & 485.8 & 554.0 & 558.9 & 13.1  & 13.1 \\
100   & 100   & 821.9 & 843.4 & 1120.8 & 1152.4 & 26.7  & 26.8  & 577.4 & 584.1 & 718.9 & 726.5 & 19.7  & 19.6  & 444.3 & 460.2 & 633.4 & 635.1 & 29.9  & 27.5 \\\addlinespace
101   & 175   & 1261.7 & 1267.7 & 1590.1 & 1614.3 & 20.7  & 21.5  & 903.4 & 916.4 & 987.7 & 997.9 & 8.5   & 8.2   & 793.1 & 822.9 & 867.4 & 868.2 & 8.6   & 5.2 \\
102   & 175   & 1232.0 & 1240.0 & 1509.7 & 1567.8 & 18.4  & 20.9  & 817.2 & 827.7 & 943.6 & 956.6 & 13.4  & 13.5  & 676.9 & 680.3 & 815.0 & 818.9 & 17.0  & 16.9 \\
103   & 175   & 1213.7 & 1224.1 & 1547.1 & 1609.4 & 21.5  & 23.9  & 860.2 & 868.3 & 955.1 & 961.3 & 9.9   & 9.7   & 791.6 & 798.8 & 794.7 & 806.9 & 0.4   & 1.0 \\
104   & 175   & 1222.7 & 1228.1 & 1580.5 & 1603.6 & 22.6  & 23.4  & 853.1 & 861.9 & 1036.4 & 1045.9 & 17.7  & 17.6  & 709.9 & 734.1 & 890.1 & 893.3 & 20.2  & 17.8 \\
105   & 175   & 1331.1 & 1338.0 & 1597.0 & 1618.8 & 16.6  & 17.3  & 968.1 & 1000.5 & 1032.4 & 1038.3 & 6.2   & 3.6   & 869.3 & 908.1 & 919.3 & 923.5 & 5.4   & 1.7 \\
106   & 175   & 1331.3 & 1336.3 & 1632.4 & 1727.0 & 18.4  & 22.6  & 1035.7 & 1042.1 & 1146.2 & 1160.7 & 9.6   & 10.2  & 894.1 & 908.0 & 1014.7 & 1031.9 & 11.9  & 12.0 \\
107   & 175   & 1339.9 & 1351.7 & 1638.3 & 1686.8 & 18.2  & 19.9  & 983.7 & 992.2 & 1115.4 & 1118.7 & 11.8  & 11.3  & 900.8 & 918.7 & 1003.4 & 1004.8 & 10.2  & 8.6 \\
108   & 175   & 1227.2 & 1252.2 & 1536.5 & 1583.8 & 20.1  & 20.9  & 814.2 & 818.7 & 945.8 & 956.2 & 13.9  & 14.4  & 697.3 & 709.0 & 766.9 & 774.0 & 9.1   & 8.4 \\
109   & 175   & 1228.8 & 1243.1 & 1587.2 & 1627.6 & 22.6  & 23.6  & 902.1 & 908.1 & 1026.3 & 1037.2 & 12.1  & 12.4  & 784.7 & 794.4 & 859.5 & 869.3 & 8.7   & 8.6 \\
110   & 175   & 1188.8 & 1213.7 & 1620.4 & 1668.6 & 26.6  & 27.3  & 880.1 & 901.2 & 967.2 & 983.2 & 9.0   & 8.3   & 714.9 & 737.5 & 787.4 & 792.2 & 9.2   & 6.9 \\\addlinespace
111   & 250   & 1353.5 & 1357.6 & 1830.0 & 1859.6 & 26.0  & 27.0  & 906.9 & 910.8 & 1050.7 & 1056.3 & 13.7  & 13.8  & 750.9 & 755.4 & 832.8 & 838.9 & 9.8   & 9.9 \\
112   & 250   & 1406.1 & 1409.3 & 1716.5 & 1789.3 & 18.1  & 21.2  & 926.4 & 957.6 & 1099.4 & 1103.6 & 15.7  & 13.2  & 710.2 & 722.2 & 846.0 & 854.6 & 16.1  & 15.5 \\
113   & 250   & 1417.6 & 1418.1 & 1731.1 & 1761.2 & 18.1  & 19.5  & 929.4 & 943.2 & 1049.0 & 1059.6 & 11.4  & 11.0  & 759.1 & 768.2 & 810.1 & 814.7 & 6.3   & 5.7 \\
114   & 250   & 1483.5 & 1485.3 & 1824.1 & 1879.0 & 18.7  & 20.9  & 1035.3 & 1038.6 & 1193.6 & 1201.6 & 13.3  & 13.6  & 871.2 & 885.2 & 978.6 & 982.8 & 11.0  & 9.9 \\
115   & 250   & 1513.0 & 1522.3 & 1823.5 & 1887.8 & 17.0  & 19.4  & 1086.5 & 1091.4 & 1244.3 & 1261.4 & 12.7  & 13.5  & 914.3 & 916.0 & 1077.8 & 1081.0 & 15.2  & 15.3 \\
116   & 250   & 1421.6 & 1434.4 & 1820.7 & 1843.2 & 21.9  & 22.2  & 961.5 & 976.1 & 1101.8 & 1110.5 & 12.7  & 12.1  & 737.4 & 748.4 & 857.8 & 861.0 & 14.0  & 13.1 \\
117   & 250   & 1576.2 & 1576.9 & 1900.1 & 1969.1 & 17.0  & 19.9  & 1220.0 & 1220.0 & 1438.0 & 1449.1 & 15.2  & 15.8  & 1104.2 & 1113.4 & 1244.8 & 1250.0 & 11.3  & 10.9 \\
118   & 250   & 1458.3 & 1467.9 & 1844.1 & 1919.7 & 20.9  & 23.5  & 1027.2 & 1030.9 & 1267.5 & 1276.0 & 19.0  & 19.2  & 873.1 & 881.1 & 1019.8 & 1042.8 & 14.4  & 15.5 \\
119   & 250   & 1423.7 & 1428.1 & 1632.8 & 1720.4 & 12.8  & 17.0  & 1089.6 & 1094.0 & 1198.9 & 1216.4 & 9.1   & 10.1  & 939.3 & 947.8 & 1057.0 & 1061.4 & 11.1  & 10.7 \\
120   & 250   & 1302.6 & 1313.9 & 1684.9 & 1740.6 & 22.7  & 24.5  & 891.2 & 891.6 & 1100.4 & 1110.5 & 19.0  & 19.7  & 704.8 & 707.4 & 869.8 & 878.5 & 19.0  & 19.5 \\
\cmidrule(lr{0.6cm}){7-8} \cmidrule(lr{0.6cm}){13-14} \cmidrule(lr){19-20}
& & & & & & 16.7 & 18.9 & & & & & 11.9 & 11.4 & & & & & 10.8 & 10.1 \\

\bottomrule
\end{tabular}
}
\caption{Results of the comparison of \ac{MD-SPP-H} and \ac{IGH} (based on the solution published by \citet{Gonzalez-R2020Truck-dronePlanning}), considering $n=50$ up to $250$ single-center distributed customers.}
\label{tab:TabBenchmarkGonazelezSingleCenterFull}
\end{table}%
}

\newcommand{\TabBenchmarkGonazelezDoubleCenterFull}[1]{
\begin{table}[]
\small
\resizebox{\textwidth}{!}{%
\centering
\begin{tabular}{rr@{\hspace{20pt}}rrrrrr@{\hspace{20pt}}rrrrrr@{\hspace{20pt}}rrrrrr}
\toprule
& & \multicolumn{6}{c@{\hspace{20pt}}}{Speed ratio 1}  & \multicolumn{6}{c@{\hspace{20pt}}}{Speed ratio 2} & \multicolumn{6}{c}{Speed ratio 3} \\
\multicolumn{2}{c@{\hspace{20pt}}}{Instance} 
& \multicolumn{2}{c}{MD-SPP-H} & \multicolumn{2}{c}{IGH} & \multicolumn{2}{c@{\hspace{20pt}}}{$\Delta$ [\%]} & \multicolumn{2}{c}{MD-SPP-H} & \multicolumn{2}{c}{IGH} & \multicolumn{2}{c@{\hspace{20pt}}}{$\Delta$ [\%]} & \multicolumn{2}{c}{MD-SPP-H} & \multicolumn{2}{c}{IGH} & \multicolumn{2}{c}{$\Delta$ [\%]} \\
\cmidrule(lr{0.6cm}){1-2}
\cmidrule(lr){3-4} \cmidrule(lr){5-6} \cmidrule(lr{0.6cm}){7-8}
\cmidrule(lr){9-10} \cmidrule(lr){11-12} \cmidrule(lrr{0.6cm}){13-14} 
\cmidrule(lr){15-16} \cmidrule(lr){17-18} \cmidrule(lr){19-20}
\multicolumn{1}{c}{ID} & \multicolumn{1}{c@{\hspace{20pt}}}{$n$} & \multicolumn{1}{c}{$z_\text{best}$} & \multicolumn{1}{c}{$z_\text{avg}$} & \multicolumn{1}{c}{$z_\text{best}$} & \multicolumn{1}{c}{$z_\text{avg}$} & \multicolumn{1}{c}{$\Delta_\text{best}$} & \multicolumn{1}{c@{\hspace{20pt}}}{$\Delta_\text{avg}$} & \multicolumn{1}{c}{$z_\text{best}$} & \multicolumn{1}{c}{$z_\text{avg}$} & \multicolumn{1}{c}{$z_\text{best}$} & \multicolumn{1}{c}{$z_\text{avg.}$} & \multicolumn{1}{c}{$\Delta_\text{best}$} & \multicolumn{1}{c@{\hspace{20pt}}}{$\Delta_\text{avg}$} & \multicolumn{1}{c}{$z_\text{best}$} & \multicolumn{1}{c}{$z_\text{avg}$} & \multicolumn{1}{c}{$z_\text{best}$} & \multicolumn{1}{c}{$z_\text{avg}$} & \multicolumn{1}{c}{$\Delta_\text{best}$} & \multicolumn{1}{c}{$\Delta_\text{avg}$} \\
\midrule
71    & 50    & 997.2 & 1002.4 & 1109.9 & 1169.1 & 10.2  & 14.3  & 719.2 & 720.4 & 755.2 & 761.4 & 4.8   & 5.4   & 603.1 & 610.3 & 690.1 & 690.3 & 12.6  & 11.6 \\
72    & 50    & 887.5 & 897.5 & 1043.7 & 1091.5 & 15.0  & 17.8  & 610.0 & 614.5 & 694.9 & 700.2 & 12.2  & 12.2  & 482.1 & 488.1 & 594.7 & 602.4 & 18.9  & 19.0 \\
73    & 50    & 807.1 & 811.1 & 953.3 & 990.7 & 15.3  & 18.1  & 574.8 & 581.7 & 625.0 & 640.4 & 8.0   & 9.2   & 513.0 & 515.0 & 541.6 & 563.6 & 5.3   & 8.6 \\
74    & 50    & 808.7 & 810.5 & 833.8 & 897.9 & 3.0   & 9.7   & 570.0 & 577.1 & 525.0 & 538.2 & -8.6  & -7.2  & 519.2 & 520.8 & 399.0 & 412.6 & -30.1 & -26.2 \\
75    & 50    & 884.3 & 886.2 & 937.7 & 1004.0 & 5.7   & 11.7  & 633.8 & 645.8 & 570.2 & 583.6 & -11.2 & -10.7 & 580.5 & 584.6 & 409.9 & 416.8 & -41.6 & -40.3 \\
76    & 50    & 946.5 & 952.4 & 1106.1 & 1230.6 & 14.4  & 22.6  & 662.8 & 668.3 & 733.9 & 742.3 & 9.7   & 10.0  & 533.7 & 543.9 & 632.4 & 634.9 & 15.6  & 14.3 \\
77    & 50    & 819.4 & 820.5 & 1000.3 & 1030.2 & 18.1  & 20.3  & 578.7 & 582.7 & 627.1 & 637.5 & 7.7   & 8.6   & 527.8 & 529.3 & 534.0 & 535.6 & 1.2   & 1.2 \\
78    & 50    & 861.3 & 875.1 & 945.6 & 1056.0 & 8.9   & 17.1  & 587.7 & 591.8 & 599.1 & 619.8 & 1.9   & 4.5   & 464.9 & 466.5 & 425.4 & 439.4 & -9.3  & -6.2 \\
79    & 50    & 855.6 & 872.3 & 978.2 & 1016.0 & 12.5  & 14.1  & 594.7 & 597.9 & 645.9 & 653.2 & 7.9   & 8.5   & 483.4 & 484.9 & 548.4 & 558.0 & 11.8  & 13.1 \\
80    & 50    & 937.5 & 947.9 & 1052.8 & 1140.1 & 10.9  & 16.9  & 675.7 & 681.9 & 711.1 & 719.4 & 5.0   & 5.2   & 590.4 & 598.0 & 598.3 & 613.6 & 1.3   & 2.5 \\\addlinespace
81    & 75    & 928.7 & 938.2 & 1063.7 & 1116.5 & 12.7  & 16.0  & 620.4 & 636.0 & 646.9 & 663.3 & 4.1   & 4.1   & 512.5 & 522.2 & 491.4 & 503.6 & -4.3  & -3.7 \\
82    & 75    & 1043.4 & 1046.5 & 1274.8 & 1314.9 & 18.2  & 20.4  & 697.8 & 708.9 & 760.8 & 769.5 & 8.3   & 7.9   & 549.4 & 563.6 & 610.9 & 639.9 & 10.1  & 11.9 \\
83    & 75    & 998.1 & 1005.1 & 1104.1 & 1159.1 & 9.6   & 13.3  & 675.2 & 687.6 & 676.1 & 687.7 & 0.1   & 0.0   & 573.8 & 577.1 & 538.7 & 543.6 & -6.5  & -6.1 \\
84    & 75    & 1140.3 & 1163.9 & 1497.4 & 1553.5 & 23.9  & 25.1  & 769.7 & 798.5 & 889.5 & 907.3 & 13.5  & 12.0  & 718.5 & 728.3 & 696.5 & 716.4 & -3.2  & -1.7 \\
85    & 75    & 1177.6 & 1195.0 & 1509.4 & 1569.9 & 22.0  & 23.9  & 816.1 & 824.3 & 978.5 & 999.1 & 16.6  & 17.5  & 660.4 & 677.7 & 795.4 & 800.0 & 17.0  & 15.3 \\
86    & 75    & 980.4 & 997.8 & 1104.2 & 1175.2 & 11.2  & 15.1  & 666.3 & 671.7 & 672.4 & 686.2 & 0.9   & 2.1   & 543.8 & 553.5 & 527.9 & 542.0 & -3.0  & -2.1 \\
87    & 75    & 984.5 & 994.9 & 1084.6 & 1215.1 & 9.2   & 18.1  & 679.2 & 689.5 & 711.1 & 722.3 & 4.5   & 4.5   & 571.5 & 584.5 & 552.2 & 557.1 & -3.5  & -4.9 \\
88    & 75    & 1281.3 & 1317.7 & 1417.4 & 1467.8 & 9.6   & 10.2  & 887.1 & 898.3 & 895.5 & 901.0 & 0.9   & 0.3   & 724.4 & 736.8 & 701.2 & 709.3 & -3.3  & -3.9 \\
89    & 75    & 996.0 & 1015.4 & 1352.3 & 1407.1 & 26.3  & 27.8  & 681.3 & 694.2 & 785.5 & 791.6 & 13.3  & 12.3  & 547.2 & 556.1 & 629.8 & 641.1 & 13.1  & 13.3 \\
90    & 75    & 971.6 & 988.3 & 1224.5 & 1253.8 & 20.6  & 21.2  & 646.2 & 654.8 & 801.4 & 817.5 & 19.4  & 19.9  & 548.2 & 552.7 & 602.7 & 619.4 & 9.0   & 10.8 \\\addlinespace
91    & 100   & 1085.7 & 1101.5 & 1466.4 & 1524.4 & 26.0  & 27.7  & 741.8 & 747.5 & 852.8 & 867.7 & 13.0  & 13.8  & 587.0 & 594.4 & 634.0 & 639.9 & 7.4   & 7.1 \\
92    & 100   & 1150.4 & 1168.4 & 1348.8 & 1382.5 & 14.7  & 15.5  & 795.8 & 802.8 & 802.5 & 811.8 & 0.8   & 1.1   & 682.3 & 691.4 & 599.4 & 620.5 & -13.8 & -11.4 \\
93    & 100   & 988.1 & 1008.1 & 1253.3 & 1337.3 & 21.2  & 24.6  & 697.4 & 704.6 & 749.8 & 765.5 & 7.0   & 8.0   & 560.2 & 569.4 & 531.8 & 538.7 & -5.3  & -5.7 \\
94    & 100   & 1131.6 & 1144.8 & 1436.4 & 1498.9 & 21.2  & 23.6  & 767.9 & 782.5 & 860.2 & 885.2 & 10.7  & 11.6  & 639.4 & 646.7 & 682.1 & 683.1 & 6.3   & 5.3 \\
95    & 100   & 1184.7 & 1208.2 & 1497.6 & 1609.3 & 20.9  & 24.9  & 819.2 & 831.9 & 941.0 & 954.2 & 12.9  & 12.8  & 635.9 & 649.6 & 758.1 & 768.6 & 16.1  & 15.5 \\
96    & 100   & 1156.0 & 1204.9 & 1465.3 & 1548.0 & 21.1  & 22.2  & 811.1 & 831.5 & 854.7 & 873.4 & 5.1   & 4.8   & 625.3 & 642.7 & 624.3 & 630.8 & -0.2  & -1.9 \\
97    & 100   & 1281.7 & 1290.4 & 1724.0 & 1791.3 & 25.7  & 28.0  & 857.8 & 866.8 & 1057.8 & 1082.7 & 18.9  & 19.9  & 713.7 & 724.5 & 838.9 & 848.6 & 14.9  & 14.6 \\
98    & 100   & 1095.5 & 1103.6 & 1458.3 & 1536.9 & 24.9  & 28.2  & 752.4 & 766.4 & 910.5 & 930.6 & 17.4  & 17.6  & 646.6 & 653.6 & 741.8 & 753.6 & 12.8  & 13.3 \\
99    & 100   & 1107.2 & 1118.4 & 1380.2 & 1483.9 & 19.8  & 24.6  & 748.2 & 764.5 & 838.6 & 860.4 & 10.8  & 11.1  & 621.1 & 630.2 & 626.3 & 634.7 & 0.8   & 0.7 \\
100   & 100   & 1240.8 & 1243.5 & 1422.5 & 1541.1 & 12.8  & 19.3  & 829.6 & 846.4 & 967.4 & 987.6 & 14.2  & 14.3  & 664.7 & 678.4 & 736.7 & 756.0 & 9.8   & 10.3 \\\addlinespace
101   & 175   & 1277.9 & 1285.0 & 1936.3 & 2021.4 & 34.0  & 36.4  & 857.2 & 857.8 & 1138.2 & 1155.0 & 24.7  & 25.7  & 696.2 & 702.4 & 892.7 & 901.0 & 22.0  & 22.0 \\
102   & 175   & 1486.0 & 1489.5 & 2090.1 & 2193.0 & 28.9  & 32.1  & 1003.5 & 1011.7 & 1192.8 & 1209.3 & 15.9  & 16.3  & 760.4 & 783.5 & 891.9 & 901.0 & 14.7  & 13.0 \\
103   & 175   & 1492.4 & 1497.9 & 2154.4 & 2201.3 & 30.7  & 32.0  & 1043.9 & 1062.1 & 1195.4 & 1213.6 & 12.7  & 12.5  & 807.0 & 820.5 & 879.3 & 889.0 & 8.2   & 7.7 \\
104   & 175   & 1480.6 & 1488.3 & 2064.8 & 2105.7 & 28.3  & 29.3  & 985.4 & 997.3 & 1182.0 & 1211.3 & 16.6  & 17.7  & 766.1 & 773.4 & 897.7 & 912.7 & 14.7  & 15.3 \\
105   & 175   & 1587.0 & 1592.4 & 2146.2 & 2197.9 & 26.1  & 27.5  & 1067.9 & 1081.1 & 1287.4 & 1302.3 & 17.1  & 17.0  & 835.5 & 847.8 & 968.5 & 984.1 & 13.7  & 13.9 \\
106   & 175   & 1505.0 & 1543.9 & 1989.3 & 2146.3 & 24.3  & 28.1  & 1055.3 & 1063.6 & 1174.7 & 1199.1 & 10.2  & 11.3  & 796.2 & 808.5 & 857.6 & 874.2 & 7.2   & 7.5 \\
107   & 175   & 1327.1 & 1337.3 & 1864.8 & 1945.5 & 28.8  & 31.3  & 931.2 & 941.6 & 1072.5 & 1098.2 & 13.2  & 14.3  & 719.6 & 729.0 & 801.4 & 815.3 & 10.2  & 10.6 \\
108   & 175   & 1550.6 & 1559.1 & 2174.0 & 2240.0 & 28.7  & 30.4  & 1043.8 & 1064.8 & 1253.6 & 1285.9 & 16.7  & 17.2  & 814.8 & 830.5 & 975.4 & 981.1 & 16.5  & 15.4 \\
109   & 175   & 1456.4 & 1461.3 & 1987.1 & 2064.3 & 26.7  & 29.2  & 979.6 & 992.0 & 1160.4 & 1193.7 & 15.6  & 16.9  & 767.2 & 781.1 & 926.9 & 932.7 & 17.2  & 16.3 \\
110   & 175   & 1576.8 & 1581.9 & 2116.1 & 2194.4 & 25.5  & 27.9  & 1069.7 & 1078.1 & 1237.3 & 1250.0 & 13.5  & 13.8  & 845.8 & 853.2 & 828.4 & 842.7 & -2.1  & -1.3 \\\addlinespace
111   & 250   & 1662.1 & 1679.6 & 2534.2 & 2595.7 & 34.4  & 35.3  & 1180.2 & 1181.6 & 1511.2 & 1527.0 & 21.9  & 22.6  & 859.1 & 882.3 & 1178.8 & 1189.1 & 27.1  & 25.8 \\
112   & 250   & 1668.3 & 1673.5 & 2467.7 & 2574.0 & 32.4  & 35.0  & 1141.5 & 1145.9 & 1406.7 & 1424.8 & 18.9  & 19.6  & 891.2 & 892.5 & 1021.2 & 1029.7 & 12.7  & 13.3 \\
113   & 250   & 1690.7 & 1693.6 & 2111.8 & 2261.6 & 19.9  & 25.1  & 1127.4 & 1133.4 & 1272.4 & 1305.5 & 11.4  & 13.2  & 883.2 & 893.5 & 925.5 & 933.9 & 4.6   & 4.3 \\
114   & 250   & 1638.0 & 1668.3 & 2327.8 & 2415.9 & 29.6  & 30.9  & 1130.5 & 1131.7 & 1325.0 & 1343.9 & 14.7  & 15.8  & 855.3 & 876.7 & 1034.7 & 1040.7 & 17.3  & 15.8 \\
115   & 250   & 1882.8 & 1890.6 & 2552.4 & 2694.7 & 26.2  & 29.8  & 1297.9 & 1303.0 & 1460.3 & 1474.8 & 11.1  & 11.6  & 966.6 & 990.7 & 1092.1 & 1100.3 & 11.5  & 10.0 \\
116   & 250   & 1730.3 & 1738.2 & 2466.7 & 2586.2 & 29.9  & 32.8  & 1126.2 & 1134.8 & 1358.4 & 1380.5 & 17.1  & 17.8  & 879.3 & 881.0 & 944.1 & 954.2 & 6.9   & 7.7 \\
117   & 250   & 1696.5 & 1698.1 & 2539.6 & 2600.2 & 33.2  & 34.7  & 1136.8 & 1138.0 & 1447.9 & 1468.5 & 21.5  & 22.5  & 864.1 & 865.6 & 1072.8 & 1080.6 & 19.5  & 19.9 \\
118   & 250   & 1775.1 & 1785.3 & 2585.9 & 2635.2 & 31.4  & 32.3  & 1220.9 & 1223.5 & 1438.9 & 1464.9 & 15.1  & 16.5  & 931.3 & 942.0 & 1033.4 & 1045.1 & 9.9   & 9.9 \\
119   & 250   & 1845.8 & 1849.4 & 2600.6 & 2664.3 & 29.0  & 30.6  & 1261.8 & 1266.8 & 1519.8 & 1542.1 & 17.0  & 17.9  & 949.6 & 958.8 & 1171.7 & 1179.7 & 19.0  & 18.7 \\
120   & 250   & 1734.3 & 1734.9 & 2450.3 & 2508.0 & 29.2  & 30.8  & 1183.5 & 1194.9 & 1399.4 & 1436.1 & 15.4  & 16.8  & 913.8 & 916.2 & 976.2 & 984.0 & 6.4   & 6.9 \\
\cmidrule(lr{0.6cm}){7-8} \cmidrule(lr{0.6cm}){13-14} \cmidrule(lr){19-20}
& & & & & & 21.3& 24.3 & & & & & 11.0 & 11.6 & & & & & 6.3& 6.5 \\
\bottomrule
\end{tabular}
}
\caption{Results of the comparison of \ac{MD-SPP-H} and \ac{IGH} (based on the solution published by \citet{Gonzalez-R2020Truck-dronePlanning}), considering $n=50$ up to $250$ double-center distributed customers.}
\label{tab:TabBenchmarkGonazelezDoubleCenterFull}
\end{table}%
}

\newcommand{\TabBenchmarkGonazelezUniform}[1]{
\begin{table}[]
\resizebox{\textwidth}{!}{
\begin{tabular}{cr|rrr|rrr|rr}
\toprule
\multicolumn{2}{c|}{} & \multicolumn{3}{c|}{MD-SPP-H} & \multicolumn{3}{c|}{IGH} & \multicolumn{2}{c}{$\Delta$} \\
\multicolumn{1}{c}{Speed ratio} & \multicolumn{1}{c|}{Customers} & \multicolumn{1}{c}{$z^\text{.}_\text{best}$} & \multicolumn{1}{c}{$z^\text{.}_\text{avg.}$} & \multicolumn{1}{c|}{Run time} & \multicolumn{1}{c}{$z^\text{.}_\text{best}$} & \multicolumn{1}{c}{$z^\text{.}_\text{avg.}$} & \multicolumn{1}{c|}{Run time} & \multicolumn{1}{c}{best} & \multicolumn{1}{c}{avg.} \\
\midrule
\multirow{5}{*}{1} 
& 50        & 400.1     & 404.5     & 61.3      & 483.5     & 509.4     & 47.6      & 17.0          & 20.5          \\
& 75        & 443.7     & 452.1     & 378.0     & 584.2     & 610.3     & 163.2     & 24.0          & 25.8          \\
& 100       & 496.3     & 506.8     & 599.1     & 674.9     & 704.9     & 387.2     & 26.4          & 28.0          \\
& 175       & 636.6     & 642.5     & 602.9     & 884.8     & 914.0     & 621.2     & 28.0          & 29.7          \\
& 250       & 736.5     & 739.0     & 617.4     & 1068.4    & 1099.8    & 756.3     & 31.1          & 32.8          \\
\cmidrule(lr){9-10}
\multicolumn{8}{c}{}    & \textbf{25.3} & \textbf{27.4} \\
\midrule
\multirow{5}{*}{2}                     
& 50        & 283.6     & 287.7     & 27.1      & 304.2     & 308.9     & 48.2      & 6.7           & 6.8           \\
& 75        & 311.9     & 319.4     & 149.7     & 352.7     & 359.3     & 163.2     & 11.3          & 10.9          \\
& 100       & 339.9     & 351.1     & 474.3     & 398.7     & 406.4     & 387.0     & 14.6          & 13.5          \\
& 175       & 432.5     & 438.7     & 601.0     & 501.2     & 511.5     & 624.3     & 13.6          & 14.2          \\
& 250       & 496.8     & 500.6     & 604.7     & 593.5     & 605.9     & 705.1     & 16.2          & 17.4          \\
\cmidrule(lr){9-10} 
\multicolumn{8}{c}{}    & \textbf{12.5} & \textbf{12.5} \\
\midrule
\multirow{5}{*}{3} 
& 50        & 244.4     & 247.7     & 21.8      & 248.8     & 250.9     & 49.1      & 1.5           & 1.0 \\
& 75        & 260.9     & 265.6     & 111.7     & 279.2     & 281.9     & 166.1     & 5.9           & 5.1           \\
& 100       & 278.8     & 284.1     & 383.4     & 315.8     & 319.0     & 394.2     & 11.4          & 10.6          \\
& 175       & 338.0     & 344.9     & 600.6     & 373.3     & 376.2     & 624.1     & 9.3           & 8.1           \\
& 250       & 383.3     & 386.3     & 602.4     & 434.8     & 438.8     & 698.9     & 11.7          & 11.8         \\
\cmidrule(lr){9-10}
\multicolumn{8}{c}{}    & \textbf{7.9}  & \textbf{7.3}  \\
\bottomrule
\end{tabular}
}
\caption{Results of the comparison of \ac{MD-SPP-H} and \ac{IGH} (based on the solution published by \citet{Gonzalez-R2020Truck-dronePlanning}), considering 50, 75, 100, 175, and 250 uniformly distributed customers. The values reported are the averages over ten instances. The bold values in the $\Delta$ columns represent the averages for each speed ratio.}
\label{tab:TabBenchmarkGonazelezUniform}
\end{table}
}

\newcommand{\TabBenchmarkGonazelezSingleCenter}[1]{
\begin{table}[]
\resizebox{\textwidth}{!}{
\begin{tabular}{cr|rrr|rrr|rr}
\toprule
\multicolumn{2}{c|}{} & \multicolumn{3}{c|}{MD-SPP-H} & \multicolumn{3}{c|}{IGH} & \multicolumn{2}{c}{$\Delta$} \\
\multicolumn{1}{c}{Speed ratio} & \multicolumn{1}{c|}{Customers} & \multicolumn{1}{c}{$z^\text{.}_\text{best}$} & \multicolumn{1}{c}{$z^\text{.}_\text{avg.}$} & \multicolumn{1}{c|}{Run time} & \multicolumn{1}{c}{$z^\text{.}_\text{best}$} & \multicolumn{1}{c}{$z^\text{.}_\text{avg.}$} & \multicolumn{1}{c|}{Run time} & \multicolumn{1}{c}{best} & \multicolumn{1}{c}{avg.} \\
\midrule
\multirow{5}{*}{1} 
& 50  & 641.6     & 650.8     & 49.5      & 714.3     & 754.6     & 46.4      & 10.0          & 13.7          \\
& 75  & 820.2     & 834.0     & 322.8     & 973.1     & 1007.9    & 157.5     & 15.9          & 17.4          \\
& 100 & 984.6     & 998.5     & 595.1     & 1194.4    & 1237.1    & 380.3     & 17.7          & 19.5          \\
& 175 & 1257.7    & 1269.5    & 602.5     & 1583.9    & 1630.8    & 622.4     & 20.6          & 22.1          \\
& 250 & 1435.6    & 1441.4    & 614.7     & 1780.8    & 1837.0    & 873.2     & 19.3          & 21.5          \\
\cmidrule(lr){9-10} % Changed this to ensure it does not affect the first two columns
\multicolumn{8}{r}{}    & \textbf{16.7} & \textbf{18.9} \\ % Corrected to align the average with no vertical line
\midrule
\multirow{5}{*}{2}                     
& 50  & 467.8     & 471.1     & 32.9      & 519.8     & 523.1     & 50.4      & 9.9           & 9.7           \\
& 75  & 605.9     & 617.7     & 252.7     & 677.5     & 684.4     & 170.4     & 10.8          & 10.0          \\
& 100 & 718.6     & 731.5     & 754.6     & 826.6     & 833.7     & 406.6     & 13.2          & 12.4          \\
& 175 & 901.8     & 913.7     & 1801.0    & 1015.6    & 1025.6    & 620.8     & 11.2          & 10.9          \\
& 250 & 1007.4    & 1015.4    & 1805.4    & 1174.4    & 1184.5    & 750.2     & 14.2          & 14.2          \\
\cmidrule(lr){9-10}
\multicolumn{8}{r}{}    & \textbf{11.9} & \textbf{11.4} \\
\midrule
\multirow{5}{*}{3} 
& 50  & 423.9     & 427.4     & 27.0      & 469.7     & 472.7     & 50.7      & 9.8           & 9.7           \\
& 75  & 540.2     & 547.1     & 163.7     & 589.2     & 591.6     & 171.1     & 8.4           & 7.6           \\
& 100 & 646.0     & 654.8     & 464.6     & 739.7     & 743.4     & 412.1     & 12.9          & 12.1          \\
& 175 & 783.3     & 801.2     & 600.7     & 871.8     & 878.3     & 623.9     & 10.1          & 8.7           \\
& 250 & 836.4     & 844.5     & 603.0     & 959.5     & 966.6     & 715.3     & 12.8          & 12.6         \\
\cmidrule(lr){9-10}
\multicolumn{8}{r}{}    & \textbf{10.8} & \textbf{10.1} \\
\bottomrule
\end{tabular}
}
\caption{Results of the comparison of \ac{MD-SPP-H} and \ac{IGH} (based on the solution published by \citet{Gonzalez-R2020Truck-dronePlanning}), considering 50, 75, 100, 175, and 250 single-center distributed customers. The values reported are the averages over ten instances. The bold values in the $\Delta$ columns represent the averages for each speed ratio.}
\label{tab:TabBenchmarkGonazelezSingleCenter}
\end{table}
}

\newcommand{\TabBenchmarkGonazelezDoubleCenter}[1]{
\begin{table}[]
\resizebox{\textwidth}{!}{
\begin{tabular}{cr|rrr|rrr|rr}
\toprule
\multicolumn{2}{c|}{} & \multicolumn{3}{c|}{MD-SPP-H} & \multicolumn{3}{c|}{IGH} & \multicolumn{2}{c}{$\Delta$} \\
\multicolumn{1}{c}{Speed ratio} & \multicolumn{1}{c|}{Customers} & \multicolumn{1}{c}{$z^\text{.}_\text{best}$} & \multicolumn{1}{c}{$z^\text{.}_\text{avg.}$} & \multicolumn{1}{c|}{Run time} & \multicolumn{1}{c}{$z^\text{.}_\text{best}$} & \multicolumn{1}{c}{$z^\text{.}_\text{avg.}$} & \multicolumn{1}{c|}{Run time} & \multicolumn{1}{c}{best} & \multicolumn{1}{c}{avg.} \\
\midrule
\multirow{5}{*}{1} 
&50        & 880.5     & 887.6     & 76.4      & 996.1     & 1062.6    & 46.3      & 11.4          & 16.3          \\
&75        & 1050.2    & 1066.3    & 501.9     & 1263.2    & 1323.3    & 163.0     & 16.3          & 19.1          \\
&100       & 1142.2    & 1159.2    & 600.3     & 1445.3    & 1525.4    & 385.9     & 20.8          & 23.9          \\
&175       & 1474.0    & 1483.7    & 603.2     & 2052.3    & 2131.0    & 619.6     & 28.2          & 30.4          \\
&250       & 1732.4    & 1741.2    & 623.9     & 2463.7    & 2553.6    & 803.5     & 29.5          & 31.7          \\
\cmidrule(lr){9-10} 
\multicolumn{8}{r}{}    & \textbf{21.3} & \textbf{24.3} \\
\midrule
\multirow{5}{*}{2}                     
& 50        & 620.7     & 626.2     & 35.9      & 648.8     & 659.6     & 49.6      & 3.8           & 4.6           \\
& 75        & 713.9     & 726.4     & 262.8     & 781.8     & 794.6     & 172.7     & 8.2           & 8.1           \\
& 100       & 782.1     & 794.5     & 530.0     & 883.5     & 901.9     & 412.4     & 11.1          & 11.5          \\
& 175       & 1003.8    & 1015.0    & 601.4     & 1189.4    & 1211.8    & 622.4     & 15.6          & 16.3          \\
& 250       & 1180.7    & 1185.4    & 608.1     & 1414.0    & 1436.8    & 772.5     & 16.4          & 17.4          \\
\cmidrule(lr){9-10}
\multicolumn{8}{r}{}    & \textbf{11.0} & \textbf{11.6} \\
\midrule
\multirow{5}{*}{3} 
& 50        & 529.8     & 534.1     & 33.8      & 537.4     & 546.7     & 48.8      & -1.4          & -0.2          \\
& 75        & 595.0     & 605.2     & 222.6     & 614.7     & 627.2     & 168.3     & 2.5           & 2.9           \\
& 100       & 637.6     & 648.1     & 649.5     & 677.3     & 687.4     & 401.3     & 4.9           & 4.8           \\
& 175       & 780.9     & 793.0     & 1800.8    & 892.0     & 903.4     & 622.8     & 12.2          & 12.0          \\
& 250       & 899.3     & 909.9     & 1804.4    & 1045.0    & 1053.7    & 738.2     & 13.5          & 13.2         \\
\cmidrule(lr){9-10}
\multicolumn{8}{r}{}    & \textbf{6.3}  & \textbf{6.5}  \\
\bottomrule
\end{tabular}
}
\caption{Results of the comparison of \ac{MD-SPP-H} and \ac{IGH} (based on the solution published by \citet{Gonzalez-R2020Truck-dronePlanning}), considering 50, 75, 100, 175, and 250 double-center distributed customers. The values reported are the averages over ten instances. The bold values in the $\Delta$ columns represent the averages for each speed ratio.}
\label{tab:TabBenchmarkGonazelezDoubleCenter}
\end{table}
}

\newcommand{\TabBenchmarkMara}[1]{
\begin{table}[]
\small
\resizebox{\textwidth}{!}{%
\centering
\begin{tabular}{lrrrrrrrr@{\hspace{20pt}}lrrrrrrrr}
\toprule
\multicolumn{9}{c@{\hspace{20pt}}}{50 Customers} & \multicolumn{9}{c}{100 Customers} \\

& \multicolumn{3}{c}{MD-SPP-H} & \multicolumn{3}{c}{ALNS} & \multicolumn{2}{c@{\hspace{20pt}}}{$\Delta$ [\%]} &       & \multicolumn{3}{c}{MD-SPP-H} & \multicolumn{3}{c}{ALNS} & \multicolumn{2}{c}{$\Delta$ [\%]} \\
\cmidrule(lr){2-4} \cmidrule(lr){5-7} \cmidrule(lr{0.6cm}){8-9}
\cmidrule(lr){11-13} \cmidrule(lr){14-16} \cmidrule(lr){17-18}
\multicolumn{1}{c}{ID} & \multicolumn{1}{c}{$z_\text{best}$} & \multicolumn{1}{c}{$z_\text{avg}$} & \multicolumn{1}{c}{$z_\text{std}$} & \multicolumn{1}{c}{$z_\text{best}$} & \multicolumn{1}{c}{$z_\text{avg}$} & \multicolumn{1}{c}{$z_\text{std}$} & \multicolumn{1}{c}{$\Delta_\text{best}$} & \multicolumn{1}{c@{\hspace{20pt}}}{$\Delta_\text{avg}$} & \multicolumn{1}{c}{ID} & \multicolumn{1}{c}{$z_\text{best}$} & \multicolumn{1}{c}{$z_\text{avg}$} & \multicolumn{1}{c}{$z_\text{std}$} & \multicolumn{1}{c}{$z_\text{best}$} & \multicolumn{1}{c}{$z_\text{avg}$} & \multicolumn{1}{c}{$z_\text{std}$} & \multicolumn{1}{c}{$\Delta_\text{best}$} & \multicolumn{1}{c}{$\Delta_\text{avg}$} \\
\midrule
mbB101 & 76.1  & 78.0  & 1.3   & 80.7  & 86.9  & 5.1   & 5.7   & 10.3  & mbE101 & 109.2 & 114.7 & 4.1   & 124.8 & 131.9 & 5.6   & 12.4  & 13.1 \\
mbB102 & 81.0  & 84.6  & 1.8   & 75.3  & 79.3  & 2.9   & -7.7  & -6.7  & mbE102 & 119.0 & 121.7 & 1.6   & 122.4 & 130.4 & 5.2   & 2.8   & 6.7 \\
mbB103 & 79.4  & 82.0  & 2.1   & 77.6  & 86.3  & 5.6   & -2.4  & 5.0   & mbE103 & 122.2 & 124.7 & 1.3   & 125.6 & 134.9 & 7.7   & 2.7   & 7.5 \\
mbB104 & 77.2  & 78.4  & 1.2   & 76.0  & 85.7  & 4.7   & -1.6  & 8.5   & mbE104 & 119.5 & 122.7 & 2.2   & 119.5 & 131.6 & 6.8   & 0.0   & 6.8 \\
mbB105 & 80.2  & 82.0  & 2.1   & 80.1  & 88.8  & 6.0   & -0.1  & 7.6   & mbE105 & 116.8 & 118.5 & 1.0   & 122.2 & 128.8 & 5.9   & 4.4   & 8.0 \\
mbB106 & 74.0  & 76.8  & 1.2   & 74.9  & 81.8  & 5.4   & 1.3   & 6.2   & mbE106 & 118.9 & 121.2 & 1.4   & 120.3 & 128.1 & 4.2   & 1.2   & 5.4 \\
mbB107 & 79.4  & 82.4  & 1.4   & 78.7  & 86.7  & 6.9   & -0.9  & 5.0   & mbE107 & 113.4 & 115.4 & 1.1   & 123.8 & 132.3 & 5.1   & 8.4   & 12.7 \\
mbB108 & 75.8  & 77.6  & 1.3   & 81.9  & 85.6  & 3.0   & 7.4   & 9.3   & mbE108 & 116.7 & 117.7 & 1.2   & 115.0 & 126.5 & 8.0   & -1.5  & 7.0 \\
mbB109 & 81.9  & 83.1  & 1.0   & 79.7  & 85.2  & 3.5   & -2.8  & 2.4   & mbE109 & 117.8 & 120.4 & 1.3   & 115.8 & 121.6 & 3.8   & -1.7  & 1.1 \\
mbB110 & 76.3  & 78.1  & 1.7   & 75.1  & 89.1  & 5.5   & -1.6  & 12.3  & mbE110 & 119.0 & 121.2 & 1.0   & 120.2 & 128.8 & 5.7   & 1.0   & 5.9 \\\addlinespace
mbC101 & 172.4 & 173.4 & 1.3   & 179.7 & 197.6 & 19.8  & 4.0   & 12.2  & mbF101 & 253.9 & 258.8 & 2.7   & 262.7 & 307.1 & 20.7  & 3.4   & 15.7 \\
mbC102 & 163.6 & 164.2 & 0.9   & 171.2 & 192.4 & 19.7  & 4.5   & 14.6  & mbF102 & 243.3 & 248.6 & 5.2   & 266.6 & 313.1 & 24.3  & 8.7   & 20.6 \\
mbC103 & 182.7 & 190.9 & 5.0   & 183.8 & 203.5 & 15.5  & 0.6   & 6.2   & mbF103 & 250.8 & 256.4 & 3.9   & 277.7 & 307.7 & 22.3  & 9.7   & 16.7 \\
mbC104 & 175.4 & 177.3 & 1.3   & 188.1 & 199.9 & 11.3  & 6.8   & 11.3  & mbF104 & 244.9 & 249.2 & 2.2   & 281.3 & 308.1 & 20.5  & 12.9  & 19.1 \\
mbC105 & 196.9 & 198.0 & 0.8   & 202.6 & 216.4 & 10.6  & 2.8   & 8.5   & mbF105 & 262.5 & 265.3 & 1.1   & 274.7 & 304.4 & 14.5  & 4.5   & 12.9 \\
mbC106 & 195.3 & 197.2 & 2.4   & 204.5 & 228.7 & 17.4  & 4.5   & 13.8  & mbF106 & 236.7 & 239.5 & 2.1   & 239.5 & 291.6 & 25.5  & 1.2   & 17.9 \\
mbC107 & 177.1 & 180.7 & 3.0   & 187.7 & 213.2 & 13.6  & 5.6   & 15.2  & mbF107 & 234.2 & 237.7 & 2.2   & 279.0 & 319.8 & 32.5  & 16.1  & 25.7 \\
mbC108 & 198.2 & 201.8 & 1.7   & 199.6 & 218.0 & 9.5   & 0.7   & 7.4   & mbF108 & 249.4 & 254.2 & 3.8   & 269.2 & 290.9 & 29.3  & 7.3   & 12.6 \\
mbC109 & 190.1 & 193.2 & 2.2   & 212.3 & 225.3 & 9.1   & 10.4  & 14.2  & mbF109 & 253.6 & 254.4 & 0.5   & 284.5 & 319.3 & 33.1  & 10.9  & 20.3 \\
mbC110 & 183.3 & 187.8 & 4.0   & 195.2 & 222.3 & 14.6  & 6.1   & 15.5  & mbF110 & 250.3 & 256.3 & 3.2   & 258.9 & 304.0 & 24.6  & 3.3   & 15.7 \\\addlinespace
mbD101 & 269.8 & 274.1 & 3.7   & 293.2 & 315.0 & 24.1  & 8.0   & 13.0  & mbG101 & 343.2 & 349.5 & 4.0   & 395.1 & 448.5 & 53.2  & 13.1  & 22.1 \\
mbD102 & 270.5 & 273.1 & 2.2   & 305.9 & 323.5 & 15.5  & 11.6  & 15.6  & mbG102 & 332.7 & 337.1 & 3.6   & 389.4 & 422.8 & 16.9  & 14.6  & 20.3 \\
mbD103 & 251.4 & 252.8 & 1.6   & 271.3 & 307.9 & 22.7  & 7.3   & 17.9  & mbG103 & 369.9 & 378.9 & 5.5   & 431.8 & 488.2 & 29.9  & 14.3  & 22.4 \\
mbD104 & 286.2 & 291.0 & 2.5   & 307.1 & 339.3 & 21.1  & 6.8   & 14.2  & mbG104 & 362.9 & 367.8 & 2.9   & 399.3 & 453.6 & 42.7  & 9.1   & 18.9 \\
mbD105 & 281.7 & 283.6 & 2.2   & 287.7 & 316.8 & 23.9  & 2.1   & 10.5  & mbG105 & 360.1 & 364.5 & 2.2   & 418.1 & 456.9 & 29.8  & 13.9  & 20.2 \\
mbD106 & 271.8 & 277.7 & 6.9   & 289.9 & 316.4 & 20.0  & 6.3   & 12.2  & mbG106 & 348.7 & 354.9 & 3.6   & 391.6 & 462.7 & 45.3  & 11.0  & 23.3 \\
mbD107 & 283.0 & 283.6 & 1.1   & 297.0 & 322.7 & 25.2  & 4.7   & 12.1  & mbG107 & 341.1 & 345.2 & 2.7   & 384.8 & 442.6 & 29.4  & 11.3  & 22.0 \\
mbD108 & 249.6 & 257.5 & 5.5   & 278.6 & 311.9 & 25.3  & 10.4  & 17.4  & mbG108 & 351.6 & 357.8 & 3.3   & 409.0 & 459.0 & 37.2  & 14.0  & 22.1 \\
mbD109 & 287.5 & 290.2 & 2.6   & 300.1 & 327.6 & 16.8  & 4.2   & 11.4  & mbG109 & 351.9 & 356.1 & 2.7   & 405.9 & 468.9 & 34.1  & 13.3  & 24.0 \\
mbD110 & 261.4 & 262.2 & 0.6   & 265.3 & 281.9 & 13.9  & 1.5   & 7.0   & mbG110 & 350.5 & 358.2 & 4.4   & 419.9 & 472.6 & 36.3  & 16.5  & 24.2 \\
\cmidrule(lr{0.6cm}){8-9} \cmidrule(lr){17-18}
& & & & & & & 3.5 & 10.3 & & & & & & & & 8.0 & 15.7 \\
\bottomrule
\end{tabular}%

}
\caption{Results of the comparison of \ac{MD-SPP-H} and \ac{ALNS}, considering 50 and 100 customers and all service area sizes of \citet{Ha2018OnDrone}.} 
\label{tab:TabBenchmarkMara}
\end{table}%
}

\newcommand{\TabBenchmarkMaraBCD}[1]{
\begin{table}[]
  \centering
    \begin{tabular}{lrrrrrrrr}
    \toprule
          & \multicolumn{3}{c}{MD-SPP-H} & \multicolumn{3}{c}{ALNS} & \multicolumn{2}{c}{$\Delta$} \\
    \multicolumn{1}{c}{Instance ID} & \multicolumn{1}{c}{$z_\text{best}$} & \multicolumn{1}{c}{$z_\text{avg}$} & \multicolumn{1}{c}{$z_\text{std}$} & \multicolumn{1}{c}{$z_\text{best}$} & \multicolumn{1}{c}{$z_\text{avg}$} & \multicolumn{1}{c}{$z_\text{std}$} & \multicolumn{1}{c}{$\Delta_\text{best}$} & \multicolumn{1}{c}{$\Delta_\text{avg}$} \\ \midrule
    mbB101 & 76.1  & 78.0  & 1.3   & 80.7  & 86.9  & 5.1   & 5.7   & 10.3 \\
    mbB102 & 81.0  & 84.6  & 1.8   & 75.3  & 79.3  & 2.9   & -7.7  & -6.7 \\
    mbB103 & 79.4  & 82.0  & 2.1   & 77.6  & 86.3  & 5.6   & -2.4  & 5.0 \\
    mbB104 & 77.2  & 78.4  & 1.2   & 76.0  & 85.7  & 4.7   & -1.6  & 8.5 \\
    mbB105 & 80.2  & 82.0  & 2.1   & 80.1  & 88.8  & 6.0   & -0.1  & 7.6 \\
    mbB106 & 74.0  & 76.8  & 1.2   & 74.9  & 81.8  & 5.4   & 1.3   & 6.2 \\
    mbB107 & 79.4  & 82.4  & 1.4   & 78.7  & 86.7  & 6.9   & -0.9  & 5.0 \\
    mbB108 & 75.8  & 77.6  & 1.3   & 81.9  & 85.6  & 3.0   & 7.4   & 9.3 \\
    mbB109 & 81.9  & 83.1  & 1.0   & 79.7  & 85.2  & 3.5   & -2.8  & 2.4 \\
    mbB110 & 76.3  & 78.1  & 1.7   & 75.1  & 89.1  & 5.5   & -1.6  & 12.3 \\
    mbC101 & 172.4 & 173.4 & 1.3   & 179.7 & 197.6 & 19.8  & 4.0   & 12.2 \\
    mbC102 & 163.6 & 164.2 & 0.9   & 171.2 & 192.4 & 19.7  & 4.5   & 14.6 \\
    mbC103 & 182.7 & 190.9 & 5.0   & 183.8 & 203.5 & 15.5  & 0.6   & 6.2 \\
    mbC104 & 175.4 & 177.3 & 1.3   & 188.1 & 199.9 & 11.3  & 6.8   & 11.3 \\
    mbC105 & 196.9 & 198.0 & 0.8   & 202.6 & 216.4 & 10.6  & 2.8   & 8.5 \\
    mbC106 & 195.3 & 197.2 & 2.4   & 204.5 & 228.7 & 17.4  & 4.5   & 13.8 \\
    mbC107 & 177.1 & 180.7 & 3.0   & 187.7 & 213.2 & 13.6  & 5.6   & 15.2 \\
    mbC108 & 198.2 & 201.8 & 1.7   & 199.6 & 218.0 & 9.5   & 0.7   & 7.4 \\
    mbC109 & 190.1 & 193.2 & 2.2   & 212.3 & 225.3 & 9.1   & 10.4  & 14.2 \\
    mbC110 & 183.3 & 187.8 & 4.0   & 195.2 & 222.3 & 14.6  & 6.1   & 15.5 \\
    mbD101 & 269.8 & 274.1 & 3.7   & 293.2 & 315.0 & 24.1  & 8.0   & 13.0 \\
    mbD102 & 270.5 & 273.1 & 2.2   & 305.9 & 323.5 & 15.5  & 11.6  & 15.6 \\
    mbD103 & 251.4 & 252.8 & 1.6   & 271.3 & 307.9 & 22.7  & 7.3   & 17.9 \\
    mbD104 & 286.2 & 291.0 & 2.5   & 307.1 & 339.3 & 21.1  & 6.8   & 14.2 \\
    mbD105 & 281.7 & 283.6 & 2.2   & 287.7 & 316.8 & 23.9  & 2.1   & 10.5 \\
    mbD106 & 271.8 & 277.7 & 6.9   & 289.9 & 316.4 & 20.0  & 6.3   & 12.2 \\
    mbD107 & 283.0 & 283.6 & 1.1   & 297.0 & 322.7 & 25.2  & 4.7   & 12.1 \\
    mbD108 & 249.6 & 257.5 & 5.5   & 278.6 & 311.9 & 25.3  & 10.4  & 17.4 \\
    mbD109 & 287.5 & 290.2 & 2.6   & 300.1 & 327.6 & 16.8  & 4.2   & 11.4 \\
    mbD110 & 261.4 & 262.2 & 0.6   & 265.3 & 281.9 & 13.9  & 1.5   & 7.0 \\ \bottomrule
    \end{tabular}%
  \caption{Results of the comparison of \ac{MD-SPP-H} and \ac{ALNS}, considering 50 customers and all service area sizes of \citet{Ha2018OnDrone}. The values reported are the averages over ten instances. The bold values in the $\Delta$ columns represent the averages for each customer size.} 
\label{tab:TabBenchmarkMaraBCD}
\end{table}%
}

\newcommand{\TabBenchmarkMaraEFG}[1]{
\begin{table}[]
  \centering
    \begin{tabular}{lrrrrrrrr}
    \toprule
          & \multicolumn{3}{c}{MD-SPP-H} & \multicolumn{3}{c}{ALNS} & \multicolumn{2}{c}{$\Delta$} \\
    \multicolumn{1}{c}{Instance ID} & \multicolumn{1}{c}{$z_\text{best}$} & \multicolumn{1}{c}{$z_\text{avg}$} & \multicolumn{1}{c}{$z_\text{std}$} & \multicolumn{1}{c}{$z_\text{best}$} & \multicolumn{1}{c}{$z_\text{avg}$} & \multicolumn{1}{c}{$z_\text{std}$} & \multicolumn{1}{c}{$\Delta_\text{best}$} & \multicolumn{1}{c}{$\Delta_\text{avg}$} \\
    \midrule
    mbE101 & 109.2 & 114.7 & 4.1   & 124.8 & 131.9 & 5.6   & 12.4  & 13.1 \\
    mbE102 & 119.0 & 121.7 & 1.6   & 122.4 & 130.4 & 5.2   & 2.8   & 6.7 \\
    mbE103 & 122.2 & 124.7 & 1.3   & 125.6 & 134.9 & 7.7   & 2.7   & 7.5 \\
    mbE104 & 119.5 & 122.7 & 2.2   & 119.5 & 131.6 & 6.8   & 0.0   & 6.8 \\
    mbE105 & 116.8 & 118.5 & 1.0   & 122.2 & 128.8 & 5.9   & 4.4   & 8.0 \\
    mbE106 & 118.9 & 121.2 & 1.4   & 120.3 & 128.1 & 4.2   & 1.2   & 5.4 \\
    mbE107 & 113.4 & 115.4 & 1.1   & 123.8 & 132.3 & 5.1   & 8.4   & 12.7 \\
    mbE108 & 116.7 & 117.7 & 1.2   & 115.0 & 126.5 & 8.0   & -1.5  & 7.0 \\
    mbE109 & 117.8 & 120.4 & 1.3   & 115.8 & 121.6 & 3.8   & -1.7  & 1.1 \\
    mbE110 & 119.0 & 121.2 & 1.0   & 120.2 & 128.8 & 5.7   & 1.0   & 5.9 \\
    mbF101 & 253.9 & 258.8 & 2.7   & 262.7 & 307.1 & 20.7  & 3.4   & 15.7 \\
    mbF102 & 243.3 & 248.6 & 5.2   & 266.6 & 313.1 & 24.3  & 8.7   & 20.6 \\
    mbF103 & 250.8 & 256.4 & 3.9   & 277.7 & 307.7 & 22.3  & 9.7   & 16.7 \\
    mbF104 & 244.9 & 249.2 & 2.2   & 281.3 & 308.1 & 20.5  & 12.9  & 19.1 \\
    mbF105 & 262.5 & 265.3 & 1.1   & 274.7 & 304.4 & 14.5  & 4.5   & 12.9 \\
    mbF106 & 236.7 & 239.5 & 2.1   & 239.5 & 291.6 & 25.5  & 1.2   & 17.9 \\
    mbF107 & 234.2 & 237.7 & 2.2   & 279.0 & 319.8 & 32.5  & 16.1  & 25.7 \\
    mbF108 & 249.4 & 254.2 & 3.8   & 269.2 & 290.9 & 29.3  & 7.3   & 12.6 \\
    mbF109 & 253.6 & 254.4 & 0.5   & 284.5 & 319.3 & 33.1  & 10.9  & 20.3 \\
    mbF110 & 250.3 & 256.3 & 3.2   & 258.9 & 304.0 & 24.6  & 3.3   & 15.7 \\
    mbG101 & 343.2 & 349.5 & 4.0   & 395.1 & 448.5 & 53.2  & 13.1  & 22.1 \\
    mbG102 & 332.7 & 337.1 & 3.6   & 389.4 & 422.8 & 16.9  & 14.6  & 20.3 \\
    mbG103 & 369.9 & 378.9 & 5.5   & 431.8 & 488.2 & 29.9  & 14.3  & 22.4 \\
    mbG104 & 362.9 & 367.8 & 2.9   & 399.3 & 453.6 & 42.7  & 9.1   & 18.9 \\
    mbG105 & 360.1 & 364.5 & 2.2   & 418.1 & 456.9 & 29.8  & 13.9  & 20.2 \\
    mbG106 & 348.7 & 354.9 & 3.6   & 391.6 & 462.7 & 45.3  & 11.0  & 23.3 \\
    mbG107 & 341.1 & 345.2 & 2.7   & 384.8 & 442.6 & 29.4  & 11.3  & 22.0 \\
    mbG108 & 351.6 & 357.8 & 3.3   & 409.0 & 459.0 & 37.2  & 14.0  & 22.1 \\
    mbG109 & 351.9 & 356.1 & 2.7   & 405.9 & 468.9 & 34.1  & 13.3  & 24.0 \\
    mbG110 & 350.5 & 358.2 & 4.4   & 419.9 & 472.6 & 36.3  & 16.5  & 24.2 \\ \bottomrule
    \end{tabular}
  \caption{Results of the comparison of \ac{MD-SPP-H} and \ac{ALNS}, considering 100 customers and all service area sizes of \citet{Ha2018OnDrone}. The values reported are the averages over ten instances. The bold values in the $\Delta$ columns represent the averages for each customer size.} 
\label{tab:TabBenchmarkMaraEFG}
\end{table}%
}

\newcommand{\TabBenchmarkMaraMix}[1]{
\begin{table}[]
\begin{tabular}{cc|rrr|rrr|rr}
\toprule

\multicolumn{2}{c|}{} & \multicolumn{3}{c|}{MD-SPP-H} & \multicolumn{3}{c|}{ALNS} & \multicolumn{2}{c}{$\Delta$} \\
\multicolumn{1}{c}{Customers} & \multicolumn{1}{c|}{Area label} & 

\multicolumn{1}{c}{$z^\text{.}_\text{best}$} & \multicolumn{1}{c}{$z^\text{.}_\text{avg.}$} & \multicolumn{1}{c|}{$z^\text{.}_\text{std.}$} &

\multicolumn{1}{c}{$z^\text{.}_\text{best}$} & \multicolumn{1}{c}{$z^\text{.}_\text{avg.}$} & 
\multicolumn{1}{c|}{$z^\text{.}_\text{std.}$} &

\multicolumn{1}{c}{best} & \multicolumn{1}{c}{avg.} \\
\midrule

\multirow{3}{*}{50} & 
B & 78.1  & 80.3 & 1.5 & 78.0  & 85.5 & 4.9   & -0.3 & 6.0  \\
& C & 183.5 & 186.4 & 2.2 & 192.5 & 211.7 & 14.1 & 4.6  & 11.9  \\
& D & 271.3 & 274.6 & 2.9 & 289.6 & 316.3  & 20.8 & 6.3  & 13.1 \\
\cmidrule(lr){9-10}
\multicolumn{8}{r}{}    & \textbf{3.5}  & \textbf{10.3}  \\ \midrule

\multirow{3}{*}{100} &
E & 117.3 & 119.8 & 1.6 & 121.0 & 129.5 & 5.8 & 3.0  & 7.4 \\
& F & 248.0 & 252.0 & 2.7 & 269.4 & 306.6  & 24.7& 7.8  & 17.7\\
& G & 351.3 & 357.0 & 3.5 & 404.5 & 457.6 & 35.5 & 13.1 & 21.9\\
\cmidrule(lr){9-10}
\multicolumn{8}{r}{}    & \textbf{8.0}  & \textbf{15.7}  \\

\bottomrule
\end{tabular}
\caption{Results of the comparison of \ac{MD-SPP-H} and \ac{ALNS}, considering 50 and 100 customers and all service area sizes of \citet{Ha2018OnDrone}. The values reported are the averages over ten instances. The bold values in the $\Delta$ columns represent the averages for each customer size.} 
\label{tab:TabBenchmarkMaraMix}
\end{table}
}

\newcommand{\TabBenchmarkAgatzKunduUniform}[1]{
\begin{table}[]
\scriptsize
\scalebox{0.9}{
\centering
\begin{tabular}{crrrrrrrrrrrr}
\toprule
          &       & \multicolumn{3}{c}{MD-SPP-H} & \multicolumn{2}{c}{EP-All} & \multicolumn{2}{c}{$\Delta$ [\%]} & \multicolumn{2}{c}{SPP-All} & \multicolumn{2}{c}{$\Delta$ [\%]} \\
\cmidrule(lr){3-5}\cmidrule(lr){6-7} \cmidrule(lr){8-9}
\cmidrule(lr){10-11}\cmidrule(lr){12-13}
    $n$     & \multicolumn{1}{c}{ID} & \multicolumn{1}{c}{$z_\text{best}$} & \multicolumn{1}{c}{$z_\text{avg}$} & \multicolumn{1}{c}{Time (avg) [s]} & \multicolumn{1}{c}{z} & \multicolumn{1}{c}{Time [s]} & \multicolumn{1}{c}{$\Delta_\text{best}$} & \multicolumn{1}{c}{$\Delta_\text{avg}$} & \multicolumn{1}{c}{z} & \multicolumn{1}{c}{Time [s]} & \multicolumn{1}{c}{$\Delta_\text{best}$} & \multicolumn{1}{c}{$\Delta_\text{avg}$} \\
    \midrule
    \multirow{10}[1]{*}{50} & 71    & 387.7 & 390.8 & 4.6   & 393.7 & 2.0   & 1.5   & 0.7   & 389.9 & 0.0   & 0.6   & -0.2 \\
          & 72    & 420.6 & 425.8 & 5.1   & 454.7 & 2.7   & 7.5   & 6.3   & 454.7 & 0.0   & 7.5   & 6.3 \\
          & 73    & 393.8 & 396.5 & 4.3   & 410.6 & 2.5   & 4.1   & 3.4   & 409.8 & 0.1   & 3.9   & 3.2 \\
          & 74    & 409.0 & 412.0 & 3.9   & 422.6 & 2.2   & 3.2   & 2.5   & 420.5 & 0.1   & 2.7   & 2.0 \\
          & 75    & 413.0 & 416.3 & 6.2   & 441.0 & 1.9   & 6.4   & 5.6   & 441.0 & 0.1   & 6.4   & 5.6 \\
          & 76    & 370.7 & 374.8 & 4.1   & 376.1 & 2.5   & 1.4   & 0.3   & 376.1 & 0.0   & 1.4   & 0.3 \\
          & 77    & 415.8 & 419.9 & 4.7   & 439.6 & 2.2   & 5.4   & 4.5   & 434.2 & 0.1   & 4.2   & 3.3 \\
          & 78    & 410.5 & 422.0 & 5.8   & 438.2 & 2.1   & 6.3   & 3.7   & 438.2 & 0.0   & 6.3   & 3.7 \\
          & 79    & 374.4 & 380.0 & 4.6   & 397.3 & 2.1   & 5.8   & 4.4   & 397.3 & 0.0   & 5.8   & 4.4 \\
          & 80    & 361.8 & 364.2 & 4.8   & 376.0 & 2.7   & 3.8   & 3.1   & 376.0 & 0.0   & 3.8   & 3.1 \\ \addlinespace
    \multirow{10}[0]{*}{75} & 81    & 459.1 & 465.9 & 20.0  & 476.6 & 10.6  & 3.7   & 2.2   & 470.4 & 0.2   & 2.4   & 0.9 \\
          & 82    & 432.9 & 441.8 & 24.9  & 454.4 & 14.1  & 4.7   & 2.8   & 454.4 & 0.2   & 4.7   & 2.8 \\
          & 83    & 435.2 & 442.7 & 17.1  & 456.2 & 8.1   & 4.6   & 3.0   & 456.1 & 0.2   & 4.6   & 2.9 \\
          & 84    & 463.1 & 469.2 & 27.6  & 467.8 & 14.5  & 1.0   & -0.3  & 463.8 & 0.2   & 0.1   & -1.2 \\
          & 85    & 478.4 & 486.8 & 31.5  & 506.9 & 11.3  & 5.6   & 4.0   & 499.2 & 0.3   & 4.2   & 2.5 \\
          & 86    & 451.5 & 460.4 & 26.7  & 493.7 & 13.5  & 8.5   & 6.7   & 493.0 & 0.3   & 8.4   & 6.6 \\
          & 87    & 460.3 & 464.6 & 22.9  & 475.1 & 12.3  & 3.1   & 2.2   & 470.8 & 0.3   & 2.2   & 1.3 \\
          & 88    & 453.9 & 460.4 & 28.9  & 520.2 & 9.1   & 12.7  & 11.5  & 517.9 & 0.2   & 12.4  & 11.1 \\
          & 89    & 421.8 & 426.5 & 24.2  & 467.0 & 9.4   & 9.7   & 8.7   & 467.0 & 0.2   & 9.7   & 8.7 \\
          & 90    & 452.5 & 457.2 & 18.0  & 470.0 & 10.8  & 3.7   & 2.7   & 468.9 & 0.2   & 3.5   & 2.5 \\\addlinespace
    \multirow{10}[0]{*}{100} & 91    & 547.4 & 558.9 & 78.1  & 575.7 & 68.6  & 4.9   & 2.9   & 569.7 & 1.7   & 3.9   & 1.9 \\
          & 92    & 487.4 & 498.5 & 79.7  & 529.3 & 30.0  & 7.9   & 5.8   & 523.6 & 1.7   & 6.9   & 4.8 \\
          & 93    & 491.2 & 503.5 & 71.2  & 547.0 & 51.2  & 10.2  & 8.0   & 535.6 & 1.0   & 8.3   & 6.0 \\
          & 94    & 512.2 & 526.6 & 71.3  & 543.8 & 73.0  & 5.8   & 3.2   & 543.8 & 1.6   & 5.8   & 3.2 \\
          & 95    & 528.3 & 541.0 & 86.0  & 591.7 & 50.7  & 10.7  & 8.6   & 591.1 & 1.6   & 10.6  & 8.5 \\
          & 96    & 546.5 & 553.3 & 56.9  & 558.0 & 78.0  & 2.1   & 0.8   & 556.1 & 1.4   & 1.7   & 0.5 \\
          & 97    & 554.4 & 565.2 & 71.2  & 592.8 & 67.5  & 6.5   & 4.7   & 590.5 & 1.7   & 6.1   & 4.3 \\
          & 98    & 514.7 & 520.4 & 50.6  & 517.0 & 38.8  & 0.5   & -0.7  & 516.8 & 0.7   & 0.4   & -0.7 \\
          & 99    & 536.9 & 546.2 & 60.8  & 574.5 & 36.7  & 6.5   & 4.9   & 558.7 & 0.7   & 3.9   & 2.2 \\
          & 100   & 516.5 & 530.8 & 89.2  & 578.7 & 33.0  & 10.7  & 8.3   & 577.3 & 0.7   & 10.5  & 8.0 \\\addlinespace
    \multirow{10}[0]{*}{175} & 101   & 684.0 & 699.2 & 463.9 & 734.9 & 452.0 & 6.9   & 4.9   & 730.4 & 7.9   & 6.3   & 4.3 \\
          & 102   & 683.1 & 704.7 & 542.8 & 728.3 & 418.1 & 6.2   & 3.2   & 727.8 & 15.8  & 6.1   & 3.2 \\
          & 103   & 674.3 & 697.3 & 660.1 & 727.0 & 985.2 & 7.2   & 4.1   & 725.8 & 16.6  & 7.1   & 3.9 \\
          & 104   & 657.3 & 670.2 & 509.8 & 703.6 & 993.3 & 6.6   & 4.8   & 685.6 & 17.5  & 4.1   & 2.2 \\
          & 105   & 695.8 & 708.3 & 459.3 & 726.9 & 813.8 & 4.3   & 2.6   & 725.6 & 8.0   & 4.1   & 2.4 \\
          & 106   & 678.1 & 689.4 & 266.3 & 691.7 & 786.4 & 2.0   & 0.3   & 688.5 & 16.7  & 1.5   & -0.1 \\
          & 107   & 654.1 & 659.0 & 430.0 & 661.4 & 654.3 & 1.1   & 0.4   & 653.7 & 7.6   & -0.1  & -0.8 \\
          & 108   & 677.4 & 689.6 & 749.1 & 728.4 & 383.6 & 7.0   & 5.3   & 718.7 & 6.8   & 5.7   & 4.0 \\
          & 109   & 686.2 & 705.6 & 495.3 & 716.9 & 377.7 & 4.3   & 1.6   & 715.6 & 8.0   & 4.1   & 1.4 \\
          & 110   & 637.0 & 644.2 & 314.8 & 654.1 & 370.1 & 2.6   & 1.5   & 650.0 & 8.8   & 2.0   & 0.9 \\\addlinespace
    \multirow{10}[1]{*}{250} & 111   & 813.2 & 822.7 & 1480.1 & 822.0 & 2944.8 & 1.1   & -0.1  & 820.2 & 35.3  & 0.9   & -0.3 \\
          & 112   & 804.9 & 812.4 & 987.4 & 806.3 & 2534.2 & 0.2   & -0.8  & 802.0 & 31.4  & -0.4  & -1.3 \\
          & 113   & 795.3 & 814.2 & 1431.2 & 820.9 & 2475.7 & 3.1   & 0.8   & 818.9 & 34.6  & 2.9   & 0.6 \\
          & 114   & 805.8 & 817.9 & 1135.6 & 830.9 & 2749.9 & 3.0   & 1.6   & 830.6 & 36.7  & 3.0   & 1.5 \\
          & 115   & 818.0 & 830.3 & 1384.5 & 845.3 & 4871.8 & 3.2   & 1.8   & 844.6 & 37.3  & 3.2   & 1.7 \\
          & 116   & 807.7 & 819.4 & 1672.9 & 831.1 & 1969.4 & 2.8   & 1.4   & 826.4 & 37.9  & 2.3   & 0.8 \\
          & 117   & 823.0 & 828.4 & 1185.1 & 840.9 & 2255.2 & 2.1   & 1.5   & 840.9 & 32.4  & 2.1   & 1.5 \\
          & 118   & 761.3 & 774.2 & 1574.0 & 790.1 & 2191.0 & 3.6   & 2.0   & 780.2 & 33.3  & 2.4   & 0.8 \\
          & 119   & 833.9 & 851.8 & 1766.0 & 871.4 & 4910.6 & 4.3   & 2.3   & 868.8 & 85.3  & 4.0   & 2.0 \\
          & 120   & 789.0 & 801.9 & 1358.6 & 811.0 & 4684.4 & 2.7   & 1.1   & 805.2 & 69.7  & 2.0   & 0.4 \\
\cmidrule(rl){8-9}\cmidrule(rl){12-13}
& & & & &  & & 4.9 & 3.3 & & &  4.3 & 2.8 \\
    \bottomrule
    \end{tabular}%
    }
  \caption{Results of the comparison of \ac{MD-SPP-H} and \ac{EP-All} and \ac{SPP-All} for various customers ranging from $n=50$ to $250$ and uniform customer distributions.} 
\label{tab:TabBenchmarkAgatzKunduUniform}
\end{table}%
}

\newcommand{\TabBenchmarkAgatzKunduSC}[1]{
\begin{table}[]
\scriptsize
\scalebox{0.9}{
  \centering
\begin{tabular}{crrrrrrrrrrrr}
\toprule
          &       & \multicolumn{3}{c}{MD-SPP-H} & \multicolumn{2}{c}{EP-All} & \multicolumn{2}{c}{$\Delta$ [\%]} & \multicolumn{2}{c}{SPP-All} & \multicolumn{2}{c}{$\Delta$ [\%]} \\
\cmidrule(lr){3-5}\cmidrule(lr){6-7} \cmidrule(lr){8-9}
\cmidrule(lr){10-11}\cmidrule(lr){12-13}
    $n$     & \multicolumn{1}{c}{ID} & \multicolumn{1}{c}{$z_\text{best}$} & \multicolumn{1}{c}{$z_\text{avg}$} & \multicolumn{1}{c}{Time (avg) [s]} & \multicolumn{1}{c}{z} & \multicolumn{1}{c}{Time [s]} & \multicolumn{1}{c}{$\Delta_\text{best}$} & \multicolumn{1}{c}{$\Delta_\text{avg}$} & \multicolumn{1}{c}{z} & \multicolumn{1}{c}{Time [s]} & \multicolumn{1}{c}{$\Delta_\text{best}$} & \multicolumn{1}{c}{$\Delta_\text{avg}$} \\
    \midrule
\multirow{10}[1]{*}{50} & 71    & 428.2 & 432.7 & 7.6   & 446.6 & 1.9   & 4.1   & 3.1   & 446.6 & 0.0   & 4.1   & 3.1 \\
      & 72    & 483.0 & 484.2 & 3.3   & 484.3 & 2.0   & 0.3   & 0.0   & 484.3 & 0.1   & 0.3   & 0.0 \\
      & 73    & 361.5 & 365.9 & 7.7   & 376.8 & 2.1   & 4.0   & 2.9   & 376.8 & 0.0   & 4.0   & 2.9 \\
      & 74    & 528.2 & 532.1 & 7.2   & 582.5 & 1.8   & 9.3   & 8.6   & 550.7 & 0.1   & 4.1   & 3.4 \\
      & 75    & 552.1 & 561.1 & 9.0   & 589.1 & 2.3   & 6.3   & 4.8   & 565.4 & 0.1   & 2.3   & 0.8 \\
      & 76    & 539.0 & 550.0 & 6.5   & 573.7 & 1.9   & 6.0   & 4.1   & 563.5 & 0.1   & 4.4   & 2.4 \\
      & 77    & 425.3 & 430.6 & 8.5   & 442.1 & 3.1   & 3.8   & 2.6   & 442.1 & 0.1   & 3.8   & 2.6 \\
      & 78    & 541.4 & 557.0 & 7.8   & 558.1 & 2.0   & 3.0   & 0.2   & 558.1 & 0.0   & 3.0   & 0.2 \\
      & 79    & 432.7 & 441.5 & 6.2   & 465.4 & 2.6   & 7.0   & 5.1   & 451.6 & 0.1   & 4.2   & 2.2 \\
      & 80    & 550.8 & 558.5 & 6.2   & 573.8 & 3.5   & 4.0   & 2.7   & 571.1 & 0.1   & 3.6   & 2.2 \\\addlinespace
\multirow{10}[0]{*}{75} & 81    & 695.6 & 715.3 & 32.5  & 767.7 & 11.3  & 9.4   & 6.8   & 701.3 & 0.2   & 0.8   & -2.0 \\
      & 82    & 560.5 & 580.1 & 28.4  & 624.4 & 10.4  & 10.2  & 7.1   & 624.4 & 0.2   & 10.2  & 7.1 \\
      & 83    & 616.6 & 624.8 & 34.6  & 631.9 & 11.0  & 2.4   & 1.1   & 631.9 & 0.2   & 2.4   & 1.1 \\
      & 84    & 684.2 & 706.5 & 34.8  & 715.5 & 14.3  & 4.4   & 1.3   & 710.8 & 0.2   & 3.7   & 0.6 \\
      & 85    & 580.3 & 589.8 & 26.0  & 588.1 & 10.4  & 1.3   & -0.3  & 585.6 & 0.2   & 0.9   & -0.7 \\
      & 86    & 720.9 & 737.4 & 45.1  & 769.3 & 14.9  & 6.3   & 4.1   & 762.5 & 0.5   & 5.5   & 3.3 \\
      & 87    & 687.0 & 702.0 & 29.2  & 739.2 & 11.7  & 7.1   & 5.0   & 738.3 & 0.2   & 6.9   & 4.9 \\
      & 88    & 689.9 & 699.9 & 34.1  & 695.3 & 11.7  & 0.8   & -0.7  & 692.5 & 0.2   & 0.4   & -1.1 \\
      & 89    & 639.1 & 650.0 & 30.7  & 697.4 & 10.7  & 8.4   & 6.8   & 665.4 & 0.2   & 4.0   & 2.3 \\
      & 90    & 693.1 & 711.9 & 32.8  & 776.8 & 16.3  & 10.8  & 8.4   & 733.6 & 0.2   & 5.5   & 3.0 \\\addlinespace
\multirow{10}[0]{*}{100} & 91    & 778.5 & 793.9 & 78.1  & 827.4 & 37.9  & 5.9   & 4.0   & 826.5 & 0.8   & 5.8   & 3.9 \\
      & 92    & 854.7 & 877.0 & 103.1 & 890.7 & 76.6  & 4.0   & 1.5   & 887.2 & 1.5   & 3.7   & 1.1 \\
      & 93    & 844.9 & 861.5 & 116.0 & 934.4 & 44.0  & 9.6   & 7.8   & 934.3 & 0.7   & 9.6   & 7.8 \\
      & 94    & 947.8 & 971.0 & 117.6 & 1015.3 & 78.2  & 6.7   & 4.4   & 997.8 & 1.5   & 5.0   & 2.7 \\
      & 95    & 738.9 & 746.0 & 54.4  & 752.7 & 62.0  & 1.8   & 0.9   & 748.9 & 1.5   & 1.3   & 0.4 \\
      & 96    & 809.5 & 829.7 & 113.7 & 866.1 & 41.3  & 6.5   & 4.2   & 861.8 & 0.7   & 6.1   & 3.7 \\
      & 97    & 831.6 & 845.5 & 120.1 & 851.4 & 70.6  & 2.3   & 0.7   & 848.0 & 1.5   & 1.9   & 0.3 \\
      & 98    & 785.7 & 803.5 & 131.3 & 840.9 & 52.5  & 6.6   & 4.5   & 833.1 & 1.6   & 5.7   & 3.6 \\
      & 99    & 662.1 & 678.2 & 95.8  & 679.7 & 45.3  & 2.6   & 0.2   & 679.1 & 0.8   & 2.5   & 0.1 \\
      & 100   & 758.3 & 768.1 & 95.9  & 797.2 & 61.3  & 4.9   & 3.7   & 793.9 & 0.9   & 4.5   & 3.3 \\\addlinespace
\multirow{10}[0]{*}{175} & 101   & 1025.9 & 1049.3 & 791.2 & 1051.6 & 848.8 & 2.4   & 0.2   & 1021.8 & 9.5   & -0.4  & -2.7 \\
      & 102   & 1052.7 & 1065.8 & 510.4 & 1071.6 & 1776.6 & 1.8   & 0.5   & 1058.5 & 7.7   & 0.5   & -0.7 \\
      & 103   & 1117.8 & 1152.0 & 586.1 & 1071.6 & 1776.6 & -4.3  & -7.5  & 1058.5 & 7.7   & -5.6  & -8.8 \\
      & 104   & 1073.4 & 1088.8 & 591.6 & 1132.8 & 544.1 & 5.2   & 3.9   & 1119.4 & 7.1   & 4.1   & 2.7 \\
      & 105   & 1074.1 & 1115.2 & 793.0 & 1093.7 & 536.7 & 1.8   & -2.0  & 1060.6 & 7.4   & -1.3  & -5.1 \\
      & 106   & 1107.6 & 1132.5 & 604.1 & 1143.0 & 695.6 & 3.1   & 0.9   & 1138.1 & 7.0   & 2.7   & 0.5 \\
      & 107   & 1171.3 & 1200.9 & 795.6 & 1217.9 & 440.7 & 3.8   & 1.4   & 1217.9 & 7.4   & 3.8   & 1.4 \\
      & 108   & 1115.8 & 1132.4 & 693.5 & 1142.6 & 487.0 & 2.4   & 0.9   & 1126.9 & 8.9   & 1.0   & -0.5 \\
      & 109   & 1095.3 & 1112.3 & 646.1 & 1129.9 & 424.3 & 3.1   & 1.6   & 1121.5 & 10.1  & 2.3   & 0.8 \\
      & 110   & 1077.0 & 1093.2 & 616.2 & 1103.9 & 564.3 & 2.4   & 1.0   & 1103.2 & 9.3   & 2.4   & 0.9 \\\addlinespace
\multirow{10}[1]{*}{250} & 111   & 1279.5 & 1284.1 & 1467.3 & 1279.9 & 3896.2 & 0.0   & -0.3  & 1270.8 & 43.6  & -0.7  & -1.0 \\
      & 112   & 1278.6 & 1300.1 & 1425.9 & 1300.1 & 3351.2 & 1.7   & 0.0   & 1264.3 & 35.5  & -1.1  & -2.8 \\
      & 113   & 1379.7 & 1396.1 & 1675.0 & 1405.4 & 2195.8 & 1.8   & 0.7   & 1389.2 & 37.5  & 0.7   & -0.5 \\
      & 114   & 1358.5 & 1363.9 & 1325.4 & 1369.0 & 2127.4 & 0.8   & 0.4   & 1364.0 & 41.6  & 0.4   & 0.0 \\
      & 115   & 1342.8 & 1353.3 & 1374.2 & 1335.8 & 2728.1 & -0.5  & -1.3  & 1328.9 & 47.3  & -1.0  & -1.8 \\
      & 116   & 1303.7 & 1351.4 & 1770.5 & 1358.2 & 2982.4 & 4.0   & 0.5   & 1335.4 & 45.0  & 2.4   & -1.2 \\
      & 117   & 1277.1 & 1287.5 & 1684.7 & 1295.4 & 2699.1 & 1.4   & 0.6   & 1291.2 & 65.2  & 1.1   & 0.3 \\
      & 118   & 1263.3 & 1272.6 & 1613.6 & 1277.5 & 4917.2 & 1.1   & 0.4   & 1271.3 & 120.2 & 0.6   & -0.1 \\
      & 119   & 1189.0 & 1199.6 & 1298.9 & 1194.7 & 2464.2 & 0.5   & -0.4  & 1190.6 & 33.8  & 0.1   & -0.8 \\
      & 120   & 1272.9 & 1303.5 & 1650.1 & 1313.4 & 2428.5 & 3.1   & 0.7   & 1312.3 & 31.6  & 3.0   & 0.7 \\
      \cmidrule(rl){8-9}\cmidrule(rl){12-13}
& & & & &  & & 3.9 & 2.2 & & &  2.7 & 0.9 \\
\bottomrule
\end{tabular}%
    }
  \caption{Results of the comparison of \ac{MD-SPP-H} and \ac{EP-All} and \ac{SPP-All} for various customers ranging from 50 to 250 and single-center customer distributions.} 
\label{tab:TabBenchmarkAgatzKunduSC}
\end{table}%
}

\newcommand{\TabBenchmarkAgatzKunduDC}[1]{
\begin{table}[]
\scriptsize
\scalebox{0.9}{
\centering
\begin{tabular}{crrrrrrrrrrrr}
\toprule
          &       & \multicolumn{3}{c}{MD-SPP-H} & \multicolumn{2}{c}{EP-All} & \multicolumn{2}{c}{$\Delta$ [\%]} & \multicolumn{2}{c}{SPP-All} & \multicolumn{2}{c}{$\Delta$ [\%]} \\
\cmidrule(lr){3-5}\cmidrule(lr){6-7} \cmidrule(lr){8-9}
\cmidrule(lr){10-11}\cmidrule(lr){12-13}
        $n$     & \multicolumn{1}{c}{ID} & \multicolumn{1}{c}{$z_\text{best}$} & \multicolumn{1}{c}{$z_\text{avg}$} & \multicolumn{1}{c}{Time (avg) [s]} & \multicolumn{1}{c}{z} & \multicolumn{1}{c}{Time [s]} & \multicolumn{1}{c}{$\Delta_\text{best}$} & \multicolumn{1}{c}{$\Delta_\text{avg}$} & \multicolumn{1}{c}{z} & \multicolumn{1}{c}{Time [s]} & \multicolumn{1}{c}{$\Delta_\text{best}$} & \multicolumn{1}{c}{$\Delta_\text{avg}$} \\
    \midrule
\multirow{10}[1]{*}{50} & 71    & 802.0 & 805.4 & 5.1   & 846.3 & 1.9   & 5.2   & 4.8   & 846.3 & 0.4   & 5.2   & 4.8 \\
      & 72    & 788.5 & 797.5 & 8.0   & 832.7 & 2.6   & 5.3   & 4.2   & 832.7 & 0.4   & 5.3   & 4.2 \\
      & 73    & 717.3 & 730.0 & 5.1   & 762.2 & 2.5   & 5.9   & 4.2   & 762.2 & 0.6   & 5.9   & 4.2 \\
      & 74    & 702.0 & 705.5 & 6.3   & 707.5 & 4.0   & 0.8   & 0.3   & 707.5 & 0.3   & 0.8   & 0.3 \\
      & 75    & 825.3 & 833.5 & 6.8   & 887.4 & 2.5   & 7.0   & 6.1   & 883.9 & 0.7   & 6.6   & 5.7 \\
      & 76    & 795.5 & 803.6 & 8.6   & 847.2 & 2.9   & 6.1   & 5.2   & 847.2 & 0.2   & 6.1   & 5.2 \\
      & 77    & 728.0 & 734.8 & 8.4   & 742.9 & 2.5   & 2.0   & 1.1   & 736.9 & 0.7   & 1.2   & 0.3 \\
      & 78    & 734.9 & 748.3 & 6.5   & 756.9 & 3.8   & 2.9   & 1.1   & 756.9 & 0.3   & 2.9   & 1.1 \\
      & 79    & 757.5 & 772.8 & 7.8   & 823.9 & 2.1   & 8.1   & 6.2   & 823.9 & 0.3   & 8.1   & 6.2 \\
      & 80    & 861.7 & 865.7 & 7.0   & 883.8 & 2.4   & 2.5   & 2.0   & 883.8 & 0.5   & 2.5   & 2.0 \\
\multirow{10}[0]{*}{75} & 81    & 858.3 & 867.6 & 34.3  & 871.4 & 13.9  & 1.5   & 0.4   & 864.4 & 0.5   & 0.7   & -0.4 \\\addlinespace
      & 82    & 886.1 & 902.0 & 31.9  & 936.4 & 17.2  & 5.4   & 3.7   & 936.4 & 0.7   & 5.4   & 3.7 \\
      & 83    & 862.6 & 872.4 & 32.0  & 881.0 & 17.0  & 2.1   & 1.0   & 878.6 & 0.5   & 1.8   & 0.7 \\
      & 84    & 1115.1 & 1138.5 & 33.2  & 1220.5 & 13.3  & 8.6   & 6.7   & 1220.5 & 0.5   & 8.6   & 6.7 \\
      & 85    & 1054.6 & 1064.2 & 30.1  & 1088.8 & 14.6  & 3.1   & 2.3   & 1088.5 & 1.1   & 3.1   & 2.2 \\
      & 86    & 913.2 & 933.5 & 26.0  & 943.8 & 9.4   & 3.2   & 1.1   & 937.9 & 0.5   & 2.6   & 0.5 \\
      & 87    & 941.7 & 947.2 & 32.2  & 963.0 & 12.8  & 2.2   & 1.6   & 963.0 & 0.9   & 2.2   & 1.6 \\
      & 88    & 1069.0 & 1101.5 & 27.0  & 1122.5 & 17.1  & 4.8   & 1.9   & 1119.3 & 0.7   & 4.5   & 1.6 \\
      & 89    & 908.5 & 931.5 & 38.0  & 1017.3 & 11.6  & 10.7  & 8.4   & 1017.3 & 1.3   & 10.7  & 8.4 \\
      & 90    & 836.5 & 848.5 & 30.6  & 865.2 & 11.0  & 3.3   & 1.9   & 865.2 & 0.5   & 3.3   & 1.9 \\
\multirow{10}[0]{*}{100} & 91    & 1042.2 & 1055.0 & 77.6  & 1080.7 & 50.8  & 3.6   & 2.4   & 1073.6 & 2.9   & 2.9   & 1.7 \\\addlinespace
      & 92    & 988.9 & 1009.7 & 99.5  & 1028.4 & 40.0  & 3.8   & 1.8   & 1018.7 & 1.4   & 2.9   & 0.9 \\
      & 93    & 990.9 & 1005.9 & 99.6  & 1052.5 & 27.0  & 5.9   & 4.4   & 1049.9 & 2.1   & 5.6   & 4.2 \\
      & 94    & 1056.3 & 1077.0 & 100.3 & 1076.6 & 55.5  & 1.9   & 0.0   & 1074.0 & 1.4   & 1.7   & -0.3 \\
      & 95    & 1091.8 & 1115.7 & 80.4  & 1132.5 & 36.8  & 3.6   & 1.5   & 1130.8 & 4.1   & 3.5   & 1.3 \\
      & 96    & 1149.4 & 1174.8 & 108.2 & 1237.7 & 44.7  & 7.1   & 5.1   & 1182.8 & 2.9   & 2.8   & 0.7 \\
      & 97    & 1214.6 & 1237.0 & 91.8  & 1259.9 & 39.6  & 3.6   & 1.8   & 1259.9 & 1.4   & 3.6   & 1.8 \\
      & 98    & 1118.8 & 1149.0 & 81.2  & 1164.9 & 47.1  & 4.0   & 1.4   & 1160.8 & 1.4   & 3.6   & 1.0 \\
      & 99    & 1003.0 & 1015.1 & 52.8  & 1028.5 & 35.4  & 2.5   & 1.3   & 1025.4 & 2.1   & 2.2   & 1.0 \\
      & 100   & 1122.9 & 1138.0 & 80.3  & 1137.0 & 52.5  & 1.2   & -0.1  & 1113.6 & 1.4   & -0.8  & -2.2 \\\addlinespace
\multirow{10}[0]{*}{175} & 101   & 1269.1 & 1307.3 & 658.7 & 1344.9 & 640.0 & 5.6   & 2.8   & 1344.3 & 22.0  & 5.6   & 2.7 \\
      & 102   & 1491.4 & 1526.0 & 433.0 & 1520.6 & 639.1 & 1.9   & -0.4  & 1502.2 & 15.1  & 0.7   & -1.6 \\
      & 103   & 1524.3 & 1544.4 & 854.1 & 1596.2 & 710.2 & 4.5   & 3.2   & 1592.9 & 26.2  & 4.3   & 3.0 \\
      & 104   & 1397.0 & 1443.7 & 780.4 & 1455.0 & 531.0 & 4.0   & 0.8   & 1445.4 & 21.1  & 3.3   & 0.1 \\
      & 105   & 1543.5 & 1572.0 & 695.0 & 1604.7 & 615.6 & 3.8   & 2.0   & 1602.6 & 15.2  & 3.7   & 1.9 \\
      & 106   & 1549.1 & 1573.3 & 758.1 & 1597.6 & 748.8 & 3.0   & 1.5   & 1597.1 & 14.3  & 3.0   & 1.5 \\
      & 107   & 1389.7 & 1405.2 & 676.4 & 1430.2 & 501.2 & 2.8   & 1.7   & 1426.3 & 28.5  & 2.6   & 1.5 \\
      & 108   & 1511.5 & 1545.6 & 688.6 & 1555.5 & 541.4 & 2.8   & 0.6   & 1548.1 & 31.1  & 2.4   & 0.2 \\
      & 109   & 1447.5 & 1522.3 & 699.4 & 1518.1 & 620.6 & 4.7   & -0.3  & 1516.2 & 14.6  & 4.5   & -0.4 \\
      & 110   & 1541.1 & 1567.8 & 903.3 & 1602.0 & 481.1 & 3.8   & 2.1   & 1583.3 & 16.9  & 2.7   & 1.0 \\
\multirow{10}[1]{*}{250} & 111   & 1762.6 & 1779.9 & 1796.4 & 1774.7 & 2834.2 & 0.7   & -0.3  & 1764.5 & 68.1  & 0.1   & -0.9 \\\addlinespace
      & 112   & 1807.6 & 1829.3 & 1402.3 & 1819.3 & 3177.3 & 0.6   & -0.5  & 1802.9 & 124.0 & -0.3  & -1.5 \\
      & 113   & 1706.3 & 1730.3 & 1631.9 & 1747.2 & 2436.7 & 2.3   & 1.0   & 1728.6 & 69.2  & 1.3   & -0.1 \\
      & 114   & 1737.6 & 1752.1 & 1662.3 & 1772.8 & 2811.4 & 2.0   & 1.2   & 1758.5 & 170.4 & 1.2   & 0.4 \\
      & 115   & 1869.5 & 1897.6 & 1708.0 & 1973.5 & 3486.4 & 5.3   & 3.8   & 1949.4 & 127.6 & 4.1   & 2.7 \\
      & 116   & 1807.1 & 1832.6 & 1333.7 & 1808.4 & 2557.6 & 0.1   & -1.3  & 1808.0 & 65.5  & 0.1   & -1.4 \\
      & 117   & 1817.4 & 1836.3 & 1420.0 & 1831.7 & 2575.2 & 0.8   & -0.3  & 1821.2 & 94.4  & 0.2   & -0.8 \\
      & 118   & 1813.8 & 1823.2 & 1354.1 & 1841.1 & 2168.5 & 1.5   & 1.0   & 1814.2 & 63.0  & 0.0   & -0.5 \\
      & 119   & 1820.4 & 1851.3 & 1597.4 & 1823.4 & 4099.7 & 0.2   & -1.5  & 1805.1 & 69.1  & -0.8  & -2.6 \\
      & 120   & 1761.9 & 1787.4 & 1680.2 & 1827.0 & 2910.7 & 3.6   & 2.2   & 1790.2 & 101.8 & 1.6   & 0.2 \\
      \cmidrule(rl){8-9}\cmidrule(rl){12-13}
& & & & &  & & 3.6 & 2.1 & & &  3.1 & 1.5 \\
\bottomrule
\end{tabular}%
    }
  \caption{Results of the comparison of \ac{MD-SPP-H} and \ac{EP-All} and \ac{SPP-All} for various customers ranging from 50 to 250 and double-center customer distributions.} 
\label{tab:TabBenchmarkAgatzKunduDC}
\end{table}%
}

\newcommand{\TabBenchmarkAgatzKundu}[1]{
\begin{table}[]
\resizebox{\textwidth}{!}{
\begin{tabular}{c|c|rrr|rr|rr||rr|rr}
\toprule

\multicolumn{2}{c|}{} & \multicolumn{3}{c|}{MD-SPP-H} & \multicolumn{2}{c|}{EP-All} & \multicolumn{2}{c||}{$\Delta$} & \multicolumn{2}{c|}{SPP-All} & \multicolumn{2}{c}{$\Delta$} \\
\multicolumn{1}{c|}{Dist.} &\multicolumn{1}{c|}{Customers} & 
\multicolumn{1}{c}{$z^\text{.}_\text{best}$} & \multicolumn{1}{c}{$z^\text{.}_\text{avg.}$} & \multicolumn{1}{c|}{Run time} & 
\multicolumn{1}{c}{$z$} & \multicolumn{1}{c|}{Run time} & 
\multicolumn{1}{c}{best} & \multicolumn{1}{c||}{avg.} &

\multicolumn{1}{c}{$z$} & \multicolumn{1}{c|}{Run time} & 
\multicolumn{1}{c}{best} & \multicolumn{1}{c}{avg.}\\
\midrule
\multirow{5}{*}{U} & 
50  & 395.7 & 400.2 & 4.8    & 415.0 & 2.3    & 4.5 & 3.5 & 413.8 & 0.1  & 4.3 & 3.2 \\
&75  & 450.9 & 457.6 & 24.2   & 478.8 & 11.4   & 5.7 & 4.3 & 476.1 & 0.2  & 5.2 & 3.8 \\
&100 & 523.5 & 534.4 & 71.5   & 560.8 & 52.8   & 6.6 & 4.6 & 556.3 & 1.3  & 5.8 & 3.9 \\
&175 & 672.7 & 686.7 & 489.1  & 707.3 & 623.4  & 4.8 & 2.9 & 702.2 & 11.4 & 4.1 & 2.1 \\
&250 & 805.2 & 817.3 & 1397.5 & 827.0 & 3158.7 & 2.6 & 1.2 & 823.8 & 43.4 & 2.2 & 0.8\\      
\cmidrule(lr){8-9} \cmidrule(lr){12-13}
\multicolumn{7}{r}{}    & \textbf{4.9}  & \textbf{3.3} & \multicolumn{2}{r}{} & \textbf{4.3}  & \textbf{2.8} \\  \midrule

\multirow{5}{*}{SC} & 
50  & 484.2  & 491.4  & 7.0    & 509.2  & 2.3    & 4.8 & 3.4 & 501.0  & 0.1  & 3.4 & 2.0  \\
&75  & 656.7  & 671.8  & 32.8   & 700.5  & 12.3   & 6.1 & 4.0 & 684.6  & 0.3  & 4.0 & 1.9  \\
&100 & 801.2  & 817.4  & 102.6  & 845.6  & 56.9   & 5.1 & 3.2 & 841.0  & 1.2  & 4.6 & 2.7  \\
&175 & 1091.1 & 1114.2 & 662.8  & 1115.9 & 809.5  & 2.2 & 0.1 & 1102.6 & 8.2  & 1.0 & -1.1 \\
&250 & 1294.5 & 1311.2 & 1528.6 & 1312.9 & 2979.0 & 1.4 & 0.1 & 1301.8 & 50.1 & 0.5 & -0.7\\
\cmidrule(lr){8-9} \cmidrule(lr){12-13}
\multicolumn{7}{r}{}    & \textbf{3.9}  & \textbf{2.2} & \multicolumn{2}{r}{} & \textbf{2.7}  & \textbf{0.9} \\  \midrule

\multirow{5}{*}{DC}&50  & 771.3  & 779.7  & 7.0    & 809.1  & 2.7    & 4.6 & 3.5 & 808.1  & 0.4  & 4.5 & 3.4  \\
&75  & 944.5  & 960.7  & 31.5   & 991.0  & 13.8   & 4.5 & 2.9 & 989.1  & 0.7  & 4.3 & 2.7  \\
&100 & 1077.9 & 1097.7 & 87.2   & 1119.9 & 42.9   & 3.7 & 2.0 & 1109.0 & 2.1  & 2.8 & 1.0  \\
&175 & 1466.4 & 1500.8 & 714.7  & 1522.5 & 602.9  & 3.7 & 1.4 & 1515.9 & 20.5 & 3.3 & 1.0  \\
&250 & 1790.4 & 1812.0 & 1558.6 & 1821.9 & 2905.8 & 1.7 & 0.5 & 1804.3 & 95.3 & 0.7 & -0.5 \\
\cmidrule(lr){8-9} \cmidrule(lr){12-13}
\multicolumn{7}{r}{}    & \textbf{3.6}  & \textbf{2.1} & \multicolumn{2}{r}{} & \textbf{3.1}  & \textbf{1.5} \\ 

\bottomrule
\end{tabular}
}
\caption{Results of the comparison of \ac{MD-SPP-H} and \ac{EP-All} and \ac{SPP-All} for various customers ranging from 50 to 250 and uniform (U), single-center (SC), and double-center (DC) customer distributions. The values reported are the averages over ten instances. The bold values in the $\Delta$ columns represent the averages for each customer distribution.} 
\label{tab:TabBenchmarkAgatzKundu}
\end{table}
}

\begin{document}
% -----------------------------------
\begin{frontmatter}


%% use the \tnoteref{} command within \title for footnotes;
%% use the \tnotetext[]{} command for the associated footnote;
%% use the \fnref command within \author or \address for footnotes;
%% use the \fntext command for the associated footnote;
%% use the \corref command within \author for corresponding author footnotes;
%% use the \cortext command for the associated footnote;
%% use the \lead command for the email address,

\title{The Flying Sidekick Traveling Salesman Problem with Multiple Drops: A Simple and Effective Heuristic Approach}


\author[ETH-Address]{Sarah K. Schaumann}
\author[PUCV-Address,TTU-Address]{Abhishake Kundu}
\author[MIT-Address]{Juan C. Pina-Pardo\corref{mycorrespondingauthor}}
\cortext[mycorrespondingauthor]{Corresponding author}
\ead{juanpina@mit.edu}
\author[MIT-Address]{Matthias Winkenbach}
\author[PUCV-Address]{Ricardo A. Gatica}
\author[ETH-Address]{Stephan M. Wagner}
\author[TTU-Address]{Timothy I. Matis}

\address[ETH-Address]{ETH Zurich, Chair of Logistics Management, Zurich, Switzerland}
\address[PUCV-Address]{Pontificia Universidad Católica de Valparaíso, Escuela de Ingeniería Industrial, Valparaíso, Chile}
\address[TTU-Address]{Texas Tech University, Department of Industrial, Manufacturing and Systems Engineering, Lubbock, USA}
\address[MIT-Address]{Massachusetts Institute of Technology, Center for Transportation and Logistics, Cambridge, USA}

\begin{abstract}
We study the \ac{FSTSP-MD}, a multi-modal last-mile delivery model where a single truck and a single drone cooperatively deliver customer packages. In the \ac{FSTSP-MD}, the drone can be launched from the truck to deliver multiple packages before it returns to the truck for a new delivery operation. The objective is to find the synchronized truck and drone delivery routes that minimize the completion time of the delivery process. We develop a simple and effective heuristic to solve the \ac{FSTSP-MD} based on an order-first, split-second scheme. The core component of our heuristic is a novel split algorithm that finds \ac{FSTSP-MD} solutions in polynomial time for a given sequence of customers. We embed this split algorithm into a simple heuristic approach that combines standard local search and diversification techniques. The simplicity of our heuristic does not sacrifice performance: we show that it consistently outperforms state-of-the-art solution approaches developed for both the \ac{FSTSP-MD} and the FSTSP (i.e., the single-drop case) through extensive numerical experiments. Based on both stylized and real-world instances, we also show that the \ac{FSTSP-MD} substantially reduces completion times compared to traditional truck-only delivery systems. We provide extensive managerial insights into the impacts of drone capabilities and customer distribution on delivery efficiency. Our discussion compares the benefits of drones with greater payload capacity and those with greater speed. We highlight which service area characteristics increase savings but also require enhanced drone capabilities.
\end{abstract}

\end{frontmatter}
% -----------------------------------
%%\linenumbers
\section{Introduction} \label{sec:introduction}

The global last-mile delivery market is expected to grow to over 200 billion U.S. dollars by 2027, marking a significant increase from 108 billion U.S. dollars in 2020 \citep{Statista2023Global2020-2027}. This growth is primarily driven by the increased e-commerce demand \citep{Samet2023EcommerceOff}. At the same time, customer expectations for fast delivery are becoming increasingly demanding. For instance, \cite{SupplyChainDive2023ConsumersToday} finds that nearly half of online consumers abandon their shopping carts if delivery times are too long or unspecified.

Meeting highly demanding consumer expectations for speedy delivery requires significant changes in technology and processes \citep{SupplyChainDive2023ConsumersToday}. Consequently, leading logistics service providers (e.g., Amazon and UPS) are investing in new technologies, such as drones, to streamline and expedite their last-mile delivery processes \citep{Vasani2023FAAFlights, Chen2023AmazonDrone}. \cite{Cornell2023Commercial2023} find that drone deliveries have increased by more than 80\% between 2021 and 2022 (equivalent to about 875,000 deliveries worldwide), with an estimated 500,000 commercial deliveries in the first half of 2023.

Cooperative delivery systems between conventional ground vehicles (such as classical trucks) and aerial cargo drones have gained increasing attention recently \citep{Moshref-Javadi2021ApplicationsReview}. The concept involves using trucks as mobile stations for one or multiple drones (see, e.g., \cite{Etherington2017Mercedes-BenzZurich}). The trucks deliver packages independently and, whenever possible and cost-effective, load the drones with packages for autonomous delivery to customers, and then meet the drones again after delivery and before the drone flight endurance is exceeded \citep{Roberti2021ExactDrone}. 

In the academic literature, the combination of trucks and drones for last-mile logistics has been predominantly investigated under the assumption that drones can only make a single delivery per sortie \citep{Moshref-Javadi2021ApplicationsReview, Dukkanci2024FacilityReview}. However, with recent breakthroughs in drone technology, this limitation is evolving. Drones can now make multiple deliveries in a single sortie, as long as they adhere to endurance and payload constraints \citep{Poikonen2020Multi-visitProblem}. An example is the Wingcopter 198 manufactured by \cite{Wingcopter2023Wingcopter198}, which can make up to three separate deliveries to multiple locations during a single sortie. 

This paper studies the \acf{FSTSP-MD}, a last-mile delivery model where a single truck and drone cooperatively deliver customer packages. In this model, which we describe in further detail in Section \ref{sec:problem_definition}, the drone can be launched from the truck to deliver multiple packages before it returns to the truck for a new delivery operation. The objective is to determine the truck and drone delivery assignments and their corresponding routes that minimize the completion time of the delivery process, defined as the time when the last vehicle returns to the depot.

We develop the \ac{MD-SPP-H}, a simple and effective heuristic approach for the \ac{FSTSP-MD} based on an order-first, split-second scheme \citep{Prins2014Order-firstReview}. \ac{MD-SPP-H} is conceptually simple, easy to implement, and achieves state-of-the-art results. It combines standard local search and diversification techniques with a novel split algorithm that finds \ac{FSTSP-MD} solutions in polynomial time for a given sequence of customers. Unlike state-of-the-art heuristics developed for the \ac{FSTSP-MD}, \ac{MD-SPP-H} allows users to define the maximum number of drops the drone can perform in a single delivery operation (or sortie), a practical constraint for existing commercial cargo drones such as the Wingcopter 198 \citep{Alamalhodaei2021WingcopterSky}. Notably, \ac{MD-SPP-H} provides the flexibility to be used for other combinations of vehicles, such as a truck paired with a motorcycle or a cargo bike. For instances of up to 250 customers, we show that \ac{MD-SPP-H} consistently outperforms state-of-the-art heuristics developed for the \ac{FSTSP} (where the drone can only deliver a single package per sortie) and the \ac{FSTSP-MD}. Finally, we investigate the effects of drone capacity, drone speed, drone flight endurance, and customer distribution using both stylized and real-world instances.

The remainder of the paper is structured as follows. In Section \ref{sec:literature}, we review the extant literature and present our main contributions. We then formally define the \ac{FSTSP-MD} in Section \ref{sec:problem_definition}. Section \ref{sec:methodology} describes the \ac{MD-SPP-H}. Section \ref{sec:heuristic_performance} presents extensive computational results to show the performance of \ac{MD-SPP-H} compared to state-of-the-art heuristics. In Section \ref{sec:managerial_insights}, we examine how the performance of cooperative truck-and-drone delivery systems depends on input parameters and use cases. Section \ref{sec:practice} discusses implications for practice and policymakers. Finally, we report our conclusions in Section \ref{sec:conclusion}.

\section{Literature Review} \label{sec:literature}

To the best of our knowledge, \cite{Murray2015TheDelivery} were the first to propose the combination of trucks and drones for last-mile logistics. The authors introduce the \ac{FSTSP}, where a single truck is supported by a single drone restricted to perform a single delivery in each sortie. Subsequently, \cite{Agatz2018OptimizationDrone} study an extension of the \ac{FSTSP} where the truck is allowed to wait at the launch location or revisit a customer location to retrieve the drone. \cite{Agatz2018OptimizationDrone} propose two order-first, split-second heuristic approaches based on local search and dynamic programming.

Since the introduction of the \ac{FSTSP}, several extensions have been proposed in the literature. As highlighted in several recent surveys such as \cite{Macrina2020Drone-aidedReview}, \cite{Moshref-Javadi2021ApplicationsReview}, and \cite{Dukkanci2024FacilityReview}, the most prominent extensions emerge from either (i) considering multiple drones, (ii) increasing the number of truck-and-drone tandems, (iii) augmenting the capability of drones (most prominently allowing multiple drops), or a combination of them. 
All these collaborative truck-and-drone systems are characterized by a high level of synchronization between the truck and drone routes \citep{Moshref-Javadi2021ApplicationsReview}. 
While beyond the scope of this review, it is worth mentioning that there are other delivery systems where either the truck or the drone serves as the primary delivery vehicle, while the other only supports the delivery process. 
In drone-centric systems, the primary vehicle is the drone, with examples including systems involving drone delivery stations \citep[see, e.g.][]{Chauhan2019MaximumDrones,Zhu2022Two-stageDrones, Hong2018APlanning} or trucks acting as mobile depots \citep[see, e.g.][]{Dukkanci2021MinimizingOptimization, YoungJeong2023DroneAnalysis, Poikonen2020TheProblem, Poikonen2020Multi-visitProblem}. We refer the reader to \cite{Dukkanci2024FacilityReviewb} for a comprehensive review on drone-centric delivery systems. 
On the other hand, in truck-centric systems, the truck is the main vehicle, and drones are used, for example, to resupply delivery trucks with newly available orders while en route \citep[see, e.g.][]{Dayarian2020Same-dayResupply, Pina-Pardo2021TheResupply, Pina-Pardo2024FleetDelivery, Pina-Pardo2024DynamicDelivery}.

We now review related works on the three main \ac{FSTSP} extensions and combinations thereof below. 

\paragraph{(i) Single truck with multiple single-drop drones}
\cite{Murray2020TheDrones} extend the classical \ac{FSTSP} to consider multiple heterogeneous drones to support a single delivery truck. They formulate a \ac{MILP} to solve small-scale instances and develop a heuristic approach to solve problems of realistic sizes. Shortly after, \cite{Raj2020TheSpeeds} investigate the benefits of allowing drones to travel at varying speeds (i.e., drone speeds are decision variables). To ensure the safe launch and recovery of multiple drones at the same location, the authors introduce a scheduling mechanism to explicitly queue the drones and avoid mid-air collisions.
\cite{Raj2020TheSpeeds} propose a three-phased heuristic approach that dynamically adjusts the drone speeds. Their computational results show that variable-speed drones tend to consume less energy per sortie and reduce truck waiting times.
\cite{Kang2021AnProblem} study a routing problem with one truck and multiple heterogeneous drones, where the truck must wait for the drones to return before continuing its travel. The authors develop a Benders decomposition algorithm that can solve instances with up to 50 customers.
Concurrently, \cite{Cavani2021ExactDrones} propose a branch-and-cut algorithm for the \ac{FSTSP} with multiple drones, which can solve instances with up to 24 customers to optimality.
Further, \cite{Moshref-Javadi2021AModels} develop a heuristic approach that can optimize different truck-and-drone routing systems, which differ depending on the level of synchronization required between the truck and the drones (e.g., the truck is allowed to wait for the drones in the same launch location or not).
%For other studies, please refer to the recent comprehensive review of \cite{Dukkanci2024FacilityReview}.

\paragraph{(ii) Multiple trucks with single-drop drone(s)}
\cite{Sacramento2019AnDrones} extend the \ac{FSTSP} and propose a routing problem that minimizes the costs of operating multiple truck-and-drone tandems in parallel to serve customers. The authors propose an \ac{ALNS} metaheuristic to solve instances with up to 200 customers. 
\cite{Wang2022TruckdroneTime} examine the impact of time-dependent road travel times on a delivery system comprised of multiple trucks-and-drone tandems. The authors propose a \ac{MILP} formulation and an iterative local search heuristic to minimize the total logistics costs (consisting of fixed and variable costs for using vehicles, as well as time-dependent delivery costs of trucks). In their numerical experiments, \cite{Wang2022TruckdroneTime} show the cost benefits of using drones when considering time-dependent traffic conditions.
Regarding works considering multiple trucks equipped with multiple drones, \cite{Schermer2019AVariants} and \cite{Schermer2019AOperations} formulate \acp{MILP} to solve small-scale instances and propose heuristics for larger-scale instances. While \cite{Schermer2019AVariants} allow drones to be launched and retrieved at the same customer location, \cite{Schermer2019AOperations} also permit dispatching drones at discrete points along the network arcs to serve customers. 
\cite{Kitjacharoenchai2019MultipleApproach} propose a \ac{MILP} formulation and a genetic algorithm heuristic for the problem variant in which drones are not dedicated to any specific truck. The authors show that combining multiple trucks and drones can substantially reduce completion times compared to traditional truck-only delivery systems.
\cite{Tamke2021ADrones} propose an exact approach that can solve instances with up to 30 nodes of the problem variant considered in \cite{Schermer2019AVariants} without drone loops.

\paragraph{(iii) Single truck with a multi-drop drone}
Our study focuses on the third extension of the canonical \ac{FSTSP}, which considers that the drone, in a single truck-and-drone tandem, can serve multiple customers per sortie. Closely related works in the literature have been proposed by \cite{Gonzalez-R2020Truck-dronePlanning}, \cite{Liu2021Two-EchelonDrone}, and \cite{WindrasMara2022AnDrops}. Before providing a detailed review of each of these works, note that other combinations of this third extension with the other two have been proposed, involving increased numbers of trucks, drones, or both. The interested reader can refer to \cite{Leon-Blanco2022AProblem}, \cite{Luo2021TheMulti-Drones}, \citet{Poikonen2020Multi-visitProblem}, and \cite{Gonzalez-R2024AProblem} for works considering one truck and multiple multi-drop drones, and to \cite{Wang2019VehicleDrones}, \cite{Kitjacharoenchai2020TwoDelivery}, \cite{Gu2022AVisits}, \cite{Luo2022ATrips}, \cite{Yin2023AWindows}, \cite{Jiang2024AServices} and \cite{Meng2024TheWindows} for studies considering multiple trucks and multiple multi-drop drones.
\color{black}

To the best of our knowledge, \cite{Gonzalez-R2020Truck-dronePlanning} developed the first \ac{MILP} formulation for the \ac{FSTSP-MD}. In contrast to our problem definition (see Section \ref{sec:problem_definition}), the authors assume that the drone has an unlimited payload capacity, so the maximum number of deliveries per sortie is only restricted by its maximum flight endurance. Since the \ac{MILP} formulation is not able to solve any of the tested instances to optimality, \cite{Gonzalez-R2020Truck-dronePlanning} also propose an \ac{IGH} approach, using a simulated annealing scheme to escape local optima. They solved instances with up to 250 customers using their \ac{IGH} approach.

% \cite{Ha2021TheDrone} present a \ac{MILP} formulation and an iterative local search approach based on an order-first, split-second scheme. Specifically, the authors extended the $\mathcal{O}(n^4)$ splitting algorithm proposed in \citet{Ha2018OnDrone} for the single-drop scenario. In their numerical experiments, \cite{Ha2021TheDrone} show the benefits of allowing multiple deliveries per drone sortie on randomly generated instances.

Subsequently, \cite{Liu2021Two-EchelonDrone} propose a route-based \ac{MILP} formulation that considers an energy consumption function for the drone battery based on the distance traveled and the weight of the packages carried by the drone (i.e., the drone is not constrained by the actual number of \emph{separate} packages it can deliver per sortie, but by the total energy it consumes per sortie). Since the proposed \ac{MILP} formulation cannot solve instances of more than five customers, the authors also propose a heuristic approach combining simulated annealing and tabu search. \cite{Liu2021Two-EchelonDrone} solved instances of up to 100 customers.

More recently, \cite{WindrasMara2022AnDrops} propose a \ac{MILP} formulation and an \ac{ALNS} heuristic. Similar to \cite{Gonzalez-R2020Truck-dronePlanning}, the maximum number of deliveries per drone sortie is only limited by the drone flight endurance (i.e., the drone is uncapacitated). Their computational results show that the \ac{MILP} formulation can solve instances of up to eight customers. Further, the authors show that their \ac{ALNS} heuristic outperforms the \ac{IGH} of \cite{Gonzalez-R2020Truck-dronePlanning}. % For the interested reader, other studies considering one truck and one multi-drop drone can be found in \cite{Luo2017AVehicle} and \cite{Masone2022TheImprovements}.

\subsection{Research Gaps and Contributions} % SO WHAT?

Our literature review reveals that most works assume that drones can only make a single delivery per sortie \citep{Moshref-Javadi2021ApplicationsReview}, even though existing commercial drones can make multiple deliveries during a single sortie \citep{Alamalhodaei2021WingcopterSky, Sakharkar2021A2ZDrone}. Further, studies addressing the \ac{FSTSP-MD} \citep{Gonzalez-R2020Truck-dronePlanning, Liu2021Two-EchelonDrone, WindrasMara2022AnDrops} consider that drones are limited by endurance only, and have an unlimited payload capacity (in terms of the number of separate packages the drone can deliver per sortie), which is inconsistent with current drone technology. Finally, due to the inherent complexity of the \ac{FSTSP-MD}, existing exact approaches fail to address instances with more than a few customers, making them impractical for real-world last-mile logistics applications (e.g., UPS visits roughly 120 customers per route; see \cite{Holland2017UPSRoutes}). Consequently, we make the following contributions to the extant literature.

First, we develop \ac{MD-SPP-H}, a heuristic approach for the \ac{FSTSP-MD} based on an order-first, split-second scheme. The core component of \ac{MD-SPP-H} is a novel split algorithm based on a shortest path problem, which finds \ac{FSTSP-MD} solutions in polynomial time for a given sequence of customers. We embed this split algorithm into a local search procedure combined with diversification techniques. %that guide the search toward promising and previously unexplored regions of the solution space. %Overall, \ac{MD-SPP-H} is a conceptually simple and easy-to-implement heuristic. 
Unlike existing solution approaches, \ac{MD-SPP-H} allows users to define the maximum number of drops the drone can perform in a single sortie, a critical feature for real-world applications of commercial cargo drones. %, in line with existing commercial cargo drones such as the Wingcopter 198 \citep{Alamalhodaei2021WingcopterSky}.
    
Second, through extensive computational experiments over well-established instances of up to 250 customers, we show that \ac{MD-SPP-H} consistently outperforms state-of-the-art heuristics developed for the \ac{FSTSP-MD} and the \ac{FSTSP} (i.e., the single-drop scenario). Additionally, compared to the solutions obtained by the exact approach of \cite{Vasquez2021AnDecomposition} over small \ac{FSTSP} instances with up to 20 customers, \ac{MD-SPP-H} solves most instances to optimality, and with a minor average optimality gap of 0.15\%.
    
Lastly, we provide managerial insights regarding the effects of drone capacity, drone speed, flight endurance, and customer distribution, based on both stylized and real-world datasets. Notably, we show that combining a truck with a multi-drop drone can substantially reduce the total time needed to serve all customers compared to the truck-only scenario. We observe diminishing marginal returns of increasing both the number of maximum possible drops per drone sortie and the drone speed, with incremental benefits depending on the remaining drone operational characteristics and the distribution of service locations.

\section{Problem Definition} \label{sec:problem_definition}

We develop a heuristic to solve the \ac{FSTSP-MD}, one of the main extensions of classical \ac{FSTSP} as introduced by \cite{Murray2015TheDelivery}. Unlike the \ac{FSTSP}, the drone can perform not just one but multiple drops per sortie. Except for this enhanced drone capability, we follow the standard assumptions of the canonical \ac{FSTSP}. For completeness, we provide a detailed problem definition and highlight when deviating from the standard \ac{FSTSP} assumptions.

A single truck and a single drone cooperatively deliver customer packages. The truck has an unlimited capacity, and can readily accommodate the drone and all packages, with no restriction on its travel time or distance. In contrast, the drone has to meet two constraints, which are a specified maximum airtime per flight and a limited payload capacity. 
The first constraint reflects practical operations, where the flight endurance is limited by the drone battery, which can be swapped whenever the drone returns to the truck. It is assumed that the battery can be swapped in negligible time. 
The second constraint limits the maximum number of packages the drone can carry and deliver to a predefined number. Unlike the \ac{FSTSP}, the drone can carry and deliver multiple packages in the same sortie. It is assumed that each customer has a unit demand (i.e., each customer orders one package) and can be served by either the truck or the drone.
Consequently, the number of customer visits the drone can undertake in a single sortie is restricted by its maximum flight endurance and payload capacity. 

The cooperative truck-and-drone delivery system starts from and finishes at the depot. The vehicles may depart and return to the depot together (meaning the truck carries the drone) or independently. The objective is to determine the truck and drone delivery assignments and their corresponding synchronized routes that minimize the completion time of the delivery process (i.e., the total elapsed time until the last vehicle arrives at the depot).

To orchestrate the operations between the truck and the drone, the following rules are defined. The drone can either be transported by the truck or dispatched to serve customers. In a drone delivery operation, the drone retrieves packages from the truck, delivers them to the customers, and then meets up with the truck or travels to the depot. Notably, both the launch and rendezvous points are confined to the depot or customer locations, and these points must be different (i.e., the truck is not allowed to stop and wait for the drone at the same location). Further, any previously visited customer location cannot be visited again, even for launch or recovery processes. Regarding the rendezvous locations, if the truck arrives first, it has to wait for the drone to be recovered; conversely, if the drone arrives first, it has to wait for the truck to recover it. The drone can only land atop the truck, so in the latter case, the drone wait time has to be accounted for when determining feasible drone delivery operations. 

%\vspace{0.5cm}
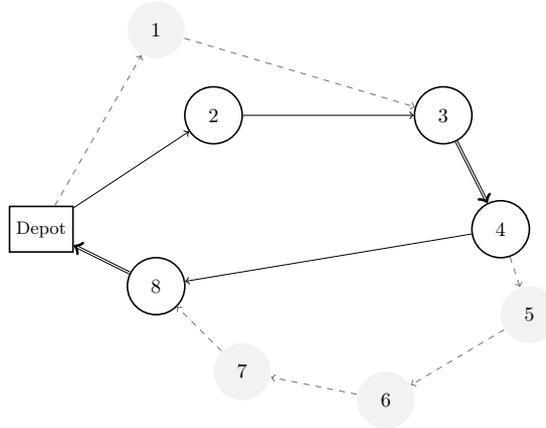
\begin{figure} [htbp]
    \centering
    \resizebox{0.45\textwidth}{!}{
	\begin{tikzpicture} 
		\begin{scope} [every node/.style={circle, thick, draw, minimum size=1cm}]
		\node [rectangle, minimum size=0.8cm] (depot) at (-1,0) {\small Depot};
		\node  (n1) at (2,2) {2};
		\node  [draw=none] [fill=A] (n2) at (1,3.5)  {1};
		\node  (n3) at (6,2) {3};
		\node  (n4) at (7,0)  {4};
		\node  [draw=none] [fill=A] (n5) at (7.5,-1.5)  {5};
		\node  [draw=none] [fill=A] (n6) at (5,-3)  {6};
		\node  [draw=none] [fill=A] (n7) at (2.5,-2.5)  {7};
  		\node  (n8) at (1,-1)  {8};
		\end{scope}
		
		% Truck:
		\foreach \from/\to in {depot/n1,n1/n3,n4/n8}
		\draw [->, line width=0.5pt] (\from) -- (\to);

  		% Drone:
		\foreach \from/\to in {depot/n2,n2/n3,n4/n5, n5/n6,n6/n7,n7/n8}
		\draw [->, dashed, F, line width=0.5pt] (\from) -- (\to);

  		% Both:
		\foreach \from/\to in {n3/n4, n8/depot}
		\draw [->, line width=0.5pt, double] (\from) -- (\to);
	
	\end{tikzpicture}
	}
	%\vspace{0.5cm}
	\caption{Illustrative example of a feasible \ac{FSTSP-MD} solution.}
	\label{fig:example}
\end{figure}

Figure \ref{fig:example} shows an illustrative example of a feasible \ac{FSTSP-MD} solution for an instance of eight customers. The truck leaves the depot to serve Customer 2, and the drone is dispatched to Customer 1. The vehicles meet at Customer 3 (the first to arrive must wait for the other). The drone is carried by the truck from Customer 3 to Customer 4. Then, the drone is launched from the truck to serve Customers 5, 6, and 7 while the truck travels to serve Customer 8. The vehicles meet again at Customer 8 and travel together to the depot.

\section{The Multi-Drop Shortest Path Problem--Based Heuristic} 
\label{sec:methodology}

Because of the impracticality of using exact approaches to address realistically-sized instances \citep[see, e.g.,][]{WindrasMara2022AnDrops,Gonzalez-R2020Truck-dronePlanning}, we develop the \acf{MD-SPP-H}, a simple and effective order-first, split-second heuristic approach (see \cite{Prins2014Order-firstReview} for a comprehensive review on order-first, split-second heuristics for routing problems).  
Due to our interest in solving practical last-mile logistics instances (which normally include more than 100 customers; see \cite{Holland2017UPSRoutes} and \cite{Merchan20242021Set}), we do not provide a \ac{MILP} formulation for the \ac{FSTSP-MD}. We provide a state-of-the-art heuristic that can solve instances with up to 250 customers in reasonable computational times (see Section \ref{sec:heuristic_performance}).
As mentioned in Section \ref{sec:literature}, existing \ac{MILP} models fail to address instances with more than eight customers. 
The interested reader can refer to \cite{Gonzalez-R2020Truck-dronePlanning} and \cite{WindrasMara2022AnDrops} for \ac{MILP} formulations of the \ac{FSTSP-MD} in which the drone payload capacity is not limited.

\begin{table}[H]
    \centering
    \caption{Summary of terms and acronyms used throughout this section.}
    \renewcommand{\arraystretch}{1}
    \scalebox{0.8}{
    \begin{tabular}{ll}
    \toprule
    {Term} & {Definition} \\
        \midrule
        \multicolumn{2}{l}{\textit{Acronyms}} \\ 
        \midrule
        \acs{TSP} & \acl{TSP}. \\
        % \acs{FSTSP} & \acl{FSTSP} \\
        \acs{FSTSP-MD} & Flying Sidekick Traveling Salesman Problem with Multiple Drops. \\ 
        % \acs{SPP} & \acl{SPP} \\
        \acs{MD-SPP} & Multi-Drop Shortest Path Problem. \\
        \acs{MD-SPP-H} & Multi-Drop Shortest Path Problem--Based Heuristic. \\
        \midrule
        \multicolumn{2}{l}{\textit{Terms}} \\ 
        \midrule
        TSP tour & Sequence of nodes where each customer appears exactly once and begins and ends at the depot. \\
        TSP solution space & Set of all possible TSP tours. \\
        FSTSP-MD solution & Feasible solution to the \ac{FSTSP-MD}. \\ 
        Benchmark solution & Best \ac{FSTSP-MD} solution found during the current improvement phase of \ac{MD-SPP-H}.\\
        \bottomrule
    \end{tabular}
    }
 \label{tab:acronyms}
\end{table}

In the remainder of this section, we begin by presenting a high-level description of the \ac{MD-SPP-H}. We then formally describe our split algorithm, referred to as the \ac{MD-SPP} in the following. Finally, we describe the generation of the initial solution and how \ac{MD-SPP-H} explores the TSP solution space. Table \ref{tab:acronyms} summarizes the acronyms and terms used throughout this section.

\subsection{High-level Heuristic Description}

A flowchart of \ac{MD-SPP-H} is presented in Figure \ref{fig:heuristic}. We begin by generating an initial TSP tour and its corresponding \ac{FSTSP-MD} solution (through the \ac{MD-SPP}), which are designated as the TSP tour and benchmark solution for the first iteration, respectively. In the improvement phase, a set of neighborhood operators is applied to the current TSP tour to generate a composite neighborhood. Through each neighborhood operator, we create a neighbor TSP tour and then solve the \ac{MD-SPP} to generate a new \ac{FSTSP-MD} solution. Next, we update the benchmark solution if the new \ac{FSTSP-MD} solution is better. After exploring the entire composite neighborhood of the current TSP tour, we update the \emph{best-so-far} (BSF) solution with the benchmark solution if the latter is better. Further, if the benchmark solution was not updated during the improvement phase, we apply perturbation to generate the TSP tour and benchmark solution for the next iteration. We repeat this iterative process until certain termination criteria (e.g., a maximum number of consecutive iterations without improvement or a certain time limit) are attained. \ac{MD-SPP-H} ends by returning the BSF solution.

\begin{figure}[H]
    \centering
    \includegraphics[width=1\textwidth]{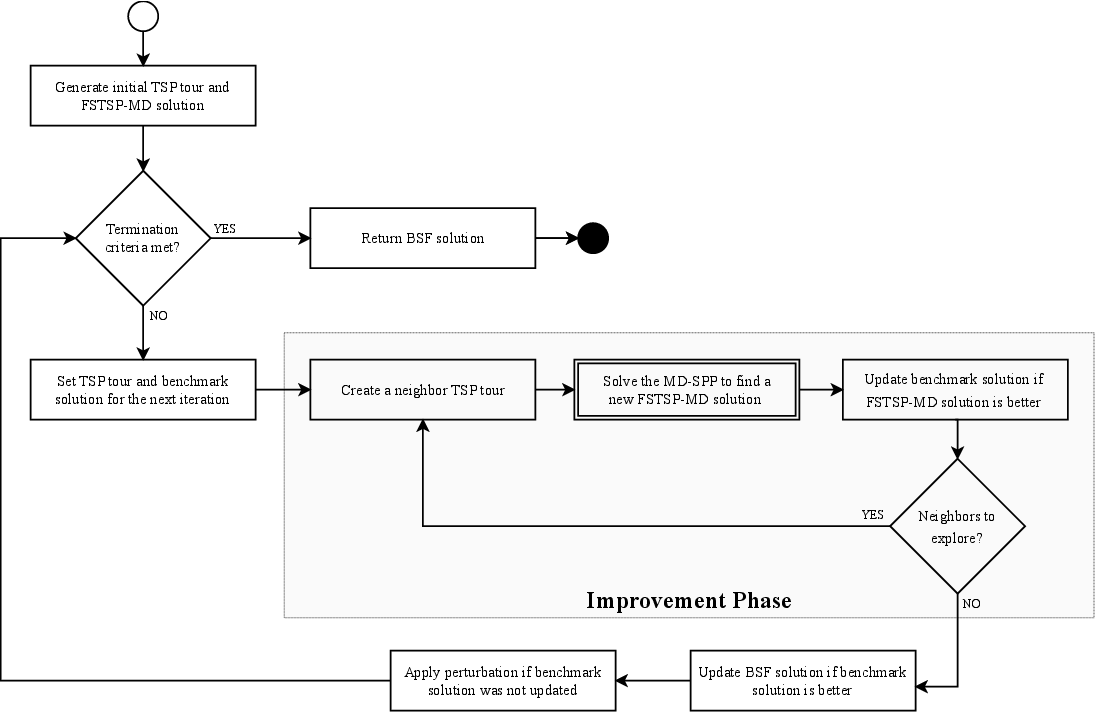}
    \caption{High-level description of the \ac{MD-SPP-H}.}
    \label{fig:heuristic}
\end{figure}

\subsection{The Multi-Drop Shortest Path Problem (Split Algorithm)}

The core component of \ac{MD-SPP-H} is the \ac{MD-SPP}, which splits a TSP tour into truck and drone routes (thus creating a feasible \ac{FSTSP-MD} solution) in polynomial time. To thoroughly explain the \ac{MD-SPP}, we start by motivating its conceptualization. We then formally define the \ac{MD-SPP} and present a polynomial-time algorithm for solving it. Throughout this section, we denote $N' = N \cup \{0\}$ as the set of nodes, where $0$ is the depot and $N = \{1,...,n\}$ is the set of customers.

\subsubsection{Motivation} \label{sec:heuristic_motivation}

The conceptualization of the \ac{MD-SPP} was inspired by the work of \cite{Kundu2021AnProblem}, who developed a split algorithm for the \ac{FSTSP} (i.e., the single-drop scenario) that finds an optimal \ac{FSTSP} solution for a given TSP tour in $\mathcal{O}(n^3)$ time. Generally speaking, this split algorithm works as follows: first, it generates a weighted directed acyclic graph by enumerating all forward-moving arc times and then finds the shortest path in the generated graph (see \ref{app:propositions} for further details). 

Extending the shortest path approach presented in \cite{Kundu2021AnProblem} for the \ac{FSTSP-MD} requires enumerating all forward-moving arc times for multiple drops. For a given TSP tour, performing this complete enumeration requires $\mathcal{O}(2^{n})$ time (see Proposition \ref{prop:complexity_extended_approach} of \ref{app:propositions}), rendering it impractical to solve realistically-sized \ac{FSTSP-MD} instances. Therefore, to formulate a polynomial-time split algorithm that can be efficiently nested in a local search scheme, we introduce a simple \emph{eligibility criterion} to enumerate only a subset of forward-moving arc times, albeit at the expense of “missing” some arc times, and thereby potential solutions that are examined when performing a complete enumeration. However, as we will discuss later, it can be argued that efficient local search and perturbation techniques can still visit those solutions that are intentionally omitted.

\subsubsection{Formulation of the \ac{MD-SPP}}

Our eligibility criterion can be interpreted as imposing an additional constraint on the problem of finding an optimal \ac{FSTSP-MD} solution for a given TSP route. Specifically, we enforce that the drone can only visit customers immediately following each potential launch node (the rest of the possible drone operations are considered ineligible). More formally, we introduce the concept of \emph{partition node}, which is defined below.

\begin{definition} \label{def:partition_node}
    (Partition node) Let $\sigma = (0,1,..., n, n+1)$ be a TSP tour whose nodes have been re-indexed based on their positions (where $0$ and $n+1$ denote the depot). Given a subsequence $\sigma'_{ij} = (i, i+1, ..., j-1, j)$ of $\sigma$, node $k \in N$ (with $i \leq k < j$) is called a \emph{partition node} if the drone can be launched at node $i \in N'$, perform the route $(i+1, i+2, ..., k-1, k)$, and be recovered at node $j \in N'$, while the truck performs the route $(i, k+1, ..., j-1, j)$. When $k = i$, the drone is carried by the truck throughout the entire subsequence.  
\end{definition}

\paragraph{Illustrative example}

Consider an instance of eight customers, as the one shown in Figure \ref{fig:example}. Further, consider the TSP tour $\sigma = (0,1,...,8,0)$ and the subsequence $\sigma'_{48} = (4,5,6,7,8)$, where Node 4 is the launch node and Node 8 is the recovery node. The enumeration of all potential truck and drone routes (between Node 4 and Node 8) is listed in Table \ref{tab:example_partition_node}. Assuming a sufficiently large drone flight endurance and payload capacity, our \ac{MD-SPP} only examines those alternatives for which a partition node exists. For illustration, Figure \ref{fig:example_partition_node} shows how the truck and drone routes are performed when Node 6 is the partition node. However, note that the omitted alternatives are examined for neighboring TSP tours of $\sigma$ listed in the last column of Table \ref{tab:example_partition_node} (these neighboring TSP tours can be readily generated by applying our local search operators over $\sigma$; see Section \ref{sec:4.3}).

\begin{table}[htbp]
  \centering
  \renewcommand{\arraystretch}{1.3}
  \caption{Illustrative example of the effects of using partition nodes for subsequence $(4,5,6,7,8)$.}
  \scalebox{0.9}{
    \begin{tabular}{llccc}
    \toprule
    \multicolumn{1}{c}{\thead{Drone route}} & \multicolumn{1}{c}{\thead{Truck route}} & \thead{Partition node} & \multicolumn{1}{c}{\thead{Examined?}} & \multicolumn{1}{c}{\thead{If not, in which TSP \\ tour is it examined?}} \\
    \midrule
    Carried by the truck & $(4,5,6,7,8)$ & 4     & \checkmark &  \\
    $(4,5,8)$ & $(4,6,7,8)$ & 5     & \checkmark &  \\
    $(4,6,8)$ & $(4,5,7,8)$ & --    & --    & $(0,1,2,3,4,6,5,7,8,0)$ \\
    $(4,7,8)$ & $(4,5,6,8)$ & --    & --    & $(0,1,2,3,4,7,5,6,8,0)$ \\
    $(4,5,6,8)$ & $(4,7,8)$ & 6     & \checkmark &  \\
    $(4,5,7,8)$ & $(4,6,8)$ & --    & --    & $(0,1,2,3,4,5,7,6,8,0)$ \\
    $(4,6,7,8)$ & $(4,5,8)$ & --    & --    & $(0,1,2,3,4,6,7,5,8,0)$ \\
    $(4,5,6,7,8)$ & $(4,8)$ & 7     & \checkmark &  \\
    \bottomrule
    \end{tabular}%
    }
   \label{tab:example_partition_node}%
\end{table}

Note that introducing a partition node allows each \ac{FSTSP-MD} solution to be generated from a unique TSP tour by construction. In contrast, without introducing partition nodes (i.e., when all potential assignments of customers to the truck and the drone are evaluated; see \ref{app:propositions}), the same \ac{FSTSP-MD} solution can be generated from different TSP tours. For example, in Table \ref{tab:example_partition_node}, the same \ac{FSTSP-MD} solution would be generated from $\sigma = (0,1,...,8,0)$ and the other TSP tours listed in the last column when performing the complete enumeration. Therefore, introducing a partition node not only allows the formulation of a polynomial-time split algorithm (see Section \ref{polytimesplit}) but also reduces redundancy (in terms of avoiding inspecting the same \ac{FSTSP-MD} solution multiple times when exploring the neighborhood of the current TSP tour).

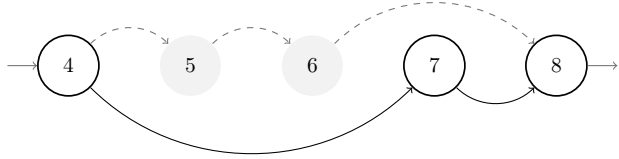
\begin{figure} [H]
    \centering
    \resizebox{0.5\textwidth}{!}{
	\begin{tikzpicture} 
		\begin{scope} [every node/.style={circle, thick, draw, minimum size=1cm}]
		\node  (n4) at (0,0)  {4};
		\node  [draw=none] [fill=A] (n5) at (2, 0)  {5};
            \node  [draw=none] [fill=A] (n6) at (4,0)  {6}; 
		\node  (n7) at (6,0)  {7};
  		\node  (n8) at (8,0)  {8};
		\end{scope}
  
        \node [invisible] (n3) at (-1,0) {};
        \node [invisible] (n9) at (9,0) {};
	
		% Truck:
            \draw [->, line width=0.5pt] (n4) to [out=-45, in=-135] (n7);
            \draw [->, line width=0.5pt] (n7) to [out=-45, in=-135] (n8);

  		% Truck:
		\foreach \from/\to in {n3/n4, n8/n9}
		\draw [->, F, line width=0.5pt] (\from) -- (\to);

  		% Drone:
    	\draw [->, dashed, F, line width=0.5pt] (n4) to [out=45, in=135] (n5);
            \draw [->, dashed, F, line width=0.5pt] (n5) to [out=45, in=135] (n6);
            \draw [->, dashed, F, line width=0.5pt] (n6) to [out=45, in=135] (n8);

    \end{tikzpicture}
	}
	%\vspace{0.5cm}
	\caption{Illustrative example of subsequence $(4,5,6,7,8)$ partitioned at Node 6. Dashed lines denote the drone route and continuous lines the truck route.}
	\label{fig:example_partition_node}
\end{figure}

Now, we formally define the \ac{MD-SPP} in Problem \ref{prob:SPP}. Using the illustrative example of Section \ref{sec:problem_definition}, if the underlying TSP tour is $\sigma = (0,1,...,8,0)$, the collection of tuples defined in Problem \ref{prob:SPP} is given by $P=\{((0,1,2,3), 1), ((3,4),3),((4,5,6,7,8), 7), ((8,0),8)\}$, which represents the feasible \ac{FSTSP-MD} solution shown in Figure \ref{fig:example}. 

\begin{problem} \label{prob:SPP}
    (The Multi-Drop Shortest Path Problem) Let $\sigma = (0,1,..., n, n+1)$ be a TSP tour whose nodes have been re-indexed based on their positions. The \ac{MD-SPP} aims to find a collection of tuples $P = \{(\sigma'_{i_0,i_1}, k_0), (\sigma'_{i_1,i_2}, k_1),..., (\sigma'_{i_{j},i_{j+1}}, k_j),...,(\sigma'_{i_{\ell},i_{\ell+1}}, k_\ell)\}$ such that (i) $k_j$ is a partition node of subsequence $\sigma'_{i_{j},i_{j+1}}$, for all $j \in \{0,...,\ell\}$, (ii) the subsequences of $P$ form the TSP tour, and (iii) the completion time is minimized. 
\end{problem}

\subsubsection{A Polynomial-Time Algorithm for Solving the \ac{MD-SPP}}
\label{polytimesplit}

Based on the notation listed in \ref{tab:notation}, Algorithm \ref{alg:1} presents a $\mathcal{O}(n^3)$--time procedure to find the optimal value of Problem \ref{prob:SPP} (i.e., the completion time at node $n+1$). The algorithm consists of three nested loops. The first (Step 2) and third loop (Step 8) of Algorithm \ref{alg:1} sequentially consider all pairs $(i,j)$ as the launch and recovery node for a drone operation, while the second loop (Step 6) iteratively increases the partition node while ensuring the operation's feasibility (i.e., respecting the maximum number of drops and the drone flight endurance). 
For each partition node $k$, the coordinated time taken by the truck-and-drone tandem to travel from the launch node $i$ to the recovery node $j$, $c_{i,j}^{k}$, is calculated in Step 14, and for each feasible operation, the completion time at the recovery node $j$ is updated using the relaxation steps in Steps 5 and 16 (particularly, Step 5 considers an operation where the truck traverses from a node to its adjacent node, carrying the drone with it).

\begin{table}[H]
    \centering
    \small
    \caption{Notation used in Algorithm \ref{alg:1}.}
    \renewcommand{\arraystretch}{1.1}
  \resizebox{\textwidth}{!}{%
    \begin{tabular}{ll}    \toprule
    Symbol &  Description \\
    \midrule
    $n$ & Number of customers.\\
    $\sigma$ & Any given TSP tour (with nodes re-indexed based on their positions).\\
    $\sigma'_{ij}$ & Any given subsequence of a TSP tour (with nodes re-indexed based on their positions). \\
    $T_{i}$ & Completion time at node $i$, with $i \in \{0,..., n+1\}$. \\
    $t_{i,j}$ & Truck travel time from node $i$ to $j$. \\
    $d_{i,j}$ & Drone travel time from node $i$ to $j$. \\
    $c_{i,j}^k$ & Coordinated time taken by the truck-and-drone tandem to travel from node $i$ to $j$, with partition node $k$. \\
    $D$ & Maximum number of drops the drone can perform between launch and recovery. \\
    $E$ & Maximum drone flight endurance (in time units). \\
    $a_i$, $\alpha_\text{D}$, $\alpha_\text{T}$ & Auxiliary variables to calculate arc times iteratively. \\
    $[i,j]$ & Set of integers between $i$ and $j$ (both included), i.e., $[i,j] = \{i,...,j\}$. \\ 
    \bottomrule
    \end{tabular}
    }
 \label{tab:notation}
\end{table}

\begin{algorithm}[htbp]
	\SetAlgoVlined
	\DontPrintSemicolon
	\setcounter{AlgoLine}{0}
	\SetKwInOut{Input}{Input}
	\SetKwInOut{Output}{Output}
	\Input{TSP tour $\sigma = (0,1,..., n, n+1)$ with nodes re-indexed based on their positions.}
	\Output{Completion time, i.e., $T_{n+1}$.}
        $T_0 \gets 0$ and $T_i \gets \infty, \forall i \in [1,n+1]$\;
        \For{$i \in [0,n]$}{
            $a_i \gets  0$\;
            \If{$T_{i}+t_{i,i+1}$ < $T_{i+1}$}{
                $T_{i+1} \gets T_{i}+t_{i,i+1}$ \Comment{Relaxation step}\; 
            }        
            \For{$k \in [i + 1, \min\{i+D,n\}]$}{
                 $a_i \gets  a_i + d_{k-1,k}$\;
                 \For{$j \in [k + 1,n+1]$}{
                    $\alpha_\text{D} \gets a_i + d_{k,j}$\;
                    \If{$j = k+1$}{
                        $\alpha_\text{T} \gets t_{i,j}$\;
                        }
                    \Else{
                        $\alpha_\text{T} \gets \alpha_\text{T} + t_{j-1,j}$\;
                    } 
                    $c_{i,j}^k \gets \max\{\alpha_\text{T}, \alpha_\text{D}\}$\;
                    \If{$T_{i} + c_{i,j}^k < T_{j}$ \textsc{\bf and}  $c_{i,j}^k \leq E $}{
                        $T_{j} \gets T_{i} + c_{i,j}^k$  \Comment{Relaxation step}\;}
                 }
            }
        }
	\caption{Algorithm to solve Problem \ref{prob:SPP}.}
	\label{alg:1}
\end{algorithm}

Mathematically, for a given TSP tour $\sigma$, note that Algorithm \ref{alg:1} solves the \ac{MD-SPP} by sequentially computing arc times $c_{i,j}$ for each subsequence $\sigma'_{ij} = (i,...,j)$ through Equation \eqref{c_ijk} and Equation \eqref{c_ij}:
\begin{align} 
c_{i,j}^k & =  
\begin{cases}
    \displaystyle \sum_{m=i}^{j-1} t_{m,m+1},& \text{if } k=i,\\
    \max\left\{\displaystyle \sum_{m=i}^{k-1} d_{m,m+1} + d_{k,j} \>,\> t_{i,j}\right\},              & \text{if } k>i \land j=k+1,\\
    \max\left\{\displaystyle \sum_{m=i}^{k-1} d_{m,m+1} + d_{k,j} \>,\> t_{i,k+1} + \sum_{m=k+1}^{j-1} t_{m,m+1}\right\},              & \text{otherwise},
\end{cases}  \label{c_ijk} \\
c_{i,j} & =  \min \{c_{i,j}^k: i \leq k <j \land k \leq i + D \land c_{i,j}^k \leq E \}, \label{c_ij}
\end{align}
where $c_{i,j}^k$ represents the arc time if node $k$, with $i \leq k <j$, partitions subsequence $\sigma'_{ij}$ (note that computing $c_{i,j}$ requires evaluating all possible partition nodes of subsequence $\sigma'_{ij}$).

\begin{proposition} \label{proposition}
    For a given TSP tour, Algorithm \ref{alg:1} solves Problem \ref{prob:SPP} in $\mathcal{O}(n^3)$ time.
\end{proposition}

\begin{proof}[Proof]
    See proof in \ref{app:propositions}.
\end{proof}

\subsection{Exploration of the TSP Solution Space} \label{sec:4.3}

This section describes how \ac{MD-SPP-H} explores the TSP solution space. We start by describing the procedures for generating the initial solution and the neighborhoods. We then present the perturbation mechanisms that allow \ac{MD-SPP-H} to escape local optima and explore unexplored regions of the solution space.

\paragraph{Generation of the initial solution} 

To create the initial TSP tour, we use the Concorde solver of \cite{Applegate2006TheStudy} (using the QSopt callable library) to obtain an optimal TSP tour. The initial \ac{FSTSP-MD} solution is then generated by solving the \ac{MD-SPP}. This \ac{FSTSP-MD} solution is designated as the BSF solution and the benchmark solution for the first iteration of \ac{MD-SPP-H}.

\paragraph{Neighborhood generation}

In the improvement phase, we use a composite neighborhood obtained by sequentially applying a set of neighborhood operators to the current TSP tour. Specifically, similar to \cite{Agatz2018OptimizationDrone}, we first explore the neighborhood of the current TSP tour generated by the one-point (1-p) operator (also known as relocate operator; see \cite{Savelsbergh1992TheDuration}), where a customer is relocated to another position within the current TSP tour; then the neighborhood generated by the two-point (2-p) operator, where two customers are swapped; and finally, the neighborhood generated by the 2-opt operator, where two arcs are removed and replaced with two new ones. Note that the entire exploration of each of these neighborhoods requires $\mathcal{O}(n^2)$ time. While exploring these neighborhoods, the benchmark solution is updated whenever a new FSTSP-MD solution with a better objective value is found.

After exploring the entire composite neighborhood of the current TSP tour, we update the BSF solution with the benchmark solution if the latter is better. Further, we maintain the benchmark solution for the next iteration if it was updated during the improvement phase (in this case, the TSP tour of the benchmark solution is also set as the TSP tour for the next iteration). Otherwise, we set the benchmark solution and the TSP tour for the next iteration according to the perturbation phase.

\paragraph{Perturbation phase} 

A critical component in designing an effective heuristic approach involves incorporating diversification (or perturbation) mechanisms to guide the search into previously unexplored regions within the solution space. As previously mentioned, we enter the perturbation phase if the benchmark solution is not updated during the improvement phase.

Depending on the number of iterations without improvement, we enter a small or big perturbation phase. 
Specifically, whenever we perform a small perturbation $\eta$ times without yielding any improvement (after subsequent neighborhood search), we perform a big perturbation. 
For the small perturbation to the current TSP tour, we randomly select two subsequences from the tour, reverse their orders, and reintegrate them into their original positions (two sequential 2-opt moves). 
For the big perturbation, we first restart the search from the TSP tour of the BSF solution. We then perform the aforementioned reversal of two subsequences (i.e., our small perturbation), followed by a mutation operator that introduces random swaps within the two subsequences. Here, each position within the subsequences has a probability $p$ of being swapped with another random position within the same subsequence.

After modifying the current TSP tour either by the small or big perturbation mechanisms, we solve the \ac{MD-SPP}, and the corresponding \ac{FSTSP-MD} solution is set as the benchmark solution for the next iteration.

\color{black}

\section{Heuristic Performance} \label{sec:heuristic_performance}

This section presents extensive computational experiments to show the performance of \ac{MD-SPP-H}. \ac{MD-SPP-H} was implemented in Java 8 on a scientific HPC cluster. The computational nodes have Intel Xeon CPU E3-1284L v4 processors, each featuring four cores (1.19 GHz nominal, 3.8 GHz peak). The nodes operated on an \text{x86\_64} architecture. Based on preliminary experiments, we set $\eta = 10$ and $p = 0.1$ in the perturbation phase of \ac{MD-SPP-H} (see Section \ref{sec:4.3}). Further, as termination criteria, we stop \ac{MD-SPP-H} after 200 consecutive iterations without improvement (a number determined through preliminary experiments) or after a given time limit (described later), whichever comes first. For the interested reader, we provide all problem instances used and detailed results in our GitHub repository at \url{https://github.com/schasarah/FSTSP-MD-Experiment-Results.git}.

We focus on realistically sized instances with 50 or more customers (see, e.g., \cite{Holland2017UPSRoutes} and \cite{Merchan20242021Set}). Due to the inherent complexity of these instances, comparing \ac{MD-SPP-H} with exact solution approaches is not possible (see Section \ref{sec:literature}). Instead, we conduct a pairwise comparative analysis against other state-of-the-art heuristics developed for the \ac{FSTSP-MD} and the \ac{FSTSP} (i.e., the single-drop scenario). Despite \ac{MD-SPP-H} being tailored for multiple drops per drone sortie, we also evaluate its performance on the \ac{FSTSP} for a comprehensive understanding.

For benchmark purposes, we consider the following solution approaches representative of the current academic state-of-the-art:
(i) the \ac{IGH} of \cite{Gonzalez-R2020Truck-dronePlanning} for the \ac{FSTSP-MD};
(ii) the \ac{ALNS} of \cite{WindrasMara2022AnDrops} for the \ac{FSTSP-MD};
(iii) the \ac{EP-All} heuristic of \cite{Agatz2018OptimizationDrone} for the \ac{FSTSP}; and
(iv) the \ac{SPP-All} heuristic of \cite{Kundu2021AnProblem} for the \ac{FSTSP}.

To enable a fair and consistent comparison, we adhere to the maximum runtimes allowed and the number of independent runs specified in each benchmark study. We compute two main metrics for our comparison. First, the relative percentage difference between the \emph{best} objective values found by each heuristic ($\Delta_\text{best}$) is given by
\begin{equation*}
    \Delta_\text{best} = \frac{z^\text{bench}_\text{best} - z^\text{MD-SPP-H}_\text{best}}{z^\text{bench}_\text{best}} \times 100,
\end{equation*}
where $z^\text{bench}_\text{best}$ and $z^\text{MD-SPP-H}_\text{best}$ are the best values reached by the benchmark and \ac{MD-SPP-H} heuristics (across all runs), respectively. Second, the relative percentage difference between the \emph{average} objective values found by each heuristic ($\Delta_\text{avg}$) is given by
\begin{equation*}
    \Delta_\text{avg} = \frac{z^\text{bench}_\text{avg} - z^\text{MD-SPP-H}_\text{avg}}{z^\text{bench}_\text{avg}} \times 100,
\end{equation*}
where $z^\text{bench}_\text{avg}$ and $z^\text{MD-SPP-H}_\text{avg}$ are the average values reached by the benchmark and \ac{MD-SPP-H} heuristics (across all runs), respectively.

Note that \ac{EP-All} and \ac{SPP-All} are deterministic; therefore, these heuristics are executed only once (thus, we simply use $z$ to denote the objective value reached by these heuristics). However, we ran \ac{MD-SPP-H} ten independent times to ensure the robustness of our results. We adjust our aforementioned metrics accordingly and compare both our best and average results to the single solution value reached by the respective method. Finally, in \ref{app:supplementary_results_performance}, we also include an additional comparison with the exact approach of \cite{Vasquez2021AnDecomposition} over small \ac{FSTSP} instances with up to 20 customers.

\subsection{Comparison with state-of-the-art solution approaches for the FSTSP-MD}
\label{sec:benchmarkFSTSPMD}

This section begins by comparing the performance of \ac{MD-SPP-H} with the seminal \ac{IGH} of \cite{Gonzalez-R2020Truck-dronePlanning}. We then extend our analysis to compare \ac{MD-SPP-H} with the recent \ac{ALNS} heuristic proposed by \cite{WindrasMara2022AnDrops}.

\paragraph{Comparison with the \ac{IGH} of \cite{Gonzalez-R2020Truck-dronePlanning}}

To test their heuristic, \cite{Gonzalez-R2020Truck-dronePlanning} use the problem instances originally introduced by \cite{Agatz2018OptimizationDrone}.
For our comparison, we select instances from the dataset with 50 up to 250 customer locations. The spatial distribution of these locations follows three distinct patterns: uniform, single-centered, and double-centered. The benchmark set contains 10 instances for each node count and distribution pattern. Travel times for both the truck and drone are based on Euclidean distances. Service time, as well as launch and recovery duration, are not considered. Lastly, all customers can be either served by the truck or the drone (i.e., all customers are drone-eligible). 

As described in Section \ref{sec:literature}, \cite{Gonzalez-R2020Truck-dronePlanning} assume the drone can perform as many drops as its flight endurance allows (i.e., no constraints limit the total number of separate packages the drone can deliver per sortie). In our \ac{MD-SPP-H}, this assumption can be readily incorporated by removing the maximum number of drone drops constraint in Algorithm \ref{alg:1} (or, equivalently, by setting $D=n$). The maximum flight endurance is assumed proportional to the average drone travel time between the nodes, represented as $E = \frac{2}{n\cdot(n+1)}\sum_{(i,j)\in A} d_{i,j}$ (where $A$ denotes the set of arcs between the nodes). To ensure a precise comparison, we use the same maximum flight endurance values as \cite{Gonzalez-R2020Truck-dronePlanning}. Finally, \cite{Gonzalez-R2020Truck-dronePlanning} impose a maximum time limit of 600 seconds for \ac{IGH}, which we also use as a termination criterion for \ac{MD-SPP-H}. Table \ref{tab:ExperimentSetting} in \ref{app:supplementary_results_performance} summarizes the experiment setting. %We also refer the reader to our GitHub repository for details on the actual runtimes required by \ac{MD-SPP-H} and \ac{IGH}. 

Before presenting the comparison between both heuristics, it is worth mentioning that \ac{IGH} assumes that the truck-and-drone delivery system starts at the depot but ends at the last customer listed in \cite{Agatz2018OptimizationDrone}'s instances. Consequentially, \ac{IGH} does not consider the time for the vehicles to return to the depot after serving the last customer listed. In contrast, \ac{MD-SPP-H} aims to minimize the completion time of the delivery process (i.e., the time when the last vehicle returns to the depot), consistent with most drone logistics studies (see, e.g., \cite{Chung2020OptimizationDirections}). 
Therefore, the comparison presented below is slightly biased since \ac{MD-SPP-H} and \ac{IGH} do not share exactly the same problem definition. 
However, the relative difference between both problems diminishes as the number of customers increases. With more customers, the total completion time increases, so variations in the final arcs of the routes constructed by the two heuristics have less impact. %Consequently, we only compare the heuristics for comparatively large instances with 50 or more customers.

Table \ref{tab:TabBenchmarkGonazelezUniformFull} presents the results for instances with uniformly distributed locations, considering drone speed ratios of 1, 2, and 3 (i.e., the ratio of the drone speed to the truck speed). Results show that \ac{MD-SPP-H} outperforms \ac{IGH} in terms of solution quality. Overall, \ac{MD-SPP-H} improves the best solution values found in 144 out of 150 instances, with $\Delta_\text{best}$ of 25.3\%, 12.5\%, and 7.9\% averaged over all instances with a drone speed ratios of 1, 2, and 3, respectively, and $\Delta_\text{avg}$ of 27.4\%, 12.5\%, and 7.3\%. The most significant relative enhancements occur when the drone operates at slower speeds (i.e., when its speed is the same as that of the truck). Based on the quality measures reported in \cite{Gonzalez-R2020Truck-dronePlanning}, we observe that \ac{IGH} seems to lead to more consistent results as drone speed increases, regardless of customer distribution. 

We refer the reader to Tables \ref{tab:TabBenchmarkGonazelezSingleCenterFull} and \ref{tab:TabBenchmarkGonazelezDoubleCenterFull} for results on single- and double-center instances. Results show that \ac{MD-SPP-H} outperforms \ac{IGH} regardless of the customer distribution. However, we observe that the percentage improvements for the non-uniform distributions are slightly smaller. For example, for a drone speed ratio of 1, the value of $\Delta_\text{best}$ averaged over all instances with uniform distribution is 25.3\%. In contrast, for non-uniform distributions, the average $\Delta_\text{best}$ values are lower, at 16.7\% and 21.3\% for single- and double-center distribution, respectively. 

\begin{table}[htbp]
  \centering
\caption{Results of the comparison of \ac{MD-SPP-H} and \ac{IGH} (based on the solution published by \citet{Gonzalez-R2020Truck-dronePlanning}), considering $n=50$ up to $250$ uniformly distributed customers.}
    \renewcommand{\arraystretch}{1.1}
        \resizebox{\textwidth}{!}{
    \begin{tabular}{ccrrrrrrrrrrrrrrrrrrrrrrrrrrr}
    \toprule
    \multirow{3}[6]{*}{$n$} & \multirow{3}[6]{*}{ID} &       & \multicolumn{8}{c}{Speed ratio 1}                             &       & \multicolumn{8}{c}{Speed ratio 2}                             &       & \multicolumn{8}{c}{Speed ratio 3} \\
\cmidrule{4-11}\cmidrule{13-20}\cmidrule{22-29}          &       &       & \multicolumn{2}{c}{MD-SPP-H} &       & \multicolumn{2}{c}{IGH} &       & \multicolumn{2}{c}{$\Delta$} &       & \multicolumn{2}{c}{MD-SPP-H} &       & \multicolumn{2}{c}{IGH} &       & \multicolumn{2}{c}{$\Delta$} &       & \multicolumn{2}{c}{MD-SPP-H} &       & \multicolumn{2}{c}{IGH} &       & \multicolumn{2}{c}{$\Delta$} \\
\cmidrule{4-5}\cmidrule{7-8}\cmidrule{10-11}\cmidrule{13-14}\cmidrule{16-17}\cmidrule{19-20}\cmidrule{22-23}\cmidrule{25-26}\cmidrule{28-29}          &       &       & \multicolumn{1}{c}{$z_\text{best}$} & \multicolumn{1}{c}{$z_\text{avg}$} &       & \multicolumn{1}{c}{$z_\text{best}$} & \multicolumn{1}{c}{$z_\text{avg}$} &       & \multicolumn{1}{c}{$\Delta_\text{best}$} & \multicolumn{1}{c}{$\Delta_\text{avg}$} &       & \multicolumn{1}{c}{$z_\text{best}$} & \multicolumn{1}{c}{$z_\text{avg}$} &       & \multicolumn{1}{c}{$z_\text{best}$} & \multicolumn{1}{c}{$z_\text{avg}$} &       & \multicolumn{1}{c}{$\Delta_\text{best}$} & \multicolumn{1}{c}{$\Delta_\text{avg}$} &       & \multicolumn{1}{c}{$z_\text{best}$} & \multicolumn{1}{c}{$z_\text{avg}$} &       & \multicolumn{1}{c}{$z_\text{best}$} & \multicolumn{1}{c}{$z_\text{avg}$} &       & \multicolumn{1}{c}{$\Delta_\text{best}$} & \multicolumn{1}{c}{$\Delta_\text{avg}$} \\
    \midrule
    \multirow{10}[2]{*}{50} & 71    &       & 405.8 & 407.0 &       & 458.6 & 502.0 &       & 11.5  & 18.9  &       & 281.5 & 282.7 &       & 312.1 & 325.3 &       & 9.8   & 13.1  &       & 246.0 & 251.5 &       & 270.8 & 275.4 &       & 9.1   & 8.7 \\
          & 72    &       & 409.7 & 417.8 &       & 495.9 & 521.9 &       & 17.4  & 19.9  &       & 296.0 & 302.5 &       & 316.0 & 316.0 &       & 6.3   & 4.3   &       & 249.2 & 254.6 &       & 268.4 & 268.7 &       & 7.1   & 5.2 \\
          & 73    &       & 403.8 & 406.7 &       & 505.3 & 522.2 &       & 20.1  & 22.1  &       & 290.6 & 294.2 &       & 316.5 & 322.5 &       & 8.2   & 8.8   &       & 243.5 & 245.0 &       & 258.4 & 262.3 &       & 5.8   & 6.6 \\
          & 74    &       & 415.3 & 419.8 &       & 474.9 & 495.0 &       & 12.6  & 15.2  &       & 287.8 & 289.0 &       & 291.0 & 291.7 &       & 1.1   & 0.9   &       & 251.7 & 254.1 &       & 232.8 & 233.5 &       & -8.1  & -8.8 \\
          & 75    &       & 409.7 & 414.4 &       & 518.5 & 543.1 &       & 21.0  & 23.7  &       & 291.5 & 298.0 &       & 314.6 & 320.7 &       & 7.4   & 7.1   &       & 256.1 & 259.5 &       & 248.7 & 252.5 &       & -2.9  & -2.8 \\
          & 76    &       & 389.8 & 392.5 &       & 475.1 & 506.4 &       & 17.9  & 22.5  &       & 278.5 & 285.0 &       & 296.6 & 306.6 &       & 6.1   & 7.0   &       & 249.4 & 250.5 &       & 238.4 & 243.4 &       & -4.6  & -2.9 \\
          & 77    &       & 416.0 & 423.6 &       & 472.6 & 505.6 &       & 12.0  & 16.2  &       & 295.3 & 301.7 &       & 311.6 & 317.2 &       & 5.2   & 4.9   &       & 253.0 & 261.1 &       & 266.7 & 266.7 &       & 5.1   & 2.1 \\
          & 78    &       & 400.7 & 402.0 &       & 466.4 & 491.4 &       & 14.1  & 18.2  &       & 275.9 & 278.4 &       & 290.9 & 291.3 &       & 5.2   & 4.4   &       & 232.2 & 235.7 &       & 238.3 & 239.2 &       & 2.6   & 1.5 \\
          & 79    &       & 370.4 & 376.2 &       & 530.1 & 539.2 &       & 30.1  & 30.2  &       & 269.1 & 273.8 &       & 311.2 & 311.7 &       & 13.5  & 12.2  &       & 235.3 & 237.1 &       & 242.8 & 244.2 &       & 3.1   & 2.9 \\
          & 80    &       & 379.5 & 385.1 &       & 437.6 & 467.5 &       & 13.3  & 17.6  &       & 269.3 & 271.8 &       & 281.9 & 286.5 &       & 4.5   & 5.1   &       & 228.1 & 228.1 &       & 223.1 & 223.1 &       & -2.3  & -2.3 \\
    \midrule
    \multirow{10}[2]{*}{75} & 81    &       & 453.1 & 459.8 &       & 566.2 & 589.4 &       & 20.0  & 22.0  &       & 311.9 & 321.7 &       & 365.7 & 368.7 &       & 14.7  & 12.7  &       & 258.5 & 265.9 &       & 296.1 & 297.0 &       & 12.7  & 10.5 \\
          & 82    &       & 416.8 & 429.9 &       & 553.6 & 567.7 &       & 24.7  & 24.3  &       & 291.9 & 303.6 &       & 317.7 & 321.8 &       & 8.1   & 5.7   &       & 251.4 & 254.4 &       & 253.7 & 253.7 &       & 0.9   & -0.3 \\
          & 83    &       & 420.1 & 423.8 &       & 549.4 & 567.0 &       & 23.5  & 25.3  &       & 291.8 & 296.4 &       & 319.4 & 325.6 &       & 8.6   & 9.0   &       & 241.3 & 243.5 &       & 238.7 & 240.2 &       & -1.1  & -1.4 \\
          & 84    &       & 446.0 & 454.5 &       & 611.3 & 666.3 &       & 27.1  & 31.8  &       & 316.8 & 321.8 &       & 381.5 & 388.4 &       & 16.9  & 17.1  &       & 256.9 & 261.6 &       & 310.1 & 319.2 &       & 17.1  & 18.0 \\
          & 85    &       & 475.9 & 478.5 &       & 616.4 & 651.2 &       & 22.8  & 26.5  &       & 338.3 & 344.9 &       & 389.8 & 395.7 &       & 13.2  & 12.8  &       & 280.4 & 286.1 &       & 302.5 & 305.4 &       & 7.3   & 6.3 \\
          & 86    &       & 452.3 & 468.5 &       & 601.6 & 616.9 &       & 24.8  & 24.1  &       & 314.2 & 326.0 &       & 368.2 & 371.2 &       & 14.7  & 12.2  &       & 263.3 & 269.6 &       & 305.5 & 309.5 &       & 13.8  & 12.9 \\
          & 87    &       & 453.7 & 463.3 &       & 573.0 & 605.9 &       & 20.8  & 23.5  &       & 321.6 & 328.8 &       & 362.4 & 376.1 &       & 11.3  & 12.6  &       & 271.7 & 274.7 &       & 284.9 & 287.5 &       & 4.6   & 4.4 \\
          & 88    &       & 451.5 & 458.2 &       & 627.3 & 644.8 &       & 28.0  & 28.9  &       & 320.4 & 330.5 &       & 351.4 & 360.7 &       & 8.8   & 8.4   &       & 261.7 & 268.2 &       & 283.9 & 286.5 &       & 7.8   & 6.4 \\
          & 89    &       & 406.9 & 415.6 &       & 584.4 & 611.3 &       & 30.4  & 32.0  &       & 292.7 & 298.2 &       & 345.3 & 348.7 &       & 15.2  & 14.5  &       & 250.9 & 255.2 &       & 275.3 & 275.5 &       & 8.9   & 7.4 \\
          & 90    &       & 460.6 & 469.3 &       & 558.6 & 582.3 &       & 17.5  & 19.4  &       & 319.8 & 322.5 &       & 325.7 & 336.1 &       & 1.8   & 4.1   &       & 272.5 & 276.5 &       & 240.9 & 244.5 &       & -13.1 & -13.1 \\
    \midrule
    \multirow{10}[2]{*}{100} & 91    &       & 512.4 & 529.5 &       & 676.3 & 701.3 &       & 24.2  & 24.5  &       & 346.3 & 359.4 &       & 395.2 & 403.8 &       & 12.4  & 11.0  &       & 287.5 & 296.5 &       & 302.7 & 307.8 &       & 5.0   & 3.7 \\
          & 92    &       & 484.9 & 493.1 &       & 652.3 & 674.1 &       & 25.7  & 26.9  &       & 337.8 & 345.4 &       & 396.8 & 409.8 &       & 14.9  & 15.7  &       & 276.2 & 279.5 &       & 304.6 & 311.0 &       & 9.3   & 10.1 \\
          & 93    &       & 483.5 & 495.6 &       & 652.5 & 697.2 &       & 25.9  & 28.9  &       & 325.3 & 339.6 &       & 388.4 & 394.9 &       & 16.2  & 14.0  &       & 267.3 & 273.0 &       & 303.3 & 306.3 &       & 11.9  & 10.9 \\
          & 94    &       & 480.6 & 496.4 &       & 639.7 & 667.2 &       & 24.9  & 25.6  &       & 335.8 & 344.9 &       & 375.3 & 382.0 &       & 10.5  & 9.7   &       & 283.5 & 285.0 &       & 307.9 & 308.1 &       & 7.9   & 7.5 \\
          & 95    &       & 510.0 & 521.2 &       & 691.2 & 739.3 &       & 26.2  & 29.5  &       & 333.8 & 352.0 &       & 416.0 & 421.9 &       & 19.8  & 16.6  &       & 274.9 & 282.2 &       & 299.1 & 301.1 &       & 8.1   & 6.3 \\
          & 96    &       & 507.7 & 511.0 &       & 726.7 & 755.2 &       & 30.1  & 32.3  &       & 347.5 & 353.3 &       & 417.7 & 423.5 &       & 16.8  & 16.6  &       & 291.3 & 292.6 &       & 340.9 & 343.6 &       & 14.5  & 14.8 \\
          & 97    &       & 495.8 & 510.7 &       & 651.1 & 694.6 &       & 23.8  & 26.5  &       & 347.6 & 356.8 &       & 377.5 & 387.4 &       & 7.9   & 7.9   &       & 276.7 & 285.0 &       & 287.0 & 291.5 &       & 3.6   & 2.2 \\
          & 98    &       & 502.8 & 509.6 &       & 686.7 & 699.9 &       & 26.8  & 27.2  &       & 345.2 & 353.0 &       & 415.8 & 419.4 &       & 17.0  & 15.8  &       & 278.9 & 285.4 &       & 354.4 & 354.6 &       & 21.3  & 19.5 \\
          & 99    &       & 494.8 & 498.0 &       & 703.5 & 731.2 &       & 29.7  & 31.9  &       & 341.7 & 350.5 &       & 412.8 & 424.1 &       & 17.2  & 17.4  &       & 275.5 & 282.7 &       & 336.3 & 344.3 &       & 18.1  & 17.9 \\
          & 100   &       & 490.6 & 502.7 &       & 668.6 & 689.2 &       & 26.6  & 27.1  &       & 337.9 & 356.5 &       & 391.8 & 397.4 &       & 13.8  & 10.3  &       & 276.0 & 279.7 &       & 321.7 & 321.9 &       & 14.2  & 13.1 \\
    \midrule
    \multirow{10}[2]{*}{175} & 101   &       & 630.3 & 634.7 &       & 888.1 & 908.3 &       & 29.0  & 30.1  &       & 429.5 & 432.3 &       & 499.7 & 511.6 &       & 14.0  & 15.5  &       & 339.0 & 347.2 &       & 405.3 & 408.3 &       & 16.4  & 15.0 \\
          & 102   &       & 623.7 & 626.3 &       & 842.7 & 912.5 &       & 26.0  & 31.4  &       & 422.5 & 425.7 &       & 520.3 & 530.7 &       & 18.8  & 19.8  &       & 327.6 & 331.0 &       & 384.3 & 390.1 &       & 14.8  & 15.1 \\
          & 103   &       & 644.8 & 652.4 &       & 868.4 & 893.3 &       & 25.7  & 27.0  &       & 434.6 & 445.2 &       & 488.2 & 499.3 &       & 11.0  & 10.9  &       & 345.2 & 359.6 &       & 359.9 & 362.4 &       & 4.1   & 0.8 \\
          & 104   &       & 630.2 & 638.7 &       & 868.4 & 895.0 &       & 27.4  & 28.6  &       & 436.1 & 436.2 &       & 493.1 & 500.4 &       & 11.6  & 12.8  &       & 341.5 & 343.9 &       & 375.0 & 376.5 &       & 8.9   & 8.6 \\
          & 105   &       & 639.1 & 642.3 &       & 882.5 & 922.7 &       & 27.6  & 30.4  &       & 439.2 & 441.4 &       & 517.7 & 523.8 &       & 15.2  & 15.7  &       & 335.5 & 336.6 &       & 363.3 & 365.1 &       & 7.7   & 7.8 \\
          & 106   &       & 651.2 & 669.4 &       & 870.1 & 891.7 &       & 25.2  & 24.9  &       & 439.6 & 446.8 &       & 473.9 & 480.7 &       & 7.2   & 7.1   &       & 349.1 & 357.3 &       & 353.9 & 354.3 &       & 1.4   & -0.8 \\
          & 107   &       & 637.8 & 641.2 &       & 892.5 & 914.5 &       & 28.5  & 29.9  &       & 430.7 & 447.5 &       & 514.1 & 519.8 &       & 16.2  & 13.9  &       & 345.4 & 350.9 &       & 386.6 & 391.9 &       & 10.7  & 10.4 \\
          & 108   &       & 658.0 & 659.3 &       & 923.8 & 949.7 &       & 28.8  & 30.6  &       & 450.1 & 458.5 &       & 522.5 & 534.7 &       & 13.8  & 14.2  &       & 341.6 & 358.9 &       & 370.8 & 373.8 &       & 7.9   & 4.0 \\
          & 109   &       & 636.6 & 643.4 &       & 902.9 & 920.4 &       & 29.5  & 30.1  &       & 428.5 & 435.2 &       & 494.6 & 505.7 &       & 13.4  & 13.9  &       & 331.8 & 335.8 &       & 348.5 & 351.7 &       & 4.8   & 4.5 \\
          & 110   &       & 613.8 & 616.7 &       & 908.1 & 932.0 &       & 32.4  & 33.8  &       & 414.4 & 418.4 &       & 488.4 & 508.6 &       & 15.2  & 17.7  &       & 323.5 & 327.7 &       & 385.3 & 387.8 &       & 16.0  & 15.5 \\
    \midrule
    \multirow{10}[2]{*}{200} & 111   &       & 710.7 & 714.2 &       & 1052.4 & 1091.1 &       & 32.5  & 34.5  &       & 493.6 & 494.9 &       & 593.9 & 607.0 &       & 16.9  & 18.5  &       & 368.7 & 369.5 &       & 438.4 & 442.8 &       & 15.9  & 16.6 \\
          & 112   &       & 732.0 & 733.9 &       & 1047.3 & 1080.6 &       & 30.1  & 32.1  &       & 492.1 & 494.4 &       & 584.1 & 600.3 &       & 15.7  & 17.6  &       & 376.8 & 382.1 &       & 418.8 & 421.9 &       & 10.0  & 9.4 \\
          & 113   &       & 743.0 & 743.4 &       & 1096.0 & 1115.2 &       & 32.2  & 33.3  &       & 494.1 & 494.9 &       & 588.1 & 603.3 &       & 16.0  & 18.0  &       & 380.0 & 383.3 &       & 439.1 & 445.3 &       & 13.5  & 13.9 \\
          & 114   &       & 734.0 & 735.4 &       & 1056.2 & 1085.2 &       & 30.5  & 32.2  &       & 499.5 & 502.5 &       & 573.7 & 596.8 &       & 12.9  & 15.8  &       & 384.7 & 387.2 &       & 430.2 & 432.8 &       & 10.6  & 10.5 \\
          & 115   &       & 748.0 & 748.6 &       & 1080.4 & 1108.7 &       & 30.8  & 32.5  &       & 501.1 & 501.2 &       & 593.8 & 603.5 &       & 15.6  & 16.9  &       & 380.8 & 387.1 &       & 422.8 & 428.0 &       & 9.9   & 9.6 \\
          & 116   &       & 715.7 & 725.8 &       & 1037.5 & 1078.1 &       & 31.0  & 32.7  &       & 491.4 & 494.5 &       & 580.4 & 589.8 &       & 15.3  & 16.2  &       & 396.3 & 397.1 &       & 409.7 & 414.4 &       & 3.3   & 4.2 \\
          & 117   &       & 750.5 & 751.3 &       & 1120.8 & 1150.7 &       & 33.0  & 34.7  &       & 508.2 & 518.3 &       & 631.2 & 640.8 &       & 19.5  & 19.1  &       & 396.2 & 397.7 &       & 468.8 & 472.8 &       & 15.5  & 15.9 \\
          & 118   &       & 692.8 & 696.2 &       & 1019.2 & 1044.8 &       & 32.0  & 33.4  &       & 483.7 & 484.4 &       & 570.5 & 584.1 &       & 15.2  & 17.1  &       & 377.2 & 381.4 &       & 425.0 & 425.9 &       & 11.2  & 10.4 \\
          & 119   &       & 805.6 & 808.5 &       & 1086.9 & 1132.5 &       & 25.9  & 28.6  &       & 511.3 & 527.1 &       & 613.8 & 621.7 &       & 16.7  & 15.2  &       & 390.5 & 395.2 &       & 435.6 & 442.2 &       & 10.4  & 10.6 \\
          & 120   &       & 732.8 & 732.9 &       & 1087.8 & 1111.1 &       & 32.6  & 34.0  &       & 493.2 & 493.5 &       & 605.5 & 611.6 &       & 18.5  & 19.3  &       & 382.0 & 382.8 &       & 459.6 & 462.3 &       & 16.9  & 17.2 \\
    \midrule
    \multicolumn{2}{c}{\textbf{Average}} &       &       &       &       &       &       &       & \textbf{25.3} & \textbf{27.4} &       &       &       &       &       &       &       & \textbf{12.5} & \textbf{12.5} &       &       &       &       &       &       &       & \textbf{7.9} & \textbf{7.3} \\
    \bottomrule
    \end{tabular}%
    }
  \label{tab:TabBenchmarkGonazelezUniformFull}%
\end{table}%

\paragraph{Comparison with the \ac{ALNS} of \cite{WindrasMara2022AnDrops}}

We now compare the performance between \ac{MD-SPP-H} and \ac{ALNS} using one of the three datasets used by \cite{WindrasMara2022AnDrops} made available to us. This dataset, initially introduced by \cite{Ha2018OnDrone} for the single-drop scenario, comprises 60 benchmark instances with 50 and 100 customers. Instances named B1-B10, C1-C10, and D1-D10 contain 50 uniformly distributed customers across service areas spanning 100, 500, and 1,000 km$^2$, respectively. Additionally, instances labeled E1-B10, F1-C10, and G1-D10 consist of 100 uniformly distributed customers across the same area dimensions. Readers can refer to \cite{Ha2018OnDrone} for a more detailed description of these instances. It is important to notice that, for all instances, 20\% of the customers chosen randomly are designated as ineligible for drone delivery. We adjusted \ac{MD-SPP-H} to accommodate this additional constraint by restricting that all customers between node $i$ and node $k$ in Algorithm \ref{alg:1} must be drone-eligible. Further, truck distances are calculated based on the Manhattan metric, while drone distances are determined using the Euclidean metric. \cite{WindrasMara2022AnDrops} set both the truck and the drone speed to 40 km/h. The drone flight endurance is set to 24 minutes, and the time required for the truck to launch and retrieve the drone is 30 and 40 seconds, respectively. During the launch and retrieval phases, the drone is assumed to be in a non-flight mode, eliminating the need to factor this time into flight endurance consumption. Moreover, since \cite{WindrasMara2022AnDrops} assume that the drone load capacity is unlimited, we set $D=n$ in Algorithm \ref{alg:1}. Finally, to ensure a fair comparison and align with the computation times reported by \cite{WindrasMara2022AnDrops}, we use maximum runtimes of 60 and 90 seconds for instances with 50 and 100 customers, respectively. We further adopt the same number of independent runs (i.e., 10 runs). For a summary of our experiment settings, please refer to Table \ref{tab:ExperimentSetting} in \ref{app:supplementary_results_performance}. \\

\begin{table}[htbp]
  \centering
    \caption{Results of the comparison of \ac{MD-SPP-H} and \ac{ALNS}, considering 50 and 100 customers and all service area sizes of \citet{Ha2018OnDrone}.}       
    \renewcommand{\arraystretch}{1.1}
    \resizebox{\textwidth}{!}{
    \begin{tabular}{ccccccccccrrrccccccccccrr}
    \toprule
    \multicolumn{12}{c}{50 Customers }                                                            &       & \multicolumn{12}{c}{100 Customers } \\
\cmidrule{1-12}\cmidrule{14-25}    \multirow{2}[4]{*}{ID} &       & \multicolumn{3}{c}{MD-SPP-H} &       & \multicolumn{3}{c}{ALNS} &       & \multicolumn{2}{c}{$\Delta$} &       & \multirow{2}[4]{*}{ID} &       & \multicolumn{3}{c}{MD-SPP-H} &       & \multicolumn{3}{c}{ALNS} &       & \multicolumn{2}{c}{$\Delta$} \\
\cmidrule{3-5}\cmidrule{7-9}\cmidrule{11-12}\cmidrule{16-18}\cmidrule{20-22}\cmidrule{24-25}          &       & $z_\text{best}$ & $z_\text{avg}$ & $z_\text{std}$ &       & $z_\text{best}$ & $z_\text{avg}$ & $z_\text{std}$ &       & \multicolumn{1}{c}{$\Delta_\text{best}$} & \multicolumn{1}{c}{$\Delta_\text{avg}$} &       &       &       & $z_\text{best}$ & $z_\text{avg}$ & $z_\text{std}$ &       & $z_\text{best}$ & $z_\text{avg}$ & $z_\text{std}$ &       & \multicolumn{1}{c}{$\Delta_\text{best}$} & \multicolumn{1}{c}{$\Delta_\text{avg}$} \\
    \midrule
    mbB101 &       & \multicolumn{1}{r}{76.1} & \multicolumn{1}{r}{78.0} & \multicolumn{1}{r}{1.3} &       & \multicolumn{1}{r}{80.7} & \multicolumn{1}{r}{86.9} & \multicolumn{1}{r}{5.1} &       & 5.7   & 10.3  &       & mbE101 &       & \multicolumn{1}{r}{109.2} & \multicolumn{1}{r}{114.7} & \multicolumn{1}{r}{4.1} &       & \multicolumn{1}{r}{124.8} & \multicolumn{1}{r}{131.9} & \multicolumn{1}{r}{5.6} &       & 12.4  & 13.1 \\
    mbB102 &       & \multicolumn{1}{r}{81.0} & \multicolumn{1}{r}{84.6} & \multicolumn{1}{r}{1.8} &       & \multicolumn{1}{r}{75.3} & \multicolumn{1}{r}{79.3} & \multicolumn{1}{r}{2.9} &       & -7.7  & -6.7  &       & mbE102 &       & \multicolumn{1}{r}{119.0} & \multicolumn{1}{r}{121.7} & \multicolumn{1}{r}{1.6} &       & \multicolumn{1}{r}{122.4} & \multicolumn{1}{r}{130.4} & \multicolumn{1}{r}{5.2} &       & 2.8   & 6.7 \\
    mbB103 &       & \multicolumn{1}{r}{79.4} & \multicolumn{1}{r}{82.0} & \multicolumn{1}{r}{2.1} &       & \multicolumn{1}{r}{77.6} & \multicolumn{1}{r}{86.3} & \multicolumn{1}{r}{5.6} &       & -2.4  & 5.0   &       & mbE103 &       & \multicolumn{1}{r}{122.2} & \multicolumn{1}{r}{124.7} & \multicolumn{1}{r}{1.3} &       & \multicolumn{1}{r}{125.6} & \multicolumn{1}{r}{134.9} & \multicolumn{1}{r}{7.7} &       & 2.7   & 7.5 \\
    mbB104 &       & \multicolumn{1}{r}{77.2} & \multicolumn{1}{r}{78.4} & \multicolumn{1}{r}{1.2} &       & \multicolumn{1}{r}{76.0} & \multicolumn{1}{r}{85.7} & \multicolumn{1}{r}{4.7} &       & -1.6  & 8.5   &       & mbE104 &       & \multicolumn{1}{r}{119.5} & \multicolumn{1}{r}{122.7} & \multicolumn{1}{r}{2.2} &       & \multicolumn{1}{r}{119.5} & \multicolumn{1}{r}{131.6} & \multicolumn{1}{r}{6.8} &       & 0.0   & 6.8 \\
    mbB105 &       & \multicolumn{1}{r}{80.2} & \multicolumn{1}{r}{82.0} & \multicolumn{1}{r}{2.1} &       & \multicolumn{1}{r}{80.1} & \multicolumn{1}{r}{88.8} & \multicolumn{1}{r}{6.0} &       & -0.1  & 7.6   &       & mbE105 &       & \multicolumn{1}{r}{116.8} & \multicolumn{1}{r}{118.5} & \multicolumn{1}{r}{1.0} &       & \multicolumn{1}{r}{122.2} & \multicolumn{1}{r}{128.8} & \multicolumn{1}{r}{5.9} &       & 4.4   & 8.0 \\
    mbB106 &       & \multicolumn{1}{r}{74.0} & \multicolumn{1}{r}{76.8} & \multicolumn{1}{r}{1.2} &       & \multicolumn{1}{r}{74.9} & \multicolumn{1}{r}{81.8} & \multicolumn{1}{r}{5.4} &       & 1.3   & 6.2   &       & mbE106 &       & \multicolumn{1}{r}{118.9} & \multicolumn{1}{r}{121.2} & \multicolumn{1}{r}{1.4} &       & \multicolumn{1}{r}{120.3} & \multicolumn{1}{r}{128.1} & \multicolumn{1}{r}{4.2} &       & 1.2   & 5.4 \\
    mbB107 &       & \multicolumn{1}{r}{79.4} & \multicolumn{1}{r}{82.4} & \multicolumn{1}{r}{1.4} &       & \multicolumn{1}{r}{78.7} & \multicolumn{1}{r}{86.7} & \multicolumn{1}{r}{6.9} &       & -0.9  & 5.0   &       & mbE107 &       & \multicolumn{1}{r}{113.4} & \multicolumn{1}{r}{115.4} & \multicolumn{1}{r}{1.1} &       & \multicolumn{1}{r}{123.8} & \multicolumn{1}{r}{132.3} & \multicolumn{1}{r}{5.1} &       & 8.4   & 12.7 \\
    mbB108 &       & \multicolumn{1}{r}{75.8} & \multicolumn{1}{r}{77.6} & \multicolumn{1}{r}{1.3} &       & \multicolumn{1}{r}{81.9} & \multicolumn{1}{r}{85.6} & \multicolumn{1}{r}{3.0} &       & 7.4   & 9.3   &       & mbE108 &       & \multicolumn{1}{r}{116.7} & \multicolumn{1}{r}{117.7} & \multicolumn{1}{r}{1.2} &       & \multicolumn{1}{r}{115.0} & \multicolumn{1}{r}{126.5} & \multicolumn{1}{r}{8.0} &       & -1.5  & 7.0 \\
    mbB109 &       & \multicolumn{1}{r}{81.9} & \multicolumn{1}{r}{83.1} & \multicolumn{1}{r}{1.0} &       & \multicolumn{1}{r}{79.7} & \multicolumn{1}{r}{85.2} & \multicolumn{1}{r}{3.5} &       & -2.8  & 2.4   &       & mbE109 &       & \multicolumn{1}{r}{117.8} & \multicolumn{1}{r}{120.4} & \multicolumn{1}{r}{1.3} &       & \multicolumn{1}{r}{115.8} & \multicolumn{1}{r}{121.6} & \multicolumn{1}{r}{3.8} &       & -1.7  & 1.1 \\
    mbB110 &       & \multicolumn{1}{r}{76.3} & \multicolumn{1}{r}{78.1} & \multicolumn{1}{r}{1.7} &       & \multicolumn{1}{r}{75.1} & \multicolumn{1}{r}{89.1} & \multicolumn{1}{r}{5.5} &       & -1.6  & 12.3  &       & mbE110 &       & \multicolumn{1}{r}{119.0} & \multicolumn{1}{r}{121.2} & \multicolumn{1}{r}{1.0} &       & \multicolumn{1}{r}{120.2} & \multicolumn{1}{r}{128.8} & \multicolumn{1}{r}{5.7} &       & 1.0   & 5.9 \\
    mbC101 &       & \multicolumn{1}{r}{172.4} & \multicolumn{1}{r}{173.4} & \multicolumn{1}{r}{1.3} &       & \multicolumn{1}{r}{179.7} & \multicolumn{1}{r}{197.6} & \multicolumn{1}{r}{19.8} &       & 4.0   & 12.2  &       & mbF101 &       & \multicolumn{1}{r}{253.9} & \multicolumn{1}{r}{258.8} & \multicolumn{1}{r}{2.7} &       & \multicolumn{1}{r}{262.7} & \multicolumn{1}{r}{307.1} & \multicolumn{1}{r}{20.7} &       & 3.4   & 15.7 \\
    mbC102 &       & \multicolumn{1}{r}{163.6} & \multicolumn{1}{r}{164.2} & \multicolumn{1}{r}{0.9} &       & \multicolumn{1}{r}{171.2} & \multicolumn{1}{r}{192.4} & \multicolumn{1}{r}{19.7} &       & 4.5   & 14.6  &       & mbF102 &       & \multicolumn{1}{r}{243.3} & \multicolumn{1}{r}{248.6} & \multicolumn{1}{r}{5.2} &       & \multicolumn{1}{r}{266.6} & \multicolumn{1}{r}{313.1} & \multicolumn{1}{r}{24.3} &       & 8.7   & 20.6 \\
    mbC103 &       & \multicolumn{1}{r}{182.7} & \multicolumn{1}{r}{190.9} & \multicolumn{1}{r}{5.0} &       & \multicolumn{1}{r}{183.8} & \multicolumn{1}{r}{203.5} & \multicolumn{1}{r}{15.5} &       & 0.6   & 6.2   &       & mbF103 &       & \multicolumn{1}{r}{250.8} & \multicolumn{1}{r}{256.4} & \multicolumn{1}{r}{3.9} &       & \multicolumn{1}{r}{277.7} & \multicolumn{1}{r}{307.7} & \multicolumn{1}{r}{22.3} &       & 9.7   & 16.7 \\
    mbC104 &       & \multicolumn{1}{r}{175.4} & \multicolumn{1}{r}{177.3} & \multicolumn{1}{r}{1.3} &       & \multicolumn{1}{r}{188.1} & \multicolumn{1}{r}{199.9} & \multicolumn{1}{r}{11.3} &       & 6.8   & 11.3  &       & mbF104 &       & \multicolumn{1}{r}{244.9} & \multicolumn{1}{r}{249.2} & \multicolumn{1}{r}{2.2} &       & \multicolumn{1}{r}{281.3} & \multicolumn{1}{r}{308.1} & \multicolumn{1}{r}{20.5} &       & 12.9  & 19.1 \\
    mbC105 &       & \multicolumn{1}{r}{196.9} & \multicolumn{1}{r}{198.0} & \multicolumn{1}{r}{0.8} &       & \multicolumn{1}{r}{202.6} & \multicolumn{1}{r}{216.4} & \multicolumn{1}{r}{10.6} &       & 2.8   & 8.5   &       & mbF105 &       & \multicolumn{1}{r}{262.5} & \multicolumn{1}{r}{265.3} & \multicolumn{1}{r}{1.1} &       & \multicolumn{1}{r}{274.7} & \multicolumn{1}{r}{304.4} & \multicolumn{1}{r}{14.5} &       & 4.5   & 12.9 \\
    mbC106 &       & \multicolumn{1}{r}{195.3} & \multicolumn{1}{r}{197.2} & \multicolumn{1}{r}{2.4} &       & \multicolumn{1}{r}{204.5} & \multicolumn{1}{r}{228.7} & \multicolumn{1}{r}{17.4} &       & 4.5   & 13.8  &       & mbF106 &       & \multicolumn{1}{r}{236.7} & \multicolumn{1}{r}{239.5} & \multicolumn{1}{r}{2.1} &       & \multicolumn{1}{r}{239.5} & \multicolumn{1}{r}{291.6} & \multicolumn{1}{r}{25.5} &       & 1.2   & 17.9 \\
    mbC107 &       & \multicolumn{1}{r}{177.1} & \multicolumn{1}{r}{180.7} & \multicolumn{1}{r}{3.0} &       & \multicolumn{1}{r}{187.7} & \multicolumn{1}{r}{213.2} & \multicolumn{1}{r}{13.6} &       & 5.6   & 15.2  &       & mbF107 &       & \multicolumn{1}{r}{234.2} & \multicolumn{1}{r}{237.7} & \multicolumn{1}{r}{2.2} &       & \multicolumn{1}{r}{279.0} & \multicolumn{1}{r}{319.8} & \multicolumn{1}{r}{32.5} &       & 16.1  & 25.7 \\
    mbC108 &       & \multicolumn{1}{r}{198.2} & \multicolumn{1}{r}{201.8} & \multicolumn{1}{r}{1.7} &       & \multicolumn{1}{r}{199.6} & \multicolumn{1}{r}{218.0} & \multicolumn{1}{r}{9.5} &       & 0.7   & 7.4   &       & mbF108 &       & \multicolumn{1}{r}{249.4} & \multicolumn{1}{r}{254.2} & \multicolumn{1}{r}{3.8} &       & \multicolumn{1}{r}{269.2} & \multicolumn{1}{r}{290.9} & \multicolumn{1}{r}{29.3} &       & 7.3   & 12.6 \\
    mbC109 &       & \multicolumn{1}{r}{190.1} & \multicolumn{1}{r}{193.2} & \multicolumn{1}{r}{2.2} &       & \multicolumn{1}{r}{212.3} & \multicolumn{1}{r}{225.3} & \multicolumn{1}{r}{9.1} &       & 10.4  & 14.2  &       & mbF109 &       & \multicolumn{1}{r}{253.6} & \multicolumn{1}{r}{254.4} & \multicolumn{1}{r}{0.5} &       & \multicolumn{1}{r}{284.5} & \multicolumn{1}{r}{319.3} & \multicolumn{1}{r}{33.1} &       & 10.9  & 20.3 \\
    mbC110 &       & \multicolumn{1}{r}{183.3} & \multicolumn{1}{r}{187.8} & \multicolumn{1}{r}{4.0} &       & \multicolumn{1}{r}{195.2} & \multicolumn{1}{r}{222.3} & \multicolumn{1}{r}{14.6} &       & 6.1   & 15.5  &       & mbF110 &       & \multicolumn{1}{r}{250.3} & \multicolumn{1}{r}{256.3} & \multicolumn{1}{r}{3.2} &       & \multicolumn{1}{r}{258.9} & \multicolumn{1}{r}{304.0} & \multicolumn{1}{r}{24.6} &       & 3.3   & 15.7 \\
    mbD101 &       & \multicolumn{1}{r}{269.8} & \multicolumn{1}{r}{274.1} & \multicolumn{1}{r}{3.7} &       & \multicolumn{1}{r}{293.2} & \multicolumn{1}{r}{315.0} & \multicolumn{1}{r}{24.1} &       & 8.0   & 13.0  &       & mbG101 &       & \multicolumn{1}{r}{343.2} & \multicolumn{1}{r}{349.5} & \multicolumn{1}{r}{4.0} &       & \multicolumn{1}{r}{395.1} & \multicolumn{1}{r}{448.5} & \multicolumn{1}{r}{53.2} &       & 13.1  & 22.1 \\
    mbD102 &       & \multicolumn{1}{r}{270.5} & \multicolumn{1}{r}{273.1} & \multicolumn{1}{r}{2.2} &       & \multicolumn{1}{r}{305.9} & \multicolumn{1}{r}{323.5} & \multicolumn{1}{r}{15.5} &       & 11.6  & 15.6  &       & mbG102 &       & \multicolumn{1}{r}{332.7} & \multicolumn{1}{r}{337.1} & \multicolumn{1}{r}{3.6} &       & \multicolumn{1}{r}{389.4} & \multicolumn{1}{r}{422.8} & \multicolumn{1}{r}{16.9} &       & 14.6  & 20.3 \\
    mbD103 &       & \multicolumn{1}{r}{251.4} & \multicolumn{1}{r}{252.8} & \multicolumn{1}{r}{1.6} &       & \multicolumn{1}{r}{271.3} & \multicolumn{1}{r}{307.9} & \multicolumn{1}{r}{22.7} &       & 7.3   & 17.9  &       & mbG103 &       & \multicolumn{1}{r}{369.9} & \multicolumn{1}{r}{378.9} & \multicolumn{1}{r}{5.5} &       & \multicolumn{1}{r}{431.8} & \multicolumn{1}{r}{488.2} & \multicolumn{1}{r}{29.9} &       & 14.3  & 22.4 \\
    mbD104 &       & \multicolumn{1}{r}{286.2} & \multicolumn{1}{r}{291.0} & \multicolumn{1}{r}{2.5} &       & \multicolumn{1}{r}{307.1} & \multicolumn{1}{r}{339.3} & \multicolumn{1}{r}{21.1} &       & 6.8   & 14.2  &       & mbG104 &       & \multicolumn{1}{r}{362.9} & \multicolumn{1}{r}{367.8} & \multicolumn{1}{r}{2.9} &       & \multicolumn{1}{r}{399.3} & \multicolumn{1}{r}{453.6} & \multicolumn{1}{r}{42.7} &       & 9.1   & 18.9 \\
    mbD105 &       & \multicolumn{1}{r}{281.7} & \multicolumn{1}{r}{283.6} & \multicolumn{1}{r}{2.2} &       & \multicolumn{1}{r}{287.7} & \multicolumn{1}{r}{316.8} & \multicolumn{1}{r}{23.9} &       & 2.1   & 10.5  &       & mbG105 &       & \multicolumn{1}{r}{360.1} & \multicolumn{1}{r}{364.5} & \multicolumn{1}{r}{2.2} &       & \multicolumn{1}{r}{418.1} & \multicolumn{1}{r}{456.9} & \multicolumn{1}{r}{29.8} &       & 13.9  & 20.2 \\
    mbD106 &       & \multicolumn{1}{r}{271.8} & \multicolumn{1}{r}{277.7} & \multicolumn{1}{r}{6.9} &       & \multicolumn{1}{r}{289.9} & \multicolumn{1}{r}{316.4} & \multicolumn{1}{r}{20.0} &       & 6.3   & 12.2  &       & mbG106 &       & \multicolumn{1}{r}{348.7} & \multicolumn{1}{r}{354.9} & \multicolumn{1}{r}{3.6} &       & \multicolumn{1}{r}{391.6} & \multicolumn{1}{r}{462.7} & \multicolumn{1}{r}{45.3} &       & 11.0  & 23.3 \\
    mbD107 &       & \multicolumn{1}{r}{283.0} & \multicolumn{1}{r}{283.6} & \multicolumn{1}{r}{1.1} &       & \multicolumn{1}{r}{297.0} & \multicolumn{1}{r}{322.7} & \multicolumn{1}{r}{25.2} &       & 4.7   & 12.1  &       & mbG107 &       & \multicolumn{1}{r}{341.1} & \multicolumn{1}{r}{345.2} & \multicolumn{1}{r}{2.7} &       & \multicolumn{1}{r}{384.8} & \multicolumn{1}{r}{442.6} & \multicolumn{1}{r}{29.4} &       & 11.3  & 22.0 \\
    mbD108 &       & \multicolumn{1}{r}{249.6} & \multicolumn{1}{r}{257.5} & \multicolumn{1}{r}{5.5} &       & \multicolumn{1}{r}{278.6} & \multicolumn{1}{r}{311.9} & \multicolumn{1}{r}{25.3} &       & 10.4  & 17.4  &       & mbG108 &       & \multicolumn{1}{r}{351.6} & \multicolumn{1}{r}{357.8} & \multicolumn{1}{r}{3.3} &       & \multicolumn{1}{r}{409.0} & \multicolumn{1}{r}{459.0} & \multicolumn{1}{r}{37.2} &       & 14.0  & 22.1 \\
    mbD109 &       & \multicolumn{1}{r}{287.5} & \multicolumn{1}{r}{290.2} & \multicolumn{1}{r}{2.6} &       & \multicolumn{1}{r}{300.1} & \multicolumn{1}{r}{327.6} & \multicolumn{1}{r}{16.8} &       & 4.2   & 11.4  &       & mbG109 &       & \multicolumn{1}{r}{351.9} & \multicolumn{1}{r}{356.1} & \multicolumn{1}{r}{2.7} &       & \multicolumn{1}{r}{405.9} & \multicolumn{1}{r}{468.9} & \multicolumn{1}{r}{34.1} &       & 13.3  & 24.0 \\
    mbD110 &       & \multicolumn{1}{r}{261.4} & \multicolumn{1}{r}{262.2} & \multicolumn{1}{r}{0.6} &       & \multicolumn{1}{r}{265.3} & \multicolumn{1}{r}{281.9} & \multicolumn{1}{r}{13.9} &       & 1.5   & 7.0   &       & mbG110 &       & \multicolumn{1}{r}{350.5} & \multicolumn{1}{r}{358.2} & \multicolumn{1}{r}{4.4} &       & \multicolumn{1}{r}{419.9} & \multicolumn{1}{r}{472.6} & \multicolumn{1}{r}{36.3} &       & 16.5  & 24.2 \\
    \midrule
    \textbf{Average} &       &       &       &       &       &       &       &       &       & \textbf{3.5} & \textbf{10.3} &       &  &       &       &       &       &       &       &       &       &       & \textbf{8.0} & \textbf{15.7} \\
    \bottomrule
    \end{tabular}%
    }
\label{tab:TabBenchmarkMara}
\end{table}%

Table \ref{tab:TabBenchmarkMara} reports the results for instances with 50 and 100 customers and different area sizes. Overall, \ac{MD-SPP-H} improves the best solution found by \ac{ALNS} in 51 out of 60 instances, with $\Delta_\text{best} =3.5\%$ for instances with 50 customers and $\Delta_\text{best} = 8.0\%$ for instances with 100 customers. The improvements are even more pronounced for the average solution values, with average improvements of $\Delta_\text{avg} = 10.3\%$ and $\Delta_\text{avg} = 15.7\%$ for instances with 50 and 100 customers, respectively. This behavior is because \ac{MD-SPP-H} produces more consistent results than \ac{ALNS} in terms of standard deviation values ($z_\text{std}$) across the ten independent runs. For example, for the medium area size instances C and F, \ac{ALNS} reports standard deviations of 14.1 and 24.7 (time units) for instances of 50 and 100 customers, respectively. In comparison, \ac{MD-SPP-H} exhibits much lower standard deviations of 2.2 and 2.7 (time units). Finally, our results show a tendency for greater improvements for larger service areas, particularly in instances D and G (i.e., 1,000 km$^2$).

\subsection{Comparison with state-of-the-art solution approaches for the FSTSP}
\label{sec:benchmarkFSTSP}

To further analyze the performance of \ac{MD-SPP-H}, we perform an additional benchmark with three solution approaches tailored for the \ac{FSTSP} (i.e., the single-drop scenario): the established EP-All heuristic of \cite{Agatz2018OptimizationDrone}, the SPP-All heuristic of \cite{Kundu2021AnProblem}, and the integer L-shaped method of \cite{Vasquez2021AnDecomposition}.
All comparisons are based on the uniform instances of \cite{Agatz2018OptimizationDrone} (extended by \cite{Vasquez2021AnDecomposition} for problem sizes of $n \in \{11,..., 20\}$ customers). For the comparison with EP-All and SPP-All, we use the results that are available on \cite{Kundu2020ResultsHeuristic}, while the best-known values provided by \cite{Vasquez2021AnDecomposition} are used to compare \ac{MD-SPP-H} with their integer L-shaped method. In all benchmarks, the drone’s speed is assumed to be double that of the truck, and the drone's maximum flight endurance is considered unlimited. We summarize the experiment setting in Table \ref{tab:ExperimentSetting} in \ref{app:supplementary_results_performance}.

\paragraph{Comparison with EP-All and SPP-All} Our comparative analysis with the heuristics developed by \citet{Agatz2018OptimizationDrone} and \citet{Kundu2021AnProblem} is summarized in Table \ref{tab:TabBenchmarkAgatzKunduUniform} for all instances with uniform customer distribution. We refer the reader to Tables \ref{tab:TabBenchmarkAgatzKunduSC} and \ref{tab:TabBenchmarkAgatzKunduDC} in \ref{app:onlieB} for corresponding results for single-center and double-center customer distribution. 

\begin{table}[htbp]
    \centering
  \caption{Results of the comparison of \ac{MD-SPP-H} and \ac{EP-All} and \ac{SPP-All} for various customers ranging from $n=50$ to $250$ and uniform customer distributions.} 
    \renewcommand{\arraystretch}{1.1}
    \scalebox{0.6}{%\resizebox{\textwidth}{!}{
    \begin{tabular}{ccrrrrrrrrrrrrrr}
    \toprule
    \multirow{2}[4]{*}{$n$} & \multirow{2}[4]{*}{ID} &       & \multicolumn{3}{c}{MD-SPP-H} &       & \multicolumn{4}{c}{EP-All}    &       & \multicolumn{4}{c}{SPP-All} \\
\cmidrule{4-6}\cmidrule{8-11}\cmidrule{13-16}          &       &       & \multicolumn{1}{c}{$z_\text{best}$} & \multicolumn{1}{c}{$z_\text{avg}$} & \multicolumn{1}{c}{Avg. runtime (s)} &       & \multicolumn{1}{c}{$z$} & \multicolumn{1}{c}{Runtime (s)} & \multicolumn{1}{c}{$\Delta_\text{best}$} & \multicolumn{1}{c}{$\Delta_\text{avg}$} &       & \multicolumn{1}{c}{$z$} & \multicolumn{1}{c}{Runtime (s)} & \multicolumn{1}{c}{$\Delta_\text{best}$} & \multicolumn{1}{c}{$\Delta_\text{avg}$} \\
    \midrule
    \multirow{10}[2]{*}{50} & \multicolumn{1}{r}{71} &       & 387.7 & 390.8 & 4.6   &       & 393.7 & 2.0   & 1.5   & 0.7   &       & 389.9 & 0.0   & 0.6   & -0.2 \\
          & \multicolumn{1}{r}{72} &       & 420.6 & 425.8 & 5.1   &       & 454.7 & 2.7   & 7.5   & 6.3   &       & 454.7 & 0.0   & 7.5   & 6.3 \\
          & \multicolumn{1}{r}{73} &       & 393.8 & 396.5 & 4.3   &       & 410.6 & 2.5   & 4.1   & 3.4   &       & 409.8 & 0.1   & 3.9   & 3.2 \\
          & \multicolumn{1}{r}{74} &       & 409.0 & 412.0 & 3.9   &       & 422.6 & 2.2   & 3.2   & 2.5   &       & 420.5 & 0.1   & 2.7   & 2.0 \\
          & \multicolumn{1}{r}{75} &       & 413.0 & 416.3 & 6.2   &       & 441.0 & 1.9   & 6.4   & 5.6   &       & 441.0 & 0.1   & 6.4   & 5.6 \\
          & \multicolumn{1}{r}{76} &       & 370.7 & 374.8 & 4.1   &       & 376.1 & 2.5   & 1.4   & 0.3   &       & 376.1 & 0.0   & 1.4   & 0.3 \\
          & \multicolumn{1}{r}{77} &       & 415.8 & 419.9 & 4.7   &       & 439.6 & 2.2   & 5.4   & 4.5   &       & 434.2 & 0.1   & 4.2   & 3.3 \\
          & \multicolumn{1}{r}{78} &       & 410.5 & 422.0 & 5.8   &       & 438.2 & 2.1   & 6.3   & 3.7   &       & 438.2 & 0.0   & 6.3   & 3.7 \\
          & \multicolumn{1}{r}{79} &       & 374.4 & 380.0 & 4.6   &       & 397.3 & 2.1   & 5.8   & 4.4   &       & 397.3 & 0.0   & 5.8   & 4.4 \\
          & \multicolumn{1}{r}{80} &       & 361.8 & 364.2 & 4.8   &       & 376.0 & 2.7   & 3.8   & 3.1   &       & 376.0 & 0.0   & 3.8   & 3.1 \\
    \midrule
    \multirow{10}[2]{*}{75} & \multicolumn{1}{r}{81} &       & 459.1 & 465.9 & 20.0  &       & 476.6 & 10.6  & 3.7   & 2.2   &       & 470.4 & 0.2   & 2.4   & 0.9 \\
          & \multicolumn{1}{r}{82} &       & 432.9 & 441.8 & 24.9  &       & 454.4 & 14.1  & 4.7   & 2.8   &       & 454.4 & 0.2   & 4.7   & 2.8 \\
          & \multicolumn{1}{r}{83} &       & 435.2 & 442.7 & 17.1  &       & 456.2 & 8.1   & 4.6   & 3.0   &       & 456.1 & 0.2   & 4.6   & 2.9 \\
          & \multicolumn{1}{r}{84} &       & 463.1 & 469.2 & 27.6  &       & 467.8 & 14.5  & 1.0   & -0.3  &       & 463.8 & 0.2   & 0.1   & -1.2 \\
          & \multicolumn{1}{r}{85} &       & 478.4 & 486.8 & 31.5  &       & 506.9 & 11.3  & 5.6   & 4.0   &       & 499.2 & 0.3   & 4.2   & 2.5 \\
          & \multicolumn{1}{r}{86} &       & 451.5 & 460.4 & 26.7  &       & 493.7 & 13.5  & 8.5   & 6.7   &       & 493.0 & 0.3   & 8.4   & 6.6 \\
          & \multicolumn{1}{r}{87} &       & 460.3 & 464.6 & 22.9  &       & 475.1 & 12.3  & 3.1   & 2.2   &       & 470.8 & 0.3   & 2.2   & 1.3 \\
          & \multicolumn{1}{r}{88} &       & 453.9 & 460.4 & 28.9  &       & 520.2 & 9.1   & 12.7  & 11.5  &       & 517.9 & 0.2   & 12.4  & 11.1 \\
          & \multicolumn{1}{r}{89} &       & 421.8 & 426.5 & 24.2  &       & 467.0 & 9.4   & 9.7   & 8.7   &       & 467.0 & 0.2   & 9.7   & 8.7 \\
          & \multicolumn{1}{r}{90} &       & 452.5 & 457.2 & 18.0  &       & 470.0 & 10.8  & 3.7   & 2.7   &       & 468.9 & 0.2   & 3.5   & 2.5 \\
    \midrule
    \multirow{10}[2]{*}{100} & \multicolumn{1}{r}{91} &       & 547.4 & 558.9 & 78.1  &       & 575.7 & 68.6  & 4.9   & 2.9   &       & 569.7 & 1.7   & 3.9   & 1.9 \\
          & \multicolumn{1}{r}{92} &       & 487.4 & 498.5 & 79.7  &       & 529.3 & 30.0  & 7.9   & 5.8   &       & 523.6 & 1.7   & 6.9   & 4.8 \\
          & \multicolumn{1}{r}{93} &       & 491.2 & 503.5 & 71.2  &       & 547.0 & 51.2  & 10.2  & 8.0   &       & 535.6 & 1.0   & 8.3   & 6.0 \\
          & \multicolumn{1}{r}{94} &       & 512.2 & 526.6 & 71.3  &       & 543.8 & 73.0  & 5.8   & 3.2   &       & 543.8 & 1.6   & 5.8   & 3.2 \\
          & \multicolumn{1}{r}{95} &       & 528.3 & 541.0 & 86.0  &       & 591.7 & 50.7  & 10.7  & 8.6   &       & 591.1 & 1.6   & 10.6  & 8.5 \\
          & \multicolumn{1}{r}{96} &       & 546.5 & 553.3 & 56.9  &       & 558.0 & 78.0  & 2.1   & 0.8   &       & 556.1 & 1.4   & 1.7   & 0.5 \\
          & \multicolumn{1}{r}{97} &       & 554.4 & 565.2 & 71.2  &       & 592.8 & 67.5  & 6.5   & 4.7   &       & 590.5 & 1.7   & 6.1   & 4.3 \\
          & \multicolumn{1}{r}{98} &       & 514.7 & 520.4 & 50.6  &       & 517.0 & 38.8  & 0.5   & -0.7  &       & 516.8 & 0.7   & 0.4   & -0.7 \\
          & \multicolumn{1}{r}{99} &       & 536.9 & 546.2 & 60.8  &       & 574.5 & 36.7  & 6.5   & 4.9   &       & 558.7 & 0.7   & 3.9   & 2.2 \\
          & \multicolumn{1}{r}{100} &       & 516.5 & 530.8 & 89.2  &       & 578.7 & 33.0  & 10.7  & 8.3   &       & 577.3 & 0.7   & 10.5  & 8.0 \\
    \midrule
    \multirow{10}[2]{*}{175} & \multicolumn{1}{r}{101} &       & 684.0 & 699.2 & 463.9 &       & 734.9 & 452.0 & 6.9   & 4.9   &       & 730.4 & 7.9   & 6.3   & 4.3 \\
          & \multicolumn{1}{r}{102} &       & 683.1 & 704.7 & 542.8 &       & 728.3 & 418.1 & 6.2   & 3.2   &       & 727.8 & 15.8  & 6.1   & 3.2 \\
          & \multicolumn{1}{r}{103} &       & 674.3 & 697.3 & 660.1 &       & 727.0 & 985.2 & 7.2   & 4.1   &       & 725.8 & 16.6  & 7.1   & 3.9 \\
          & \multicolumn{1}{r}{104} &       & 657.3 & 670.2 & 509.8 &       & 703.6 & 993.3 & 6.6   & 4.8   &       & 685.6 & 17.5  & 4.1   & 2.2 \\
          & \multicolumn{1}{r}{105} &       & 695.8 & 708.3 & 459.3 &       & 726.9 & 813.8 & 4.3   & 2.6   &       & 725.6 & 8.0   & 4.1   & 2.4 \\
          & \multicolumn{1}{r}{106} &       & 678.1 & 689.4 & 266.3 &       & 691.7 & 786.4 & 2.0   & 0.3   &       & 688.5 & 16.7  & 1.5   & -0.1 \\
          & \multicolumn{1}{r}{107} &       & 654.1 & 659.0 & 430.0 &       & 661.4 & 654.3 & 1.1   & 0.4   &       & 653.7 & 7.6   & -0.1  & -0.8 \\
          & \multicolumn{1}{r}{108} &       & 677.4 & 689.6 & 749.1 &       & 728.4 & 383.6 & 7.0   & 5.3   &       & 718.7 & 6.8   & 5.7   & 4.0 \\
          & \multicolumn{1}{r}{109} &       & 686.2 & 705.6 & 495.3 &       & 716.9 & 377.7 & 4.3   & 1.6   &       & 715.6 & 8.0   & 4.1   & 1.4 \\
          & \multicolumn{1}{r}{110} &       & 637.0 & 644.2 & 314.8 &       & 654.1 & 370.1 & 2.6   & 1.5   &       & 650.0 & 8.8   & 2.0   & 0.9 \\
    \midrule
    \multirow{10}[2]{*}{250} & \multicolumn{1}{r}{111} &       & 813.2 & 822.7 & 1480.1 &       & 822.0 & 2944.8 & 1.1   & -0.1  &       & 820.2 & 35.3  & 0.9   & -0.3 \\
          & \multicolumn{1}{r}{112} &       & 804.9 & 812.4 & 987.4 &       & 806.3 & 2534.2 & 0.2   & -0.8  &       & 802.0 & 31.4  & -0.4  & -1.3 \\
          & \multicolumn{1}{r}{113} &       & 795.3 & 814.2 & 1431.2 &       & 820.9 & 2475.7 & 3.1   & 0.8   &       & 818.9 & 34.6  & 2.9   & 0.6 \\
          & \multicolumn{1}{r}{114} &       & 805.8 & 817.9 & 1135.6 &       & 830.9 & 2749.9 & 3.0   & 1.6   &       & 830.6 & 36.7  & 3.0   & 1.5 \\
          & \multicolumn{1}{r}{115} &       & 818.0 & 830.3 & 1384.5 &       & 845.3 & 4871.8 & 3.2   & 1.8   &       & 844.6 & 37.3  & 3.2   & 1.7 \\
          & \multicolumn{1}{r}{116} &       & 807.7 & 819.4 & 1672.9 &       & 831.1 & 1969.4 & 2.8   & 1.4   &       & 826.4 & 37.9  & 2.3   & 0.8 \\
          & \multicolumn{1}{r}{117} &       & 823.0 & 828.4 & 1185.1 &       & 840.9 & 2255.2 & 2.1   & 1.5   &       & 840.9 & 32.4  & 2.1   & 1.5 \\
          & \multicolumn{1}{r}{118} &       & 761.3 & 774.2 & 1574.0 &       & 790.1 & 2191.0 & 3.6   & 2.0   &       & 780.2 & 33.3  & 2.4   & 0.8 \\
          & \multicolumn{1}{r}{119} &       & 833.9 & 851.8 & 1766.0 &       & 871.4 & 4910.6 & 4.3   & 2.3   &       & 868.8 & 85.3  & 4.0   & 2.0 \\
          & \multicolumn{1}{r}{120} &       & 789.0 & 801.9 & 1358.6 &       & 811.0 & 4684.4 & 2.7   & 1.1   &       & 805.2 & 69.7  & 2.0   & 0.4 \\
    \midrule
    \multicolumn{2}{c}{\textbf{Average}} &       &       &       &       &       &       &       & \textbf{4.9} & \textbf{3.3} &       &       &       & \textbf{4.3} & \textbf{2.8} \\
    \bottomrule
    \end{tabular}%
    }
\label{tab:TabBenchmarkAgatzKunduUniform}
\end{table}%

First, the results clearly demonstrate that \ac{MD-SPP-H} significantly improves solution quality compared to EP-All, with similar or reduced runtimes. The reduction in runtime is particularly noticeable for large instances with 175 and 250 customers.
In terms of solution quality, \ac{MD-SPP-H} improves the best-found solution value by $\Delta_\text{best} = 4.9\%$, $3.9\%$, and $3.6\%$ (average over the 50 instances) for uniform, single-center and double-center customer distribution, respectively. The average solution value reached by \ac{MD-SPP-H} improves the solution found by EP-All by $\Delta_\text{avg} = 3.3\%$, $2.2\%$, and $2.1\%$ (average over the 50 instances). %Therefore, \ac{MD-SPP-H} outperforms EP-All in terms of both solution quality and computational efficiency.

Second, we compare our results with the SPP-All heuristic of \cite{Kundu2021AnProblem}. The notable advancement in their study, compared to \cite{Agatz2018OptimizationDrone}, is the significant reduction in computation time while maintaining or slightly improving solution quality. For instance, SPP-All requires an average runtime of 43.4 seconds for instances with 250 customers. Regarding solution quality, \cite{Kundu2021AnProblem} report achieving an overall mean improvement of 0.77\% over EP-All (in contrast, \ac{MD-SPP-H} improves the overall solution quality by $\Delta_\text{best} = 4.9\%$ and $\Delta_\text{avg} = 3.3\%$). Our result tables show that \ac{MD-SPP-H} improves the best solution values in 48 out of 50 instances, with $\Delta_\text{best} = 4.3\%$, $2.7\%$ and $3.1\%$ (average over the 50 instances) for uniform, single-center and double-center distributions. 
%The improvement varies between different customer sizes. Focusing on uniform distributions, the improvements peak for instances with 100 customers with $\Delta_\text{best} = 5.8\%$ (average over the 10 instances). 
The average solution value reached by \ac{MD-SPP-H} improves the best-known solution found by SPP-All by $\Delta_\text{avg} = 2.8\%, 0.9\%$ and 1.5\% for uniform, single-center and double-center distributions, respectively.

Regarding runtimes, \ac{MD-SPP-H} requires an average of five seconds for instances with 50 uniformly distributed customers, while an average of 71 seconds for instances with 100 customers. For instances with 250 customers, the average runtime of \ac{MD-SPP-H} is around 23 minutes. While these runtimes significantly exceed those achieved by \ac{SPP-All}, they remain reasonable when considering the \ac{FSTSP-MD}'s inherent complexity.
The reader can refer to Figure \ref{fig:FigBenchmarkFSTSPImpTrajectoryBest} in \ref{app:supplementary_results_performance} for additional insights into the relative percentage difference (over time) between the best-so-far solution and the best overall solution found by \ac{MD-SPP-H}. For example, for instances with 250 customers, the solutions found by \ac{MD-SPP-H} after 10 minutes are within 4\% of the best overall solutions. For real-world applications, this implies that last-mile logistics managers could execute \ac{MD-SPP-H} for just a couple of minutes and get high-quality solutions. 

\paragraph{Comparison with the integer L-shaped method of \cite{Vasquez2021AnDecomposition}} Lastly, we compare \ac{MD-SPP-H} with the integer L-shaped method developed by \cite{Vasquez2021AnDecomposition}. 
%(Recall that a meaningful comparison for the \ac{FSTSP-MD} is not possible as existing exact approaches can only address instances of up to eight customers; see Section \ref{sec:literature}.)  
This exact method finds optimal solutions for \ac{FSTSP} instances with up to 20 customers. Table \ref{tab:comparison_vasquez} in \ref{app:supplementary_results_performance} reports the comparison using the best-known values provided by \cite{Vasquez2021AnDecomposition}. Overall, \ac{MD-SPP-H} solves to optimality 73 out of the 88 instances that the integer L-shaped method solves to optimality. Further, the average runtime of \ac{MD-SPP-H} is less than half a second. Finally, for instances with more than 17 customers, \ac{MD-SPP-H} improves the best solutions found by the integer L-shaped method by 1.4\%.

In summary, although designed for the multi-drop scenario, \ac{MD-SPP-H} marks a substantial step forward in improving solution quality for the \ac{FSTSP} while maintaining practical runtimes.

\section{Impact of Individual Parameters and Their Interaction on the Delivery System Performance} \label{sec:managerial_insights}

In Section \ref{sec:heuristic_performance}, we showed the superior performance of \ac{MD-SPP-H} compared to state-of-the-art solution approaches for both the \ac{FSTSP-MD} and the \ac{FSTSP}. In this section, we leverage our proposed solution approach to improve our understanding of collaborative truck-and-drone delivery systems for last-mile logistics. 

For the analyses presented in this section, we rely on the well-established \ac{FSTSP} instances from \citet{Agatz2018OptimizationDrone}. We replicate the conditions in their study by adhering to the following assumptions. For all scenarios, both the truck and the drone travel the Euclidean distance between the nodes, and the truck travels at a unit speed. Customer service times and drone launching and recovery times are negligible \citep[see also, e.g.,][]{Gonzalez-R2020Truck-dronePlanning}. Additionally, we conduct experiments on real-world routing data, which are presented in \ref{app:realworld}. In contrast to the stylized dataset, the real-world routing dataset includes actual customer distributions and average travel times for the truck based on the road network. Noteworthy, the results presented in \ref{app:realworld} for real-world instances closely align with those obtained for the stylized instances of \cite{Agatz2018OptimizationDrone} analyzed in this section.

Throughout this section, we use the parameter values of $\eta = 10$ and $p = 0.1$ (used in the perturbation phase of \ac{MD-SPP-H}) and the computational setting described in Section \ref{sec:heuristic_performance}.

\subsection{Design of Numerical Experiments}
\label{sec:num_exp}

\paragraph{Baseline scenario}
We establish a baseline scenario following the one chosen by \cite{Agatz2018OptimizationDrone} for the single-drop case. We select instances with 100 uniformly distributed customers within the service area. The drone travels at twice the speed of the truck, and we consider an unlimited drone flight endurance. We set the baseline maximum number of drone drops to $D = 2$. 

\paragraph{Scenarios of analysis}
Building on this baseline, we analyze the effects of varying the following input parameters: the maximum number of drone drops, the drone speed ratio (compared to the truck speed), the drone's maximum flight endurance, and the number and geographical distribution of the customers. In total, our experiment design gives rise to 900 scenarios, i.e., unique parameter combinations. For each scenario, we include ten instances (i.e., we solve a total of 9,000 instances). A summary of our experiment design is presented in Table \ref{tab:ExpSetting_insights}. In the following, we explain our parameter choices.

\paragraph{1) Number of drone drops} 
We systematically vary the number of drone drops within $D \in \{1, 2, 4, 6, 10\}$ to capture potential practical limitations of drone payload capacity, whether due to weight or equipment constraints. Note that we also include the case of a single-drop drone to facilitate a direct comparison of our findings with previous studies.

\paragraph{2) Drone speed ratio} 
Starting with a baseline drone speed ratio (i.e., the ratio of the drone speed to the truck speed) of 2, we decrease and increase the ratio to 1 and 3, respectively. These drone speed ratios are used in several prior studies \citep[see, e.g.,][]{Gonzalez-R2020Truck-dronePlanning, Roberti2021ExactDrone, Schermer2019AVariants}. 

\paragraph{3) Drone flight endurance} 

Through preliminary experiment runs, we identified 100 time units as an effectively unlimited drone flight endurance level (when all remaining parameters take their baseline values). 
Starting with 100 time units as the upper limit, we decrease the drone flight endurance in steps of 25 time units, exploring values down to 25 time units. These values are chosen to cover a wide range of restriction levels. A comparison between actual flight duration and maximum flight endurance across all scenarios with uniform customer distribution, as shown in Figure \ref{fig:FigMIUniformMaxFlightDuration}, reveals that our chosen range extends from unrestricted to notably restricted drone flight endurance. 

\paragraph{4) Number and 5) distribution of the customers} 
As provided in the dataset by \citet{Agatz2018OptimizationDrone}, the number of customers within the service area varies between 50 and 250 customers. Further, in addition to the uniform customer distribution, we also consider instances with single-center and double-center customer distributions. Note that instances with uniform customer distributions are limited to a service area of $100 \times 100$ distance square units, resembling highly dense urban settings. In contrast, instances with single-center and double-center customer distributions cover approximately $300 \times 300$ and $500 \times 300$ distance square units, respectively. These configurations may resemble less dense service areas where customers are concentrated around one or two central locations.\\ 

\begin{table}[htbp]
    \centering
    \renewcommand{\arraystretch}{1.1}
    \caption{Summary of parameter choices for numerical experiments on problem instances from  \citet{Agatz2018OptimizationDrone}}
    \scalebox{0.9}{
    \begin{tabular}{lll}
    \toprule
    Type  & Parameter & Value \\
    \midrule
    \multirow{3}[1]{*}{General} 
          & Total number of instances  & 150 \\
          & Customer distribution & Uniform$^*$, single-center, double-center \\
          & Number of customers  &  $\{50,75,100^*,175,250\}$  \\
          
    \midrule
    \multirow{8}[1]{*}{Vehicle} & Drone eligibility  & 100\% \\
          & Truck distance  &  Euclidean [unit] \\
          & Drone distance &  Euclidean [unit] \\
          & Truck speed  & 1 \\
          & Drone speed ratio &  $\{1,2^*,3\}$  \\
          & Flight endurance &   $\{25,50,75,100^*\}$  \\
          & Number of drops  &   $\{1,2^*,4,6,10\}$  \\
          & Launch/retrieval time  &  0/0  \\
    \midrule
    \multirow{2}[1]{*}{CPU} & Run-time limit  &  10 min (for $D \leq 4$), 30 min (for $D \geq 6$) \\
          & Runs  & 5 \\
    \bottomrule
    \multicolumn{3}{l}{\small \textit{Note.} * denotes values used in the baseline scenario.} 
    \end{tabular}%
    }
  \label{tab:ExpSetting_insights}%
\end{table}%

For each scenario, we evaluate the performance of the cooperative truck-and-drone system compared to the truck-only system. By default, we calculate the percentage of completion time savings as 
\begin{equation*}
    \Delta = \frac{z^{\text{TSP}}-z^{\text{MD-SPP-H}}}{z^{\text{TSP}}} \times 100,
\end{equation*}
where $z^{\text{TSP}}$ is the completion time of the truck-only delivery system obtained by the Concorde solver \citep{Applegate2006TheStudy} and $z^\text{MD-SPP-H}$ is the completion time of the truck-and-drone delivery system obtained by \ac{MD-SPP-H}. If not indicated otherwise, we report the average time saving over the ten instances for each scenario. 

\subsection{Direct Effects}
\label{sec:direct}

For our baseline parameter setting (see Section \ref{sec:num_exp}), the truck-and-drone system yields average completion time savings of 41.6\% compared to the truck-only system (see Figure \ref{fig:FigMISummary}(a) for two drops). In the following, we first discuss the direct effects of the three drone operational parameters on the delivery system's performance. Second, we investigate how each of the service area's characteristics, i.e., the number and distribution of customers, affect the time savings.

\begin{figure}[ht]
    \centering
    %[width=1.0\textwidth]
    \includegraphics{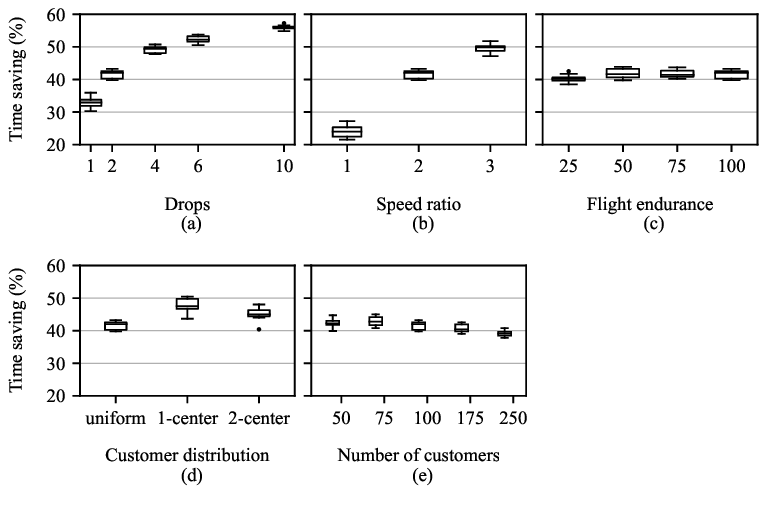}
    \caption{Percentage of time savings for ten problem instances, obtained by varying a single parameter from our baseline scenario.}  
    \label{fig:FigMISummary}
\end{figure}

\paragraph{Altering the number of drone drops}

Figure \ref{fig:FigMISummary}(a) shows that the truck-and-drone delivery system significantly reduces delivery duration compared to the truck-only system, with time savings increasing from 33\% for a single-drop drone to 56\% for a drone capable of making up to ten drops. However, our results show that this relationship is not linear, but is characterized by rapidly diminishing marginal returns as the number of drops increases.

\paragraph{Altering the drone speed ratio}

The relative speed of the drone compared to the truck emerges as a critical factor influencing the achievable time savings. For the baseline drone speed ratio of 2, Figure \ref{fig:FigMISummary}(b) shows completion time savings of 41.6\%, which decrease to 24.1\% for a speed ratio of 1 and increase to 49.6\% for a speed ratio of 3. Again, we observe diminishing marginal returns as the speed ratio increases. This is likely due to the need for coordination between truck and drone movements. In other words, if we only increase the drone's speed while keeping the other system parameters fixed, the drone will have to wait for the truck more frequently and for longer periods. % Note that this finding is consistent with conclusions that \cite{Agatz2018OptimizationDrone} obtained for the single-drop case.

\paragraph{Altering the drone flight endurance}

Intuitively, 
% a sufficiently high drone flight endurance allows the system to reap the full potential of time savings offered by combining trucks and drones for last-mile logistics. Conversely, 
the potential time savings are diminished if the drone's ability to move independently is limited due to limited flight endurance. For instance, Figure \ref{fig:FigMISummary}(c) shows a reduction in time savings by 1.5\% as the flight endurance level drops to 25 time units. For endurance levels of 50 and higher, this parameter does not impact the time savings, i.e., the savings plateau. 
% This implies that, while meeting a certain flight endurance threshold is essential to leveraging the full efficiency gains, further improvement beyond this saturation point yields no additional benefits.
While the impact of reducing the drone flight endurance is almost negligible in our baseline scenario, it is more pronounced for other parameter combinations. We refer to our discussion in Section \ref{sec:interation}.

\paragraph{Non-uniform customer distributions}
Figure \ref{fig:FigMISummary}(d) shows that the attainable time savings depend on the customer distributions, with average time savings increasing from 41.6\% for uniform instances to 47.7\% and 45.1\% for single- and double-center instances, respectively. This implies that drones are more advantageous when serving distribution areas where customers are concentrated around a few central locations, with only a small number of customers located in outlying regions. As \cite{Agatz2018OptimizationDrone} highlights for single-drop drones, a plausible explanation is that the drone supports the truck by serving customers furthest from the center of the cluster(s) efficiently.

\paragraph{Altering the number of customers}

As we vary the number of customers between 50 and 250, the relative time savings vary between 39.2\% and 42.8\%  (see Figure \ref{fig:FigMISummary}(e)).
However, our results do not reveal a clear and unambiguous relationship between the number of customers and the obtainable time savings (similar results were obtained for the single- and double-center customer distributions). 
For our baseline scenario, relative time savings remain relatively stable despite significant changes in the number of customers. 
%This implies that customer density plays a minor role in the benefits achieved by the truck-and-drone system (an observation that, of course, highly relies on the chosen baseline parameters). Therefore, w
We conduct an additional investigation in Section \ref{sec:interation} to examine this effect in more detail.

\subsection{Parameter Interactions}
\label{sec:interation}

After examining how individual parameter impacts the dynamics of the delivery system, we further investigate how these effects interact by conducting a full factorial analysis. We split this section into two parts. 
In Section \ref{system-inherent-dynamics}, we first discuss the delivery system's inherent dynamics, i.e., how the three operational drone parameters interact to impact the relative time savings. 
Section \ref{operational-environment} then discusses how the characteristic of the service area impacts its inherent dynamics (and, consequentially, the time savings). 
% To this end, we discuss the dominant patterns emerging from our full factorial experiment data. %The results for uniform customer distributions are summarized in Figure \ref{fig:FigMIUniformTimeRatio}. Results for single- and double-center customer distributions are presented in Figures \ref{fig:FigMISingleTimeRatio} and \ref{fig:FigMIDoubleTimeRatio} in \ref{app:MI}, respectively.

\begin{figure}[ht]
    \RawFloats
    \centering
    %[width=0.9\textwidth]
    \includegraphics{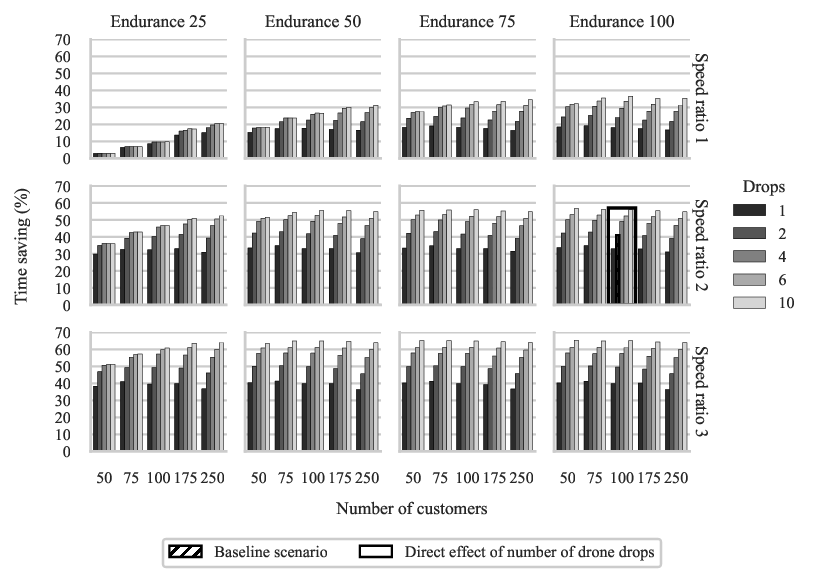}
    \caption{Average percentage time savings over ten problem instances from a truck-and-drone system compared to the truck-only alternative for varying parameter settings and uniform customer distribution.}
    \label{fig:FigMIUniformTimeRatio}
\end{figure}

\subsubsection{System-Inherent Dynamics} \label{system-inherent-dynamics}

% In the following, we explore how different levels of one drone operational parameter affect the impact of changes in another parameter on time savings. We stick to the baseline operational environment of 100 uniformly distributed customers.

\paragraph{Interaction of key drivers of time savings: drone drops and speed ratio}

Figure \ref{fig:FigMIUniformTimeRatio} shows that the drone speed ratio moderates the benefits of increasing the number of possible drone drops.
With the baseline endurance of 100 time units and a drone speed ratio of 1 (2, 3), time savings increase from around 18.1\% (33.0\%, 39.8\%) for a single-drop drone to around 36.6\% (56.0\%, 65.0\%) for a ten-drop drone. This represents an absolute increase of 18.5 (23.0, 25.2) percentage points, which indicates that the faster the drone, the larger the absolute gain from increasing the number of drops. However, the relative percentage increase is about 102.2\% (69.7\%, 63.3\%). Therefore, the faster the drone, the lower the relative gain from more drops. 
This is an intuitive result since a faster drone has a larger time-saving potential to begin with, even in the single-drop case, than a drone that travels at the same speed as the truck.

\paragraph{Flight endurance as a prerequisite for multiple drops}

A sufficient flight endurance limit is crucial for achieving the full benefits of using a multi-drop drone. We observe a truncation effect for low flight endurance levels: while we still see measurable benefits from moving from a low number of drops per drone sortie to a slightly higher number, the savings plateau, and no further improvements are obtained beyond a certain number of drops per sortie (see Figure \ref{fig:FigMIUniformTimeRatio} for an endurance level of 25 time units and speed ratio of 2).
This effect is more pronounced and sets in at higher levels of flight endurance as the drone speed ratio decreases (see Figure \ref{fig:FigMIUniformTimeRatio} for an endurance level of 50 comparing speed ratios 1 and 2). This is expected because a very restrictive flight endurance keeps the drone from exploiting the ability to make many drops per sortie unless these drops are closely co-located.
Further, our analysis reveals that the sum of all drone flying times is inversely related to the number of drops that a drone can make per flight (cf., Figure \ref{fig:FigMIUniformTotalFlightDuration} in \ref{app:MI}). Although the maximum duration of a single drone flight increases with a higher number of drops, we note that the aggregate duration of all drone flights combined decreases as the maximum number of possible drops increases. 

\begin{figure}[ht]
     \RawFloats
     %[width=0.9\textwidth]
     \centering   
     \includegraphics{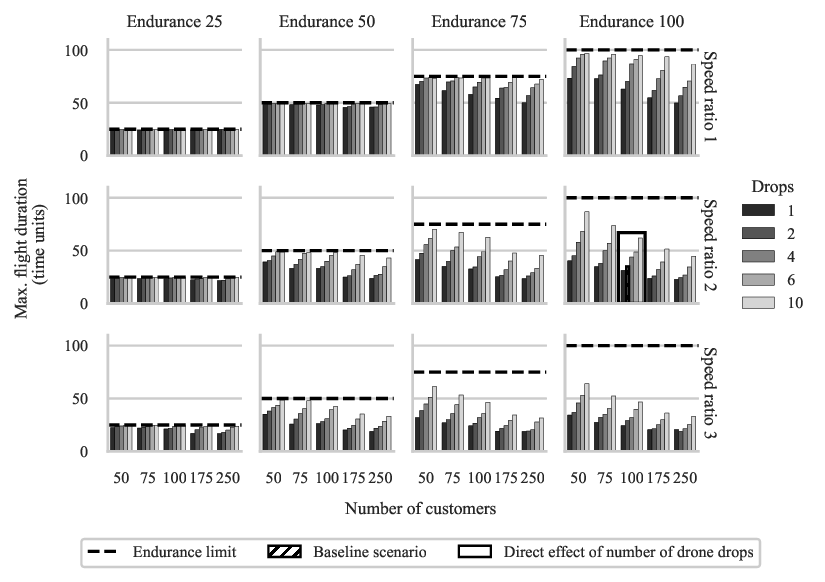}
     \caption{Comparison of average observed maximum drone flight duration over ten instances and imposed drone flight endurance limit for instances with uniform customer distribution.}
     \label{fig:FigMIUniformMaxFlightDuration}
\end{figure}

\subsubsection{Impact of the Operational Environment} \label{operational-environment}

%The operational environment (i.e., customer density and distribution) significantly affects the distances that the truck and the drone need to travel. For instance, the lower the customer density, the greater the distance between customers, requiring longer drone flight distances. Similarly, transitioning from a uniform to a non-uniform customer distribution increases drone flight distances, as the drone is typically used to predominantly serve remote customers (see Section \ref{sec:direct}). 
%\blue{We examine the impact of varying customer density and distribution in the following.} 
% While the potential savings in these scenarios might be particularly high, they can challenge the drone's flight endurance. %As such, shifts in the operational environment can act as stress tests for the collaborative truck-and-drone delivery system. This highlights the importance of carefully assessing the suitability of the drone's endurance and speed to reach optimal performance for different use cases.

\paragraph{Serving areas with low customer density} 

For most parameter combinations depicted in Figure \ref{fig:FigMIUniformTimeRatio}, the time savings from the truck-and-drone system are independent of the number of customers and, therefore, the customer density. Only for low endurance levels and speed ratios, the achievable time savings quickly diminish as the number of customers decreases, to the point where virtually no time savings can be obtained (see the top left panel of Figure \ref{fig:FigMIUniformTimeRatio}). This suggests that the benefits of adding a multi-drop drone are largely independent of the average distance between customers in the service area as long as the drone's flight endurance is sufficient and the drone is comparably fast. % However, without these requirements, the drone \blue{cannot} effectively contribute to reducing delivery completion times, particularly in service areas with a low customer density. 

\paragraph{Serving areas with non-uniform customer distribution}

We first consider scenarios in which the drone is effectively not restricted in flight endurance. In these cases, our results reveal that the impact of increasing the number of drone drops on time savings is consistent across varying customer distributions, that is, we identify the same general pattern of marginal diminishing returns (i.e., similar incremental time savings as the number of drone drops increases). However, the absolute level of time savings is higher in non-uniform compared to uniform customer distribution (cf., Figure \ref{fig:FigMIUniformTimeRatio} and Figures \ref{fig:FigMISingleTimeRatio}, and \ref{fig:FigMIDoubleTimeRatio} in \ref{app:MI} for a flight endurance level of 100 time units and a speed ratio of 3).
Interestingly, moving from unrestricted drone movements to more restrictive levels of flight endurance and lower speed ratios, we find that a non-uniform customer distribution might act as a catalyst for the effects of endurance and speed. As non-uniform distributions demand longer distances per drone flight, the same level of drone flight endurance and speed have a stronger limiting effect in non-uniform than in uniform scenarios. %, which in turn impacts the potential time savings in these settings.
% Lastly, the sum of the duration of all drone flights in non-uniform distributions (as seen for single-center distributions in Figure \ref{fig:FigMISingleTotalFlightDuration} in \ref{app:MI}) significantly surpasses that in the uniform case. For a uniform distribution, the total flight duration does not exceed 1000 time units. In contrast, the maximum total flight duration in our dataset for single-center distribution reaches 1500 time units, marking an increase of 50\%. %This discrepancy highlights the unique operational requirements in different customer distribution settings and underscores the importance of considering the characteristics of the specific use case when planning for drone-assisted deliveries.

\section{Implications for Practice}
\label{sec:practice}

The findings from our study of the impact of different parameters on time savings lead to several implications and recommendations for practitioners and policymakers. 
First, our work shows that the application of collaborative truck-and-drone delivery systems can significantly reduce total delivery times. 
Second, our experiment data show that the achievable relative truck distance savings closely mirror the relative time savings, i.e., both metrics take on similar values. As a result, trucks would travel fewer miles and spend less time on the road, reducing emissions and congestion. 
While the potential savings are considerable, their extent is primarily determined by the characteristics of the service area and available drone technology, underscoring the need for case-dependent evaluations.

Specifically, our analysis indicates that the attainable time savings strongly depend on the speed ratio between the truck and the drone and the maximum number of drops. For both parameters, we observe diminishing marginal returns. We find that single-drop drones or drones capable of making a few drops -- which are already available on the market \citep{Sakharkar2021A2ZDrone, Wingcopter2023Wingcopter198} -- that travel at a moderate speed (e.g., twice the speed of a traditional delivery truck) can lead to significant time savings. For a uniform customer distribution and sufficient flight endurance levels, the time savings from adding a drone that travels at twice the speed of the truck range between 33.1\% for a single-drop drone and 48.7\% for a drone capable of up to four drops. Additionally, allowing the drone to serve multiple customers in one sortie reduces the number of back-and-forth trips performed by the drone between the truck and the customers it serves. As a result, the collaborative system completes its deliveries faster, and the drone spends less time in flight mode throughout the entire delivery process. However, considering the increased technological and operational complexity and costs associated with systems designed for numerous drops (i.e., five or more), a thorough economic assessment based on the specific use case is warranted before adopting and deploying such systems. 

Considering the trade-off between investing in increasing either drone speed or the number of possible drops, we find that both input parameters alter the dynamics of the truck-and-drone delivery system in distinct ways, and the incremental improvements will depend on the remaining operational parameters of the drone and the characteristics of the service area. For instance, we find for our baseline scenario that increasing the number of drone drops from two to four yields the same efficiency gains as increasing the drone-to-truck speed ratio from 2:1 to 3:1 (cf. Figure \ref{fig:FigMISummary}(a) and (b)). Therefore, when adopting drone delivery operations, balancing the investment in drone speed against the practical benefits of multiple drops is crucial, focusing on the most cost-effective combination for the specific use case.

Finally, our study can help last-mile delivery companies decide for which operational environments (characterized by the density and spatial distribution of customers) they should consider investing in a truck-and-drone delivery system, and policymakers decide for which environments they should provide the policy framework. We find that the operational environment can act as a stress test for the collaborative system, highlighting the importance of carefully assessing the suitability of the drone's flight endurance and speed to reach optimal performance for different use cases. For example, we show that potential time savings can be particularly high when serving non-uniformly distributed customers or service areas with low demand density, but enhanced drone capabilities (speed and endurance) are needed to reap the full benefits of the collaborative system.

\section{Conclusion} \label{sec:conclusion}

In this paper, we investigate the benefits of combining conventional ground delivery vehicles and aerial cargo drones for last-mile logistics. This approach leverages the numerous advantages of using drones, including their low per-vehicle capital expenditure costs, reduced carbon footprint, and travel directness and speed. Specifically, we explore the value of supporting ground delivery vehicles with drones capable of making multiple package deliveries per sortie.

We study the \acf{FSTSP-MD}, a multi-modal last-mile delivery model where a single truck is supported by a single multi-drop drone during the delivery process. In this model, the drone can be launched from the truck en route to make autonomous package deliveries to multiple customers before returning to the truck. The \ac{FSTSP-MD} aims to determine the synchronized truck and drone delivery routes that minimize the completion time of the delivery process.

We propose the \acf{MD-SPP-H} to solve realistically-sized \ac{FSTSP-MD} instances. \ac{MD-SPP-H} is a simple and highly effective order-first, split-second heuristic that combines standard local search and diversification techniques with a novel shortest-path problem that finds \ac{FSTSP-MD} solutions (for a given sequence of customers) in polynomial time. 

Using well-established instances with between 50 and 250 customers, we show that \ac{MD-SPP-H} outperforms state-of-the-art heuristics developed for the \ac{FSTSP-MD} and the \ac{FSTSP}. Further, compared to the exact solution approach developed by \cite{Vasquez2021AnDecomposition} for the \ac{FSTSP}, we show that \ac{MD-SPP-H} solves most instances (with up to 20 customers) to optimality in less than half a second, with only a minor average optimality gap of 0.15\%. Several managerial insights and policy implications are also presented regarding using drones to boost the efficiency of traditional, ground-based delivery vehicles.

%Regarding the benefits of using multi-drop drones for last-mile logistics, our scenario analysis indicates that the attainable time savings strongly depend on the maximum number of drops and the speed ratio between the truck and the drone. For both parameters, we observe diminishing marginal returns. This implies that investing in drones that can travel at a moderate speed (e.g., twice the speed of a traditional delivery truck) and can perform two or three drops, such as the A2Z Drone \citep{Sakharkar2021A2ZDrone} or the Wingcopter 198 \citep{Wingcopter2023Wingcopter198}, can lead to significant time savings. For a uniform customer distribution and sufficient flight endurance levels, the time savings of adding a drone that travels at twice the speed of the truck range between 33\% savings for a single-drop drone and 56\% for a drone capable of up to ten drops. We further find that changes to the operational environment, i.e., non-uniform distribution or lower customer density, require significantly enhanced drone capabilities. Consequently, it is imperative to carefully assess the suitability of a drone for the particular application at hand.

There are several potential avenues for future research. For instance, \ac{MD-SPP-H} could be extended to allow the truck to launch and recover the drone at locations other than customer nodes (see, e.g., \cite{Schermer2019AOperations}). Further, the \ac{FSTSP-MD} assumes a constant drone flight speed and a flight endurance independent of carrying weight. Thus, future research could enrich our \ac{MD-SPP} formulation to account for more realistic drone parameters (see, e.g., \cite{Jeong2019Truck-droneZones} and \cite{Raj2020TheSpeeds}). Finally, a natural extension is to consider multiple trucks equipped with multiple drones. For instance, future research could propose another split algorithm to accommodate multiple drones and use our \ac{MD-SPP-H} framework to efficiently explore the TSP solution space. Since simultaneously launching and retrieving multiple drones entails significant challenges, particularly with respect to air traffic management (see, e.g., \cite{Murray2020TheDrones}), this extension could be interesting when considering pairing the delivery truck with other more conventional vehicles, such as cargo bikes.

% -----------------------------------
\begin{acronym}
\acro{IGH}{Iterated Greedy Heuristic}
\acro{MD-SPP}{\textit{Multi-Drop Shortest Path Problem}}
\acro{MD-SPP-H}{\textit{Multi-Drop Shortest Path Problem--Based Heuristic}}
\acro{FSTSP}{Flying Sidekick Traveling Salesman Problem}
\acro{FSTSP-MD}{\textit{Flying Sidekick Traveling Salesman Problem with Multiple Drops}}
\acro{SPP}{Shortest Path Problem}
\acro{TSP}{Traveling Salesman Problem}
\acro{MILP}{Mixed-Integer Linear Programming}
\acro{ALNS}{Adaptive Large Neighborhood Search}
\acro{EP-All}{Exact Partitioning -- All (three local searches)}
\acro{SPP-All}{Shortest Path Problem -- All (three local searches)}
\end{acronym}
% \newpage
% \section*{Acknowledgment}
% \label{sec:acknowledgment}

\begin{comment}

The authors wish to thank ...

Furthermore, the authors would like to thank the three anonymous referees for their valuable comments and corrections.

\end{comment}

% -----------------------------------
% \pagebreak
\bibliography{./Literature/references.bib}

% -----------------------------------
% \pagebreak
\begin{appendix}
% \newpage
\section{Proof of Propositions} \label{app:propositions}

\paragraph{Complexity of split algorithms} 

To find an optimal \ac{FSTSP} solution for a given TSP tour, the split algorithm of \cite{Kundu2021AnProblem} solves the shortest path problem in a weighted directed acyclic graph. Specifically, given the TSP tour $\sigma = (0,1,..., n, n+1)$ (whose nodes have been re-indexed based on their positions), the authors build the graph $\mathcal{G} = (\mathcal{N}, \mathcal{A})$, where $\mathcal{N} = \{0,1,..., n, n+1\}$ and $\mathcal{A} = \{(i,j): i \in \mathcal{N}\backslash \{n+1\} \land j \in \{i+1,...,n+1\}\}$. Each arc $(i,j) \in \mathcal{A}$ is called a \emph{forward-moving arc}. To solve the shortest path problem, the authors enumerate all forward-moving arc times associated with each $(i,j) \in \mathcal{A}$. Notably, each $(i,j) \in \mathcal{A}$ has associated a maximum of $j-i$ forward-moving arc times, defined by the maximum number of single-drop drone operations between nodes $i$ and $j$ (including the one where the drone is carried by the truck). As \cite{Kundu2021AnProblem} prove, the complete enumeration of all forward-moving arc times can be done in $\mathcal{O}(n^3)$ time for the \ac{FSTSP}. The interested reader can refer to \cite{Kundu2021AnProblem} for further details.

For the \ac{FSTSP-MD}, Proposition \ref{prop:complexity_extended_approach} states that, without the introduction of partition nodes, solving the shortest path problem by enumerating all forward-moving arc times requires $\mathcal{O}(2^{n})$ time.

\begin{proposition} \label{prop:complexity_extended_approach}
    For a given TSP tour, solving the shortest path problem by enumerating all forward-moving arcs requires $\mathcal{O}(2^{n})$ time.
\end{proposition}

\begin{proof}
    Consider the TSP tour $\sigma = (0,1,..., n, n+1)$, whose nodes have been re-indexed based on their positions. Assume that the maximum drone flight endurance constraint is removed. Now consider a subsequence $\sigma'_{ij} = (i, i+1, ..., j-1, j)$ of the TSP tour $\sigma$, where $i$ is the launch node and $j$ is the recovery node. The total number of operations where the drone performs exactly $m$ deliveries is given by ${j-i-1 \choose m}$. Therefore, when we relax the constraint for the maximum number of drone drops, the maximum number of possible drone operations (including the one where the drone is carried by the truck throughout the entire subsequence) is given by $\sum_{m=0}^{j-i-1} {j-i-1 \choose m} $ = $2^{j-i-1} $ (representing the total number of forward-moving arc times between nodes $i$ and $j$). 
    
    Now, since the TSP tour $\sigma$ has $n+2$ nodes, there is a maximum of $\sum_{i=0}^{n}\sum_{j=i+1}^{n+1} 2^{j-i-1} = 2^{n+2}-n-3 $ possible forward-moving arc times (i.e., coordinated truck-and-drone routes) for the given TSP tour $\sigma$. Therefore, a split algorithm without the introduction of partition nodes requires $\mathcal{O}(2^{n})$ time.
\end{proof}

\paragraph{Proof of Proposition \ref{proposition}} Let us introduce Definition \ref{def:optimal_partition_node} and Lemmas \ref{lemma_upper_bound} and \ref{lemma_relaxation} before proving Proposition \ref{proposition}.

\begin{definition} \label{def:optimal_partition_node}
    (Optimal partition node) Let $\sigma = (0,1,..., n, n+1)$ be a TSP tour whose nodes have been re-indexed based on their positions. Given a subsequence $\sigma'_{ij} = (i, i+1, ..., j-1, j)$ of $\sigma$, we refer to node $k^* \in N$ as an \emph{optimal partition node} if $k^* \in \argmin_k\{c_{i,j}^k: i \leq k <j \land k \leq i+D \land c_{i,j}^k \leq E \}$. Therefore, $c_{i,j} = c_{i,j}^{k^*}$.
\end{definition}

\begin{lemma} \label{lemma_upper_bound}
    (Upper bound) Assume that Algorithm \ref{alg:1} is applied to the TSP solution $\sigma = (0,1,..., n, n+1)$. At any time of the run, for each $j \in \{0,..., n+1\}$, we always have $T_{j} \geq \delta(0, j)$, where $\delta(0,j)$ represents the total time of the true shortest path from node $0$ to node $j$. If we ever find $T_{j} = \delta(0,j)$, the value of $T_{j}$ never changes in subsequent iterations.
\end{lemma}
    
\begin{proof}
    We proceed by induction on the number of relaxation steps $r$.

    \textit{Base case.} In the initialization step (i.e., when $r=0$), we have that $T_{0}=0$ and $T_{j}=\infty$, for all $j \in \{1,..., n+1\}$ (see Step 1 in Algorithm \ref{alg:1}). Now, since $\delta(0,0) = 0$, we have that $T_{0} = \delta(0,0)$. Further, if node $j \in \{1,..., n+1\}$ is reachable from node $0$, then $\delta(0,j)$ will have a certain value; therefore, $T_{j} \geq \delta(0,j)$.

    \textit{Induction step.} Given $j \in \{1,..., n+1\}$, assume that $T_{j} \geq \delta(0,j)$ after $r-1$ relaxation steps. Let us now consider the $r^\text{th}$ relaxation step on the arc $(i,j)$, where $0 \leq i < j$. In this relaxation step, $T_{j}$ is updated to $T_{j} = T_{i} + c_{i,j}$, where $c_{i,j}$ represents the coordinated truck-and-drone time over the subsequence $(i,..., j)$ with the optimal partition. Since we already assume that the Lemma holds for the $r-1$ relaxation step, we have that $T_{i} \geq \delta(0,i)$. Therefore, $T_{j} = T_{i} + c_{i,j} \geq \delta(0,i) + c_{i,j}$. Further, since any shortest path satisfies triangular inequality, we have that $\delta(0,i) + c_{i,j} \geq \delta(0,j)$. Consequently, $T_{j} \geq \delta(0,i) + c_{i,j} \geq \delta(0,j)$.

    Note that, when we reach $T_{j} = \delta(0,j)$, we cannot reduce the value any further, as we just showed $T_{j} \geq \delta(0,j)$. Consequently, in subsequent relaxations, the value of $T_{j}$ remains unchanged.
\end{proof}

\begin{lemma} \label{lemma_relaxation}
    (Sequential relaxations) Assume that Algorithm \ref{alg:1} was solved for the TSP solution $\sigma = (0,1,..., n, n+1)$. Given a node $r \in \{1,..., n+1\}$, let $P_r = \{(\sigma'_{i_0,i_1}, k_0),...,(\sigma'_{i_{m-1},i_m}, k_{m-1})\}$ be the collection of $m$ tuples representing the shortest path from node $0$ to node $r$ (where node $i_m = r$). The sequential relaxation steps of the subsequences of $P_r$ produce $T_{i_m} = \delta(0, i_m)$, that is, $T_{r} = \delta(0, r)$.    
\end{lemma}

\begin{proof}
    We proceed by induction on the $m^\text{th}$ tuple of $P_r$, showing that after the sequential relaxations of the $m$ subsequences $\sigma'_{i_0,i_1},...,\sigma'_{i_{m-1},i_m}$, we have that $T_{i_m} = \delta(0, i_m)$.
    
    \textit{Base case.} If $m = 1$, then $P_r$ = $(\sigma'_{i_0,i_1}, k_0)$, with $\sigma'_{i_0,i_1} = (0,...,r)$. We know that $T_{i_m} \geq \delta(0,i_m)$ from Lemma \ref{lemma_upper_bound}. We also know that, by definition, node $k_0$ is an optimal partition node for subsequence $\sigma'_{i_0,i_1}$. Since we have only one tuple representing the shortest path, then $T_{i_m}$ = $\delta(0,i_m)$ after the relaxation step.
        
    \textit{Induction step.} Assume that the first $m-1$ subsequences $\sigma'_{i_0,i_1},...,\sigma'_{i_{m-2},i_{m-1}}$ have been relaxed sequentially based on the given order. 
    From Lemma \ref{lemma_upper_bound}, we have that $T_{i_{m-1}}$ = $\delta(0,i_{m-1})$ at the end of the $m-1$ relaxations. 
    The subsequence $\sigma'_{i_{m-1},i_m}$ is only relaxed after the relaxation of the previous subsequence $\sigma'_{i_{m-2},i_{m-1}}$. The definition of optimal partition node ensures that the relaxation of the subsequence $\sigma'_{i_{m-1},i_m}$ is optimal, and by Lemma \ref{lemma_upper_bound}, at the time of this call, we have that $T_{i_m} \geq \delta(0,i_m)$. 
    
    Since $P_r$ is the shortest path containing the subsequence $\sigma'_{i_{m-1},i_m}$, we also have that $\delta(0,i_m)$ = $\delta(0,i_{m-1})$ + $\delta(i_{m-1},i_{m})$. Therefore, $T_{i_{m}} = \delta(0,i_{m-1})$ + $\delta(i_{m-1},i_{m})$ = $\delta(0,i_{m})$. The value of $T_{i_m}$ will remain unchanged in subsequent relaxation steps due to Lemma \ref{lemma_upper_bound}.
\end{proof}

We now prove Proposition \ref{proposition}, i.e., Algorithm \ref{alg:1} solves Problem \ref{prob:SPP} in $\mathcal{O}(n^3)$ time for a given TSP tour.

\begin{proof}[Proof]
     Let $\sigma = (0,1,..., n, n+1)$ be a TSP tour whose nodes have been re-indexed based on their positions. For each subsequence $\sigma'_{ij} =(i, i+1, ..., j-1, j)$ of $\sigma$, we have a maximum of $j-i$ potential drone operations defined by the number of possible partition nodes (including the operation where the drone is carried by the truck throughout the entire subsequence). Thus, the maximum number of drone operations is $\sum_{i=0}^{n}\sum_{j=i+1}^{n+1} j - i = \frac{(n+1)(n+2)(n+3)}{6}$. Consequentially, Algorithm \ref{alg:1} has a $\mathcal{O}(n^3)$ complexity.

     Now, since the TSP tour $\sigma = (0, 1,..., n, n+1)$ is topologically sorted, we process it in order of their subsequences. By Lemma \ref{lemma_relaxation}, the collection of tuples $P$ obtained by Algorithm \ref{alg:1} represents the shortest path for $\sigma$, and Lemma \ref{lemma_upper_bound} ensures that $T_{n+1} = \delta(0,n+1)$.  
\end{proof}

\begin{corollary}
    For a given TSP tour $\sigma$, the shortest path has at most $n+1$ tuples, i.e., $|P| \leq n+1$. When $|P|=n+1$, the collection of tuples corresponds to the TSP tour itself (i.e., the truck makes all the deliveries, with the drone being idle).
\end{corollary}

\section{Supplementary Results of Section \ref{sec:heuristic_performance}} \label{app:supplementary_results_performance}

Table \ref{tab:ExperimentSetting} summarizes the experiment setting used to compare \ac{MD-SPP-H} with the benchmark solution approaches developed for the \ac{FSTSP-MD} and \ac{FSTSP}.

\begin{table}[htbp]
    \centering
    \renewcommand{\arraystretch}{1.3}
    \caption{Summary of experiment setting used to compare \ac{MD-SPP-H} benchmark solutions approaches.}
    \resizebox{\textwidth}{!}{
    \begin{tabular}{rlrllll}
    \toprule
    \multicolumn{1}{l}{\multirow{2}[4]{*}{ Type  }} & \multirow{2}[4]{*}{Parameter } &       & \multicolumn{2}{c}{\textbf{\ac{FSTSP-MD}}} &       & \multicolumn{1}{c}{\multirow{2}[4]{*}{\textbf{\ac{FSTSP}}}} \\
\cmidrule{4-5}          &       &       & \multicolumn{1}{c}{\ac{IGH}} & \multicolumn{1}{c}{\ac{ALNS}} &       &  \\
    \midrule
    \multicolumn{1}{l}{\multirow{3}[2]{*}{General}} & Total number of instances   &       & 150   & 60    &       & 150 \\
          & Customer distribution  &       &   Uniform, single-and double-center  &   Uniform (B-G)  &       &  Uniform, single- and double-center  \\
          & Number of customers   &       &   $\{50,75,100,175,250\}$   &   $\{50,100\}$   &       &   $\{10, 11, ..., 20, 50,75,100,175,250\}$   \\
    \midrule
    \multicolumn{1}{l}{\multirow{8}[2]{*}{Vehicles}} & Drone eligibility   &       & 100\% & 80\%  &       & 100\% \\
          & Truck distance   &       &   Euclidean [unit]  &   Manhattan [km]  &       &   Euclidean [unit]  \\
          & Drone distance  &       &   Euclidean [unit]  &   Euclidean [km]  &       &   Euclidean [unit]  \\
          & Truck speed   &       & 1     &   40 km/h  &       & 1 \\
          & Drone speed ratio  &       & $\{1,2,3\}$   &   40 km/h  &       & 2 \\
          & Flight endurance  &       &   $E = \frac{2}{n\cdot(n+1)}\sum_{(i,j)\in A} d_{i,j}$   &   24 min    &       & $\infty$   \\
          & Number of drops   &       &   $\infty$   &   $\infty$   &       & 1 \\
          & Launch/retrieval time   &       &   0/0  &   30s/40s    &       &   0/0   \\
    \midrule
    \multicolumn{1}{l}{\multirow{2}[2]{*}{CPU}} & Run-time limit   &       &   600s   &   60s ($n=50$), 90s ($n = 100$)   &       &   --  \\
          & Runs  &       & 10    & 10    &       &   10 (for \ac{MD-SPP-H})  \\
    \bottomrule
    \end{tabular}%
    }
  \label{tab:ExperimentSetting}%
\end{table}%

\paragraph{\ac{MD-SPP-H} convergence results}

Given a \ac{FSTSP-MD} instance, we compute the following metric to understand the convergence of \ac{MD-SPP-H} over time: 
\begin{equation*}
        \text{Difference } (\%) = \frac{z^\text{MD-SPP-H}_\text{t} - z^\text{MD-SPP-H}_\text{best}}{z^\text{MD-SPP-H}_\text{best}} \times 100,
    \end{equation*}
where $z^\text{MD-SPP-H}_\text{best}$ is the value of the best solution found by \ac{MD-SPP-H} and $z^\text{MD-SPP-H}_\text{t}$ is the value of the best-so-far solution at time step $t$. Figure \ref{fig:FigBenchmarkFSTSPImpTrajectoryBest} reports the percentage difference between the best solution found so far and the best overall solution found by \ac{MD-SPP-H} (each line denotes one instance and the values are those achieved in the best of the ten runs). For instances with 100 customers, results show that the solutions found by \ac{MD-SPP-H} after one minute are within 2\% of the best overall solutions. Further, for 250-customer instances, the solutions found after 10 minutes are within 4\% of the best overall solutions.

\begin{figure}[ht]
    \centering
    
    \begin{subfigure}{0.7\textwidth}
        \centering
        
        \includegraphics[width=\linewidth]{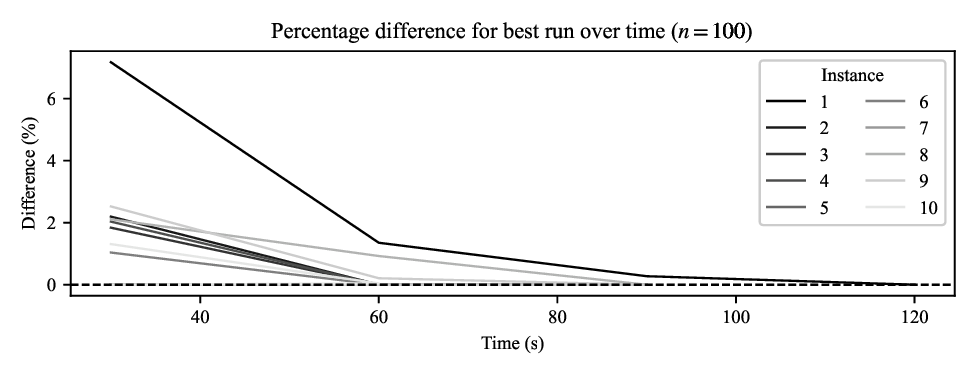}
        %\label{fig:sub1}
    \end{subfigure}
    \vspace{-0.5cm}
    \begin{subfigure}{0.7\textwidth}
        \centering
        \vspace{-0.5cm}
        \includegraphics[width=\linewidth]{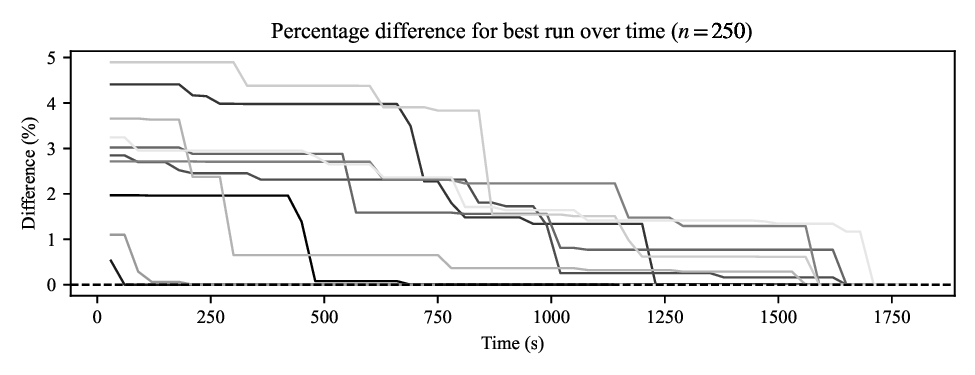}
        %\label{fig:sub2}
    \end{subfigure}
    
    \caption{Average percentage difference over time for instances with $n= 100$ and $n=250$ customers (tracking solutions after 30 seconds)} 
    \label{fig:FigBenchmarkFSTSPImpTrajectoryBest}
\end{figure}

\paragraph{Comparison with the integer L-shaped method of \cite{Vasquez2021AnDecomposition}}

Table \ref{tab:comparison_vasquez} shows the comparison with the integer L-shaped method of \cite{Vasquez2021AnDecomposition}. The authors provided us with their result data using a maximum runtime of one hour. For the integer L-shaped method, Column ``\# Optimal” reports the number of instances solved to optimality. For the \ac{MD-SPP-H}, we report the results achieved by \ac{MD-SPP-H} for the best of the ten runs. Column ``\# Optimal” reports the number of instances where \ac{MD-SPP-H} finds the same optimal value as the integer L-shaped method; and Column ``\# Worse” (resp., \# Better) details the number of solutions where \ac{MD-SPP-H} achieves a worse (resp., better) solution than the integer L-shaped method. Further, considering those instances where the integer L-shaped method does not prove optimality, Column ``\# Equal” shows the number of solutions for which the integer L-shaped method and \ac{MD-SPP-H} find the same value. Column ``Opt. gap (\%)” is the average optimality gap reached by \ac{MD-SPP-H} over all instances solved to optimality by the integer L-shaped method; Column ``Diff (\%)” shows the average percentage difference between the best-known value achieved by the integer L-shaped method (i.e., the upper bound given by the solver) and the value of \ac{MD-SPP-H} over those instances not solved to optimality by the integer L-shaped method (a negative value indicates that \ac{MD-SPP-H} outperforms the integer L-shaped method); and Column ``Run time (s)” details the average run time (over 10 instances) needed by \ac{MD-SPP-H}.

For instances with up to 17 customers, for which the integer L-shaped method finds the optimal solution for all instances, Table \ref{tab:comparison_vasquez} shows that the \ac{MD-SPP-H} reaches the optimal solution for most of the instances in a very short run time of less than 0.2 seconds. The optimality gap over those instances is below or equal to 0.4\%. 
For instances with 18 and up to 20 customers, \ac{MD-SPP-H} performs worse only for 3 out of 30 instances. For most instances, \ac{MD-SPP-H} reaches the same solution value as the integer L-shaped method or even improves the solution. On average, \ac{MD-SPP-H} improves the solution quality by 1.4\% for instances that are not solved to optimality by the integer L-shaped method; see Column Diff (\%). 

\begin{table}[ht]
  \centering
  \renewcommand{\arraystretch}{1.2}
  \caption{Comparison between \ac{MD-SPP-H} and the integer L-shaped method of \cite{Vasquez2021AnDecomposition}.}
  \scalebox{0.7}{
    \begin{tabular}{crrrrrrrrrrr}
    \toprule
    \multirow{2}[3]{*}{Customers} & \multicolumn{1}{c}{\multirow{2}[3]{*}{Instances}} &       & \multicolumn{1}{c}{\textbf{L-shaped}} &       & \multicolumn{7}{c}{\textbf{MD-SPP-H}} \\
\cmidrule{4-4}\cmidrule{6-12}          &       &       & \multicolumn{1}{c}{\# Optimal} &       & \multicolumn{1}{c}{\# Optimal} & \multicolumn{1}{c}{\# Worse} & \multicolumn{1}{c}{\# Better} & \multicolumn{1}{c}{\# Equal} & \multicolumn{1}{c}{Opt. gap (\%)} & \multicolumn{1}{c}{Diff (\%)} & \multicolumn{1}{c}{Run time (s)} \\
    \midrule
    \multicolumn{1}{r}{10} & 10    &       & 10    &       & 8     & 2     & 0     & 0     & 0.10  & \--    & 0.04 \\
    \multicolumn{1}{r}{11} & 10    &       & 10    &       & 7     & 3     & 0     & 0     & 0.40  & \--    & 0.05 \\
    \multicolumn{1}{r}{12} & 10    &       & 10    &       & 7     & 3     & 0     & 0     & 0.23  & \--    & 0.06 \\
    \multicolumn{1}{r}{13} & 10    &       & 10    &       & 9     & 1     & 0     & 0     & 0.33  & \--    & 0.09 \\
    \multicolumn{1}{r}{14} & 10    &       & 10    &       & 9     & 1     & 0     & 0     & 0.05  & \--    & 0.10 \\
    \multicolumn{1}{r}{15} & 10    &       & 10    &       & 8     & 2     & 0     & 0     & 0.18  & \--    & 0.12 \\
    \multicolumn{1}{r}{16} & 10    &       & 10    &       & 9     & 1     & 0     & 0     & 0.12  & \--    & 0.15 \\
    \multicolumn{1}{r}{17} & 10    &       & 10    &       & 10    & 0     & 0     & 0     & 0.00  & \--    & 0.18 \\
    \multicolumn{1}{r}{18} & 10    &       & 6     &       & 4     & 2     & 2     & 2     & 0.25  & -1.61 & 0.28 \\
    \multicolumn{1}{r}{19} & 10    &       & 1     &       & 1     & 1     & 6     & 2     & 0.00  & -1.08 & 0.37 \\
    \multicolumn{1}{r}{20} & 10    &       & 1     &       & 1     & 0     & 8     & 1     & 0.00  & -1.51 & 0.43 \\
    \midrule
    \textbf{All} & \textbf{110} &       & \textbf{88} &       & \textbf{73} & \textbf{16} & \textbf{16} & \textbf{5} & \textbf{0.15} & \textbf{-1.40} & \textbf{0.17} \\
    \bottomrule
    \end{tabular}%
    }
  \label{tab:comparison_vasquez}%
\end{table}%

\clearpage  
\pagenumbering{gobble}

\appendix
\setcounter{section}{0} 
\renewcommand{\thesection}{Online Appendix \Alph{section}}

\section{Comparison with solution approaches for the FSTSP-MD} \label{app:onlieA}

\paragraph{Comparison with \ac{IGH} over single- and double-center instances} Tables 
\ref{tab:TabBenchmarkGonazelezSingleCenterFull} and \ref{tab:TabBenchmarkGonazelezDoubleCenterFull} report the results comparing \ac{MD-SPP-H} and the \ac{IGH} of \cite{Gonzalez-R2020Truck-dronePlanning}, considering the single- and double-center customer distributions and at least 50 customers. We observe similar, yet slightly reduced improvements for the two non-uniform distributions compared to the uniform distribution.

\begin{table}[htbp]
\centering
\caption{Results of the comparison of \ac{MD-SPP-H} and \ac{IGH} (based on the solution published by \citet{Gonzalez-R2020Truck-dronePlanning}), considering $n=50$ up to $250$ single-center distributed customers.}
\renewcommand{\arraystretch}{1.1}
\resizebox{\textwidth}{!}{
\begin{tabular}{ccrrrrrrrrrrrrrrrrrrrrrrrrrrr}
\toprule
\multirow{3}[6]{*}{$n$} & \multirow{3}[6]{*}{ID} &       & \multicolumn{8}{c}{Speed ratio 1}                             &       & \multicolumn{8}{c}{Speed ratio 2}                             &       & \multicolumn{8}{c}{Speed ratio 3} \\
\cmidrule{4-11}\cmidrule{13-20}\cmidrule{22-29}      &       &       & \multicolumn{2}{c}{MD-SPP-H} &       & \multicolumn{2}{c}{IGH} &       & \multicolumn{2}{c}{$\Delta$} &       & \multicolumn{2}{c}{MD-SPP-H} &       & \multicolumn{2}{c}{IGH} &       & \multicolumn{2}{c}{$\Delta$} &       & \multicolumn{2}{c}{MD-SPP-H} &       & \multicolumn{2}{c}{IGH} &       & \multicolumn{2}{c}{$\Delta$} \\
\cmidrule{4-5}\cmidrule{7-8}\cmidrule{10-11}\cmidrule{13-14}\cmidrule{16-17}\cmidrule{19-20}\cmidrule{22-23}\cmidrule{25-26}\cmidrule{28-29}      &       &       & \multicolumn{1}{c}{$z_\text{best}$} & \multicolumn{1}{c}{$z_\text{avg}$} &       & \multicolumn{1}{c}{$z_\text{best}$} & \multicolumn{1}{c}{$z_\text{avg}$} &       & \multicolumn{1}{c}{$\Delta_\text{best}$} & \multicolumn{1}{c}{$\Delta_\text{avg}$} &       & \multicolumn{1}{c}{$z_\text{best}$} & \multicolumn{1}{c}{$z_\text{avg}$} &       & \multicolumn{1}{c}{$z_\text{best}$} & \multicolumn{1}{c}{$z_\text{avg}$} &       & \multicolumn{1}{c}{$\Delta_\text{best}$} & \multicolumn{1}{c}{$\Delta_\text{avg}$} &       & \multicolumn{1}{c}{$z_\text{best}$} & \multicolumn{1}{c}{$z_\text{avg}$} &       & \multicolumn{1}{c}{$z_\text{best}$} & \multicolumn{1}{c}{$z_\text{avg}$} &       & \multicolumn{1}{c}{$\Delta_\text{best}$} & \multicolumn{1}{c}{$\Delta_\text{avg}$} \\
\midrule
\multirow{10}[2]{*}{50} & 71    &       & 516.8 & 528.8 &       & 571.8 & 618.3 &       & 9.6   & 14.5  &       & 350.1 & 353.6 &       & 370.9 & 371.5 &       & 5.6   & 4.8   &       & 275.5 & 278.3 &       & 292.6 & 292.6 &       & 5.8   & 4.9 \\
      & 72    &       & 608.3 & 619.2 &       & 744.8 & 793.5 &       & 18.3  & 22.0  &       & 413.4 & 416.9 &       & 501.0 & 505.6 &       & 17.5  & 17.5  &       & 375.0 & 378.2 &       & 445.5 & 449.8 &       & 15.8  & 15.9 \\
      & 73    &       & 454.2 & 457.3 &       & 492.7 & 535.0 &       & 7.8   & 14.5  &       & 319.4 & 323.3 &       & 348.4 & 348.4 &       & 8.3   & 7.2   &       & 254.2 & 256.7 &       & 315.4 & 315.6 &       & 19.4  & 18.7 \\
      & 74    &       & 626.5 & 631.1 &       & 758.3 & 792.1 &       & 17.4  & 20.3  &       & 428.1 & 432.0 &       & 490.5 & 490.5 &       & 12.7  & 11.9  &       & 357.8 & 360.1 &       & 446.0 & 448.1 &       & 19.8  & 19.6 \\
      & 75    &       & 791.0 & 806.2 &       & 830.8 & 857.6 &       & 4.8   & 6.0   &       & 616.7 & 621.5 &       & 648.5 & 654.3 &       & 4.9   & 5.0   &       & 561.7 & 571.9 &       & 629.3 & 629.3 &       & 10.7  & 9.1 \\
      & 76    &       & 663.5 & 672.7 &       & 749.9 & 812.6 &       & 11.5  & 17.2  &       & 442.0 & 443.9 &       & 527.0 & 528.3 &       & 16.1  & 16.0  &       & 376.5 & 378.1 &       & 439.1 & 447.2 &       & 14.2  & 15.5 \\
      & 77    &       & 656.6 & 671.5 &       & 700.7 & 749.6 &       & 6.3   & 10.4  &       & 511.0 & 513.9 &       & 578.8 & 580.3 &       & 11.7  & 11.5  &       & 551.8 & 556.9 &       & 543.8 & 543.9 &       & -1.5  & -2.4 \\
      & 78    &       & 733.1 & 737.2 &       & 791.1 & 807.6 &       & 7.3   & 8.7   &       & 504.2 & 506.2 &       & 529.0 & 536.6 &       & 4.7   & 5.7   &       & 498.4 & 502.0 &       & 466.6 & 478.4 &       & -6.8  & -4.9 \\
      & 79    &       & 634.1 & 638.2 &       & 663.2 & 690.2 &       & 4.4   & 7.5   &       & 535.2 & 536.7 &       & 546.0 & 548.0 &       & 2.0   & 2.1   &       & 511.1 & 513.8 &       & 499.9 & 501.4 &       & -2.2  & -2.5 \\
      & 80    &       & 731.6 & 746.0 &       & 839.8 & 889.4 &       & 12.9  & 16.1  &       & 558.1 & 563.4 &       & 657.9 & 667.8 &       & 15.2  & 15.6  &       & 476.5 & 478.3 &       & 619.0 & 621.0 &       & 23.0  & 23.0 \\
\midrule
\multirow{10}[2]{*}{75} & 81    &       & 919.7 & 939.9 &       & 1103.6 & 1139.8 &       & 16.7  & 17.5  &       & 686.1 & 700.5 &       & 742.4 & 752.5 &       & 7.6   & 6.9   &       & 567.5 & 579.6 &       & 678.4 & 680.0 &       & 16.3  & 14.8 \\
      & 82    &       & 686.1 & 692.2 &       & 870.3 & 905.5 &       & 21.2  & 23.6  &       & 470.2 & 483.4 &       & 556.8 & 569.9 &       & 15.6  & 15.2  &       & 380.1 & 386.7 &       & 444.2 & 453.8 &       & 14.4  & 14.8 \\
      & 83    &       & 742.4 & 756.3 &       & 919.6 & 951.4 &       & 19.3  & 20.5  &       & 516.5 & 520.5 &       & 647.7 & 654.0 &       & 20.2  & 20.4  &       & 486.4 & 489.8 &       & 545.2 & 547.2 &       & 10.8  & 10.5 \\
      & 84    &       & 762.5 & 780.8 &       & 924.8 & 967.6 &       & 17.5  & 19.3  &       & 540.3 & 551.2 &       & 602.3 & 604.1 &       & 10.3  & 8.8   &       & 477.7 & 487.8 &       & 493.3 & 497.5 &       & 3.2   & 1.9 \\
      & 85    &       & 765.9 & 773.4 &       & 896.8 & 917.6 &       & 14.6  & 15.7  &       & 575.8 & 587.7 &       & 643.0 & 646.5 &       & 10.4  & 9.1   &       & 500.8 & 511.3 &       & 558.6 & 560.3 &       & 10.4  & 8.8 \\
      & 86    &       & 939.2 & 957.5 &       & 1056.5 & 1089.0 &       & 11.1  & 12.1  &       & 732.2 & 745.7 &       & 753.7 & 759.1 &       & 2.8   & 1.8   &       & 666.0 & 672.4 &       & 652.6 & 652.8 &       & -2.0  & -3.0 \\
      & 87    &       & 850.2 & 866.1 &       & 971.2 & 1012.2 &       & 12.5  & 14.4  &       & 651.5 & 667.4 &       & 632.2 & 637.8 &       & -3.0  & -4.6  &       & 593.9 & 600.0 &       & 538.2 & 538.2 &       & -10.4 & -11.5 \\
      & 88    &       & 965.9 & 984.5 &       & 1037.7 & 1083.6 &       & 6.9   & 9.1   &       & 790.1 & 805.6 &       & 861.2 & 867.4 &       & 8.3   & 7.1   &       & 773.5 & 776.3 &       & 800.2 & 801.1 &       & 3.3   & 3.1 \\
      & 89    &       & 746.5 & 749.2 &       & 915.5 & 940.4 &       & 18.5  & 20.3  &       & 519.6 & 528.9 &       & 623.0 & 628.6 &       & 16.6  & 15.9  &       & 469.2 & 474.4 &       & 567.5 & 568.6 &       & 17.3  & 16.6 \\
      & 90    &       & 823.9 & 839.8 &       & 1034.6 & 1071.7 &       & 20.4  & 21.6  &       & 576.1 & 585.8 &       & 713.3 & 724.3 &       & 19.2  & 19.1  &       & 487.3 & 492.6 &       & 613.8 & 616.0 &       & 20.6  & 20.0 \\
\midrule
\multirow{10}[2]{*}{100} & 91    &       & 989.9 & 1009.1 &       & 1175.6 & 1205.6 &       & 15.8  & 16.3  &       & 725.4 & 740.6 &       & 803.8 & 810.1 &       & 9.8   & 8.6   &       & 732.5 & 735.6 &       & 738.5 & 741.7 &       & 0.8   & 0.8 \\
      & 92    &       & 1031.3 & 1038.0 &       & 1276.7 & 1318.0 &       & 19.2  & 21.2  &       & 715.1 & 739.0 &       & 862.0 & 870.1 &       & 17.0  & 15.1  &       & 590.3 & 595.1 &       & 764.2 & 764.2 &       & 22.8  & 22.1 \\
      & 93    &       & 1056.5 & 1065.9 &       & 1258.6 & 1309.1 &       & 16.1  & 18.6  &       & 762.9 & 774.4 &       & 912.4 & 921.9 &       & 16.4  & 16.0  &       & 704.8 & 716.2 &       & 853.6 & 853.6 &       & 17.4  & 16.1 \\
      & 94    &       & 1109.6 & 1126.2 &       & 1318.6 & 1385.8 &       & 15.9  & 18.7  &       & 776.7 & 793.5 &       & 880.4 & 883.7 &       & 11.8  & 10.2  &       & 795.1 & 807.5 &       & 756.4 & 758.5 &       & -5.1  & -6.5 \\
      & 95    &       & 852.6 & 859.6 &       & 1025.2 & 1090.3 &       & 16.8  & 21.2  &       & 633.5 & 637.4 &       & 727.7 & 735.4 &       & 12.9  & 13.3  &       & 529.8 & 537.6 &       & 603.2 & 609.0 &       & 12.2  & 11.7 \\
      & 96    &       & 982.7 & 998.1 &       & 1191.7 & 1254.1 &       & 17.5  & 20.4  &       & 687.8 & 704.3 &       & 822.8 & 828.6 &       & 16.4  & 15.0  &       & 607.0 & 615.7 &       & 719.0 & 720.6 &       & 15.6  & 14.6 \\
      & 97    &       & 1094.4 & 1106.3 &       & 1240.1 & 1259.2 &       & 11.7  & 12.1  &       & 808.5 & 818.2 &       & 846.1 & 852.6 &       & 4.4   & 4.0   &       & 698.6 & 707.9 &       & 772.3 & 776.1 &       & 9.5   & 8.8 \\
      & 98    &       & 1119.6 & 1140.2 &       & 1337.3 & 1357.5 &       & 16.3  & 16.0  &       & 947.0 & 957.9 &       & 1060.2 & 1064.8 &       & 10.7  & 10.0  &       & 876.3 & 886.6 &       & 1002.2 & 1016.6 &       & 12.6  & 12.8 \\
      & 99    &       & 787.5 & 798.3 &       & 999.4 & 1038.9 &       & 21.2  & 23.2  &       & 551.3 & 565.2 &       & 632.2 & 643.3 &       & 12.8  & 12.1  &       & 481.4 & 485.8 &       & 554.0 & 558.9 &       & 13.1  & 13.1 \\
      & 100   &       & 821.9 & 843.4 &       & 1120.8 & 1152.4 &       & 26.7  & 26.8  &       & 577.4 & 584.1 &       & 718.9 & 726.5 &       & 19.7  & 19.6  &       & 444.3 & 460.2 &       & 633.4 & 635.1 &       & 29.9  & 27.5 \\
\midrule
\multirow{10}[2]{*}{175} & 101   &       & 1261.7 & 1267.7 &       & 1590.1 & 1614.3 &       & 20.7  & 21.5  &       & 903.4 & 916.4 &       & 987.7 & 997.9 &       & 8.5   & 8.2   &       & 793.1 & 822.9 &       & 867.4 & 868.2 &       & 8.6   & 5.2 \\
      & 102   &       & 1232.0 & 1240.0 &       & 1509.7 & 1567.8 &       & 18.4  & 20.9  &       & 817.2 & 827.7 &       & 943.6 & 956.6 &       & 13.4  & 13.5  &       & 676.9 & 680.3 &       & 815.0 & 818.9 &       & 17.0  & 16.9 \\
      & 103   &       & 1213.7 & 1224.1 &       & 1547.1 & 1609.4 &       & 21.5  & 23.9  &       & 860.2 & 868.3 &       & 955.1 & 961.3 &       & 9.9   & 9.7   &       & 791.6 & 798.8 &       & 794.7 & 806.9 &       & 0.4   & 1.0 \\
      & 104   &       & 1222.7 & 1228.1 &       & 1580.5 & 1603.6 &       & 22.6  & 23.4  &       & 853.1 & 861.9 &       & 1036.4 & 1045.9 &       & 17.7  & 17.6  &       & 709.9 & 734.1 &       & 890.1 & 893.3 &       & 20.2  & 17.8 \\
      & 105   &       & 1331.1 & 1338.0 &       & 1597.0 & 1618.8 &       & 16.6  & 17.3  &       & 968.1 & 1000.5 &       & 1032.4 & 1038.3 &       & 6.2   & 3.6   &       & 869.3 & 908.1 &       & 919.3 & 923.5 &       & 5.4   & 1.7 \\
      & 106   &       & 1331.3 & 1336.3 &       & 1632.4 & 1727.0 &       & 18.4  & 22.6  &       & 1035.7 & 1042.1 &       & 1146.2 & 1160.7 &       & 9.6   & 10.2  &       & 894.1 & 908.0 &       & 1014.7 & 1031.9 &       & 11.9  & 12.0 \\
      & 107   &       & 1339.9 & 1351.7 &       & 1638.3 & 1686.8 &       & 18.2  & 19.9  &       & 983.7 & 992.2 &       & 1115.4 & 1118.7 &       & 11.8  & 11.3  &       & 900.8 & 918.7 &       & 1003.4 & 1004.8 &       & 10.2  & 8.6 \\
      & 108   &       & 1227.2 & 1252.2 &       & 1536.5 & 1583.8 &       & 20.1  & 20.9  &       & 814.2 & 818.7 &       & 945.8 & 956.2 &       & 13.9  & 14.4  &       & 697.3 & 709.0 &       & 766.9 & 774.0 &       & 9.1   & 8.4 \\
      & 109   &       & 1228.8 & 1243.1 &       & 1587.2 & 1627.6 &       & 22.6  & 23.6  &       & 902.1 & 908.1 &       & 1026.3 & 1037.2 &       & 12.1  & 12.4  &       & 784.7 & 794.4 &       & 859.5 & 869.3 &       & 8.7   & 8.6 \\
      & 110   &       & 1188.8 & 1213.7 &       & 1620.4 & 1668.6 &       & 26.6  & 27.3  &       & 880.1 & 901.2 &       & 967.2 & 983.2 &       & 9.0   & 8.3   &       & 714.9 & 737.5 &       & 787.4 & 792.2 &       & 9.2   & 6.9 \\
\midrule
\multirow{10}[2]{*}{200} & 111   &       & 1353.5 & 1357.6 &       & 1830.0 & 1859.6 &       & 26.0  & 27.0  &       & 906.9 & 910.8 &       & 1050.7 & 1056.3 &       & 13.7  & 13.8  &       & 750.9 & 755.4 &       & 832.8 & 838.9 &       & 9.8   & 9.9 \\
      & 112   &       & 1406.1 & 1409.3 &       & 1716.5 & 1789.3 &       & 18.1  & 21.2  &       & 926.4 & 957.6 &       & 1099.4 & 1103.6 &       & 15.7  & 13.2  &       & 710.2 & 722.2 &       & 846.0 & 854.6 &       & 16.1  & 15.5 \\
      & 113   &       & 1417.6 & 1418.1 &       & 1731.1 & 1761.2 &       & 18.1  & 19.5  &       & 929.4 & 943.2 &       & 1049.0 & 1059.6 &       & 11.4  & 11.0  &       & 759.1 & 768.2 &       & 810.1 & 814.7 &       & 6.3   & 5.7 \\
      & 114   &       & 1483.5 & 1485.3 &       & 1824.1 & 1879.0 &       & 18.7  & 20.9  &       & 1035.3 & 1038.6 &       & 1193.6 & 1201.6 &       & 13.3  & 13.6  &       & 871.2 & 885.2 &       & 978.6 & 982.8 &       & 11.0  & 9.9 \\
      & 115   &       & 1513.0 & 1522.3 &       & 1823.5 & 1887.8 &       & 17.0  & 19.4  &       & 1086.5 & 1091.4 &       & 1244.3 & 1261.4 &       & 12.7  & 13.5  &       & 914.3 & 916.0 &       & 1077.8 & 1081.0 &       & 15.2  & 15.3 \\
      & 116   &       & 1421.6 & 1434.4 &       & 1820.7 & 1843.2 &       & 21.9  & 22.2  &       & 961.5 & 976.1 &       & 1101.8 & 1110.5 &       & 12.7  & 12.1  &       & 737.4 & 748.4 &       & 857.8 & 861.0 &       & 14.0  & 13.1 \\
      & 117   &       & 1576.2 & 1576.9 &       & 1900.1 & 1969.1 &       & 17.0  & 19.9  &       & 1220.0 & 1220.0 &       & 1438.0 & 1449.1 &       & 15.2  & 15.8  &       & 1104.2 & 1113.4 &       & 1244.8 & 1250.0 &       & 11.3  & 10.9 \\
      & 118   &       & 1458.3 & 1467.9 &       & 1844.1 & 1919.7 &       & 20.9  & 23.5  &       & 1027.2 & 1030.9 &       & 1267.5 & 1276.0 &       & 19.0  & 19.2  &       & 873.1 & 881.1 &       & 1019.8 & 1042.8 &       & 14.4  & 15.5 \\
      & 119   &       & 1423.7 & 1428.1 &       & 1632.8 & 1720.4 &       & 12.8  & 17.0  &       & 1089.6 & 1094.0 &       & 1198.9 & 1216.4 &       & 9.1   & 10.1  &       & 939.3 & 947.8 &       & 1057.0 & 1061.4 &       & 11.1  & 10.7 \\
      & 120   &       & 1302.6 & 1313.9 &       & 1684.9 & 1740.6 &       & 22.7  & 24.5  &       & 891.2 & 891.6 &       & 1100.4 & 1110.5 &       & 19.0  & 19.7  &       & 704.8 & 707.4 &       & 869.8 & 878.5 &       & 19.0  & 19.5 \\
\midrule
\multicolumn{2}{c}{\textbf{Average}} &       &       &       &       &       &       &       & \textbf{16.7} & \textbf{18.9} &       &       &       &       &       &       &       & \textbf{11.9} & \textbf{11.4} &       &       &       &       &       &       &       & \textbf{10.8} & \textbf{10.1} \\
\bottomrule
\end{tabular}
}
\label{tab:TabBenchmarkGonazelezSingleCenterFull}%
\end{table}%

\begin{table}[htbp]
\centering
\caption{Results of the comparison of \ac{MD-SPP-H} and \ac{IGH} (based on the solution published by \citet{Gonzalez-R2020Truck-dronePlanning}), considering $n=50$ up to $250$ double-center distributed customers.}
\renewcommand{\arraystretch}{1.1}
\resizebox{\textwidth}{!}{
\begin{tabular}{ccrrrrrrrrrrrrrrrrrrrrrrrrrrr}
\toprule
\multirow{3}[6]{*}{$n$} & \multirow{3}[6]{*}{ID} &       & \multicolumn{8}{c}{Speed ratio 1}                             &       & \multicolumn{8}{c}{Speed ratio 2}                             &       & \multicolumn{8}{c}{Speed ratio 3} \\
\cmidrule{4-11}\cmidrule{13-20}\cmidrule{22-29}      &       &       & \multicolumn{2}{c}{MD-SPP-H} &       & \multicolumn{2}{c}{IGH} &       & \multicolumn{2}{c}{$\Delta$} &       & \multicolumn{2}{c}{MD-SPP-H} &       & \multicolumn{2}{c}{IGH} &       & \multicolumn{2}{c}{$\Delta$} &       & \multicolumn{2}{c}{MD-SPP-H} &       & \multicolumn{2}{c}{IGH} &       & \multicolumn{2}{c}{$\Delta$} \\
\cmidrule{4-5}\cmidrule{7-8}\cmidrule{10-11}\cmidrule{13-14}\cmidrule{16-17}\cmidrule{19-20}\cmidrule{22-23}\cmidrule{25-26}\cmidrule{28-29}      &       &       & \multicolumn{1}{c}{$z_\text{best}$} & \multicolumn{1}{c}{$z_\text{avg}$} &       & \multicolumn{1}{c}{$z_\text{best}$} & \multicolumn{1}{c}{$z_\text{avg}$} &       & \multicolumn{1}{c}{$\Delta_\text{best}$} & \multicolumn{1}{c}{$\Delta_\text{avg}$} &       & \multicolumn{1}{c}{$z_\text{best}$} & \multicolumn{1}{c}{$z_\text{avg}$} &       & \multicolumn{1}{c}{$z_\text{best}$} & \multicolumn{1}{c}{$z_\text{avg}$} &       & \multicolumn{1}{c}{$\Delta_\text{best}$} & \multicolumn{1}{c}{$\Delta_\text{avg}$} &       & \multicolumn{1}{c}{$z_\text{best}$} & \multicolumn{1}{c}{$z_\text{avg}$} &       & \multicolumn{1}{c}{$z_\text{best}$} & \multicolumn{1}{c}{$z_\text{avg}$} &       & \multicolumn{1}{c}{$\Delta_\text{best}$} & \multicolumn{1}{c}{$\Delta_\text{avg}$} \\
\midrule
\multirow{10}[2]{*}{50} & 71    &       & 997.2 & 1002.4 &       & 1109.9 & 1169.1 &       & 10.2  & 14.3  &       & 719.2 & 720.4 &       & 755.2 & 761.4 &       & 4.8   & 5.4   &       & 603.1 & 610.3 &       & 690.1 & 690.3 &       & 12.6  & 11.6 \\
      & 72    &       & 887.5 & 897.5 &       & 1043.7 & 1091.5 &       & 15.0  & 17.8  &       & 610.0 & 614.5 &       & 694.9 & 700.2 &       & 12.2  & 12.2  &       & 482.1 & 488.1 &       & 594.7 & 602.4 &       & 18.9  & 19.0 \\
      & 73    &       & 807.1 & 811.1 &       & 953.3 & 990.7 &       & 15.3  & 18.1  &       & 574.8 & 581.7 &       & 625.0 & 640.4 &       & 8.0   & 9.2   &       & 513.0 & 515.0 &       & 541.6 & 563.6 &       & 5.3   & 8.6 \\
      & 74    &       & 808.7 & 810.5 &       & 833.8 & 897.9 &       & 3.0   & 9.7   &       & 570.0 & 577.1 &       & 525.0 & 538.2 &       & -8.6  & -7.2  &       & 519.2 & 520.8 &       & 399.0 & 412.6 &       & -30.1 & -26.2 \\
      & 75    &       & 884.3 & 886.2 &       & 937.7 & 1004.0 &       & 5.7   & 11.7  &       & 633.8 & 645.8 &       & 570.2 & 583.6 &       & -11.2 & -10.7 &       & 580.5 & 584.6 &       & 409.9 & 416.8 &       & -41.6 & -40.3 \\
      & 76    &       & 946.5 & 952.4 &       & 1106.1 & 1230.6 &       & 14.4  & 22.6  &       & 662.8 & 668.3 &       & 733.9 & 742.3 &       & 9.7   & 10.0  &       & 533.7 & 543.9 &       & 632.4 & 634.9 &       & 15.6  & 14.3 \\
      & 77    &       & 819.4 & 820.5 &       & 1000.3 & 1030.2 &       & 18.1  & 20.3  &       & 578.7 & 582.7 &       & 627.1 & 637.5 &       & 7.7   & 8.6   &       & 527.8 & 529.3 &       & 534.0 & 535.6 &       & 1.2   & 1.2 \\
      & 78    &       & 861.3 & 875.1 &       & 945.6 & 1056.0 &       & 8.9   & 17.1  &       & 587.7 & 591.8 &       & 599.1 & 619.8 &       & 1.9   & 4.5   &       & 464.9 & 466.5 &       & 425.4 & 439.4 &       & -9.3  & -6.2 \\
      & 79    &       & 855.6 & 872.3 &       & 978.2 & 1016.0 &       & 12.5  & 14.1  &       & 594.7 & 597.9 &       & 645.9 & 653.2 &       & 7.9   & 8.5   &       & 483.4 & 484.9 &       & 548.4 & 558.0 &       & 11.8  & 13.1 \\
      & 80    &       & 937.5 & 947.9 &       & 1052.8 & 1140.1 &       & 10.9  & 16.9  &       & 675.7 & 681.9 &       & 711.1 & 719.4 &       & 5.0   & 5.2   &       & 590.4 & 598.0 &       & 598.3 & 613.6 &       & 1.3   & 2.5 \\
\midrule
\multirow{10}[2]{*}{75} & 81    &       & 928.7 & 938.2 &       & 1063.7 & 1116.5 &       & 12.7  & 16.0  &       & 620.4 & 636.0 &       & 646.9 & 663.3 &       & 4.1   & 4.1   &       & 512.5 & 522.2 &       & 491.4 & 503.6 &       & -4.3  & -3.7 \\
      & 82    &       & 1043.4 & 1046.5 &       & 1274.8 & 1314.9 &       & 18.2  & 20.4  &       & 697.8 & 708.9 &       & 760.8 & 769.5 &       & 8.3   & 7.9   &       & 549.4 & 563.6 &       & 610.9 & 639.9 &       & 10.1  & 11.9 \\
      & 83    &       & 998.1 & 1005.1 &       & 1104.1 & 1159.1 &       & 9.6   & 13.3  &       & 675.2 & 687.6 &       & 676.1 & 687.7 &       & 0.1   & 0.0   &       & 573.8 & 577.1 &       & 538.7 & 543.6 &       & -6.5  & -6.1 \\
      & 84    &       & 1140.3 & 1163.9 &       & 1497.4 & 1553.5 &       & 23.9  & 25.1  &       & 769.7 & 798.5 &       & 889.5 & 907.3 &       & 13.5  & 12.0  &       & 718.5 & 728.3 &       & 696.5 & 716.4 &       & -3.2  & -1.7 \\
      & 85    &       & 1177.6 & 1195.0 &       & 1509.4 & 1569.9 &       & 22.0  & 23.9  &       & 816.1 & 824.3 &       & 978.5 & 999.1 &       & 16.6  & 17.5  &       & 660.4 & 677.7 &       & 795.4 & 800.0 &       & 17.0  & 15.3 \\
      & 86    &       & 980.4 & 997.8 &       & 1104.2 & 1175.2 &       & 11.2  & 15.1  &       & 666.3 & 671.7 &       & 672.4 & 686.2 &       & 0.9   & 2.1   &       & 543.8 & 553.5 &       & 527.9 & 542.0 &       & -3.0  & -2.1 \\
      & 87    &       & 984.5 & 994.9 &       & 1084.6 & 1215.1 &       & 9.2   & 18.1  &       & 679.2 & 689.5 &       & 711.1 & 722.3 &       & 4.5   & 4.5   &       & 571.5 & 584.5 &       & 552.2 & 557.1 &       & -3.5  & -4.9 \\
      & 88    &       & 1281.3 & 1317.7 &       & 1417.4 & 1467.8 &       & 9.6   & 10.2  &       & 887.1 & 898.3 &       & 895.5 & 901.0 &       & 0.9   & 0.3   &       & 724.4 & 736.8 &       & 701.2 & 709.3 &       & -3.3  & -3.9 \\
      & 89    &       & 996.0 & 1015.4 &       & 1352.3 & 1407.1 &       & 26.3  & 27.8  &       & 681.3 & 694.2 &       & 785.5 & 791.6 &       & 13.3  & 12.3  &       & 547.2 & 556.1 &       & 629.8 & 641.1 &       & 13.1  & 13.3 \\
      & 90    &       & 971.6 & 988.3 &       & 1224.5 & 1253.8 &       & 20.6  & 21.2  &       & 646.2 & 654.8 &       & 801.4 & 817.5 &       & 19.4  & 19.9  &       & 548.2 & 552.7 &       & 602.7 & 619.4 &       & 9.0   & 10.8 \\
\midrule
\multirow{10}[2]{*}{100} & 91    &       & 1085.7 & 1101.5 &       & 1466.4 & 1524.4 &       & 26.0  & 27.7  &       & 741.8 & 747.5 &       & 852.8 & 867.7 &       & 13.0  & 13.8  &       & 587.0 & 594.4 &       & 634.0 & 639.9 &       & 7.4   & 7.1 \\
      & 92    &       & 1150.4 & 1168.4 &       & 1348.8 & 1382.5 &       & 14.7  & 15.5  &       & 795.8 & 802.8 &       & 802.5 & 811.8 &       & 0.8   & 1.1   &       & 682.3 & 691.4 &       & 599.4 & 620.5 &       & -13.8 & -11.4 \\
      & 93    &       & 988.1 & 1008.1 &       & 1253.3 & 1337.3 &       & 21.2  & 24.6  &       & 697.4 & 704.6 &       & 749.8 & 765.5 &       & 7.0   & 8.0   &       & 560.2 & 569.4 &       & 531.8 & 538.7 &       & -5.3  & -5.7 \\
      & 94    &       & 1131.6 & 1144.8 &       & 1436.4 & 1498.9 &       & 21.2  & 23.6  &       & 767.9 & 782.5 &       & 860.2 & 885.2 &       & 10.7  & 11.6  &       & 639.4 & 646.7 &       & 682.1 & 683.1 &       & 6.3   & 5.3 \\
      & 95    &       & 1184.7 & 1208.2 &       & 1497.6 & 1609.3 &       & 20.9  & 24.9  &       & 819.2 & 831.9 &       & 941.0 & 954.2 &       & 12.9  & 12.8  &       & 635.9 & 649.6 &       & 758.1 & 768.6 &       & 16.1  & 15.5 \\
      & 96    &       & 1156.0 & 1204.9 &       & 1465.3 & 1548.0 &       & 21.1  & 22.2  &       & 811.1 & 831.5 &       & 854.7 & 873.4 &       & 5.1   & 4.8   &       & 625.3 & 642.7 &       & 624.3 & 630.8 &       & -0.2  & -1.9 \\
      & 97    &       & 1281.7 & 1290.4 &       & 1724.0 & 1791.3 &       & 25.7  & 28.0  &       & 857.8 & 866.8 &       & 1057.8 & 1082.7 &       & 18.9  & 19.9  &       & 713.7 & 724.5 &       & 838.9 & 848.6 &       & 14.9  & 14.6 \\
      & 98    &       & 1095.5 & 1103.6 &       & 1458.3 & 1536.9 &       & 24.9  & 28.2  &       & 752.4 & 766.4 &       & 910.5 & 930.6 &       & 17.4  & 17.6  &       & 646.6 & 653.6 &       & 741.8 & 753.6 &       & 12.8  & 13.3 \\
      & 99    &       & 1107.2 & 1118.4 &       & 1380.2 & 1483.9 &       & 19.8  & 24.6  &       & 748.2 & 764.5 &       & 838.6 & 860.4 &       & 10.8  & 11.1  &       & 621.1 & 630.2 &       & 626.3 & 634.7 &       & 0.8   & 0.7 \\
      & 100   &       & 1240.8 & 1243.5 &       & 1422.5 & 1541.1 &       & 12.8  & 19.3  &       & 829.6 & 846.4 &       & 967.4 & 987.6 &       & 14.2  & 14.3  &       & 664.7 & 678.4 &       & 736.7 & 756.0 &       & 9.8   & 10.3 \\
\midrule
\multirow{10}[2]{*}{175} & 101   &       & 1277.9 & 1285.0 &       & 1936.3 & 2021.4 &       & 34.0  & 36.4  &       & 857.2 & 857.8 &       & 1138.2 & 1155.0 &       & 24.7  & 25.7  &       & 696.2 & 702.4 &       & 892.7 & 901.0 &       & 22.0  & 22.0 \\
      & 102   &       & 1486.0 & 1489.5 &       & 2090.1 & 2193.0 &       & 28.9  & 32.1  &       & 1003.5 & 1011.7 &       & 1192.8 & 1209.3 &       & 15.9  & 16.3  &       & 760.4 & 783.5 &       & 891.9 & 901.0 &       & 14.7  & 13.0 \\
      & 103   &       & 1492.4 & 1497.9 &       & 2154.4 & 2201.3 &       & 30.7  & 32.0  &       & 1043.9 & 1062.1 &       & 1195.4 & 1213.6 &       & 12.7  & 12.5  &       & 807.0 & 820.5 &       & 879.3 & 889.0 &       & 8.2   & 7.7 \\
      & 104   &       & 1480.6 & 1488.3 &       & 2064.8 & 2105.7 &       & 28.3  & 29.3  &       & 985.4 & 997.3 &       & 1182.0 & 1211.3 &       & 16.6  & 17.7  &       & 766.1 & 773.4 &       & 897.7 & 912.7 &       & 14.7  & 15.3 \\
      & 105   &       & 1587.0 & 1592.4 &       & 2146.2 & 2197.9 &       & 26.1  & 27.5  &       & 1067.9 & 1081.1 &       & 1287.4 & 1302.3 &       & 17.1  & 17.0  &       & 835.5 & 847.8 &       & 968.5 & 984.1 &       & 13.7  & 13.9 \\
      & 106   &       & 1505.0 & 1543.9 &       & 1989.3 & 2146.3 &       & 24.3  & 28.1  &       & 1055.3 & 1063.6 &       & 1174.7 & 1199.1 &       & 10.2  & 11.3  &       & 796.2 & 808.5 &       & 857.6 & 874.2 &       & 7.2   & 7.5 \\
      & 107   &       & 1327.1 & 1337.3 &       & 1864.8 & 1945.5 &       & 28.8  & 31.3  &       & 931.2 & 941.6 &       & 1072.5 & 1098.2 &       & 13.2  & 14.3  &       & 719.6 & 729.0 &       & 801.4 & 815.3 &       & 10.2  & 10.6 \\
      & 108   &       & 1550.6 & 1559.1 &       & 2174.0 & 2240.0 &       & 28.7  & 30.4  &       & 1043.8 & 1064.8 &       & 1253.6 & 1285.9 &       & 16.7  & 17.2  &       & 814.8 & 830.5 &       & 975.4 & 981.1 &       & 16.5  & 15.4 \\
      & 109   &       & 1456.4 & 1461.3 &       & 1987.1 & 2064.3 &       & 26.7  & 29.2  &       & 979.6 & 992.0 &       & 1160.4 & 1193.7 &       & 15.6  & 16.9  &       & 767.2 & 781.1 &       & 926.9 & 932.7 &       & 17.2  & 16.3 \\
      & 110   &       & 1576.8 & 1581.9 &       & 2116.1 & 2194.4 &       & 25.5  & 27.9  &       & 1069.7 & 1078.1 &       & 1237.3 & 1250.0 &       & 13.5  & 13.8  &       & 845.8 & 853.2 &       & 828.4 & 842.7 &       & -2.1  & -1.3 \\
\midrule
\multirow{10}[2]{*}{200} & 111   &       & 1662.1 & 1679.6 &       & 2534.2 & 2595.7 &       & 34.4  & 35.3  &       & 1180.2 & 1181.6 &       & 1511.2 & 1527.0 &       & 21.9  & 22.6  &       & 859.1 & 882.3 &       & 1178.8 & 1189.1 &       & 27.1  & 25.8 \\
      & 112   &       & 1668.3 & 1673.5 &       & 2467.7 & 2574.0 &       & 32.4  & 35.0  &       & 1141.5 & 1145.9 &       & 1406.7 & 1424.8 &       & 18.9  & 19.6  &       & 891.2 & 892.5 &       & 1021.2 & 1029.7 &       & 12.7  & 13.3 \\
      & 113   &       & 1690.7 & 1693.6 &       & 2111.8 & 2261.6 &       & 19.9  & 25.1  &       & 1127.4 & 1133.4 &       & 1272.4 & 1305.5 &       & 11.4  & 13.2  &       & 883.2 & 893.5 &       & 925.5 & 933.9 &       & 4.6   & 4.3 \\
      & 114   &       & 1638.0 & 1668.3 &       & 2327.8 & 2415.9 &       & 29.6  & 30.9  &       & 1130.5 & 1131.7 &       & 1325.0 & 1343.9 &       & 14.7  & 15.8  &       & 855.3 & 876.7 &       & 1034.7 & 1040.7 &       & 17.3  & 15.8 \\
      & 115   &       & 1882.8 & 1890.6 &       & 2552.4 & 2694.7 &       & 26.2  & 29.8  &       & 1297.9 & 1303.0 &       & 1460.3 & 1474.8 &       & 11.1  & 11.6  &       & 966.6 & 990.7 &       & 1092.1 & 1100.3 &       & 11.5  & 10.0 \\
      & 116   &       & 1730.3 & 1738.2 &       & 2466.7 & 2586.2 &       & 29.9  & 32.8  &       & 1126.2 & 1134.8 &       & 1358.4 & 1380.5 &       & 17.1  & 17.8  &       & 879.3 & 881.0 &       & 944.1 & 954.2 &       & 6.9   & 7.7 \\
      & 117   &       & 1696.5 & 1698.1 &       & 2539.6 & 2600.2 &       & 33.2  & 34.7  &       & 1136.8 & 1138.0 &       & 1447.9 & 1468.5 &       & 21.5  & 22.5  &       & 864.1 & 865.6 &       & 1072.8 & 1080.6 &       & 19.5  & 19.9 \\
      & 118   &       & 1775.1 & 1785.3 &       & 2585.9 & 2635.2 &       & 31.4  & 32.3  &       & 1220.9 & 1223.5 &       & 1438.9 & 1464.9 &       & 15.1  & 16.5  &       & 931.3 & 942.0 &       & 1033.4 & 1045.1 &       & 9.9   & 9.9 \\
      & 119   &       & 1845.8 & 1849.4 &       & 2600.6 & 2664.3 &       & 29.0  & 30.6  &       & 1261.8 & 1266.8 &       & 1519.8 & 1542.1 &       & 17.0  & 17.9  &       & 949.6 & 958.8 &       & 1171.7 & 1179.7 &       & 19.0  & 18.7 \\
      & 120   &       & 1734.3 & 1734.9 &       & 2450.3 & 2508.0 &       & 29.2  & 30.8  &       & 1183.5 & 1194.9 &       & 1399.4 & 1436.1 &       & 15.4  & 16.8  &       & 913.8 & 916.2 &       & 976.2 & 984.0 &       & 6.4   & 6.9 \\
\midrule
\multicolumn{2}{c}{\textbf{Average}} &       &       &       &       &       &       &       & \textbf{21.3} & \textbf{24.3} &       &       &       &       &       &       &       & \textbf{11.0} & \textbf{11.6} &       &       &       &       &       &       &       & \textbf{6.3} & \textbf{6.5} \\
\bottomrule
\end{tabular}
}
\label{tab:TabBenchmarkGonazelezDoubleCenterFull}%
\end{table}%

\section{Comparison with solution approaches for the FSTSP} \label{app:onlieB}

\paragraph{Comparison with EP-All and SPP-All over single- and double-center instances} Tables \ref{tab:TabBenchmarkAgatzKunduSC} and \ref{tab:TabBenchmarkAgatzKunduDC} report the results achieved by \ac{MD-SPP-H}, EP-All heuristic and SPP-All heuristic, considering the single- and double-center customer distributions and at least 50 customers. The observed improvements are modestly lower than those noted for uniform customer distributions.

%Averaged over all 50 instances with single-center customer distribution, \ac{MD-SPP-H} improves the best-found solution value of EP-All by $\Delta_\text{best}$ = 3.9\% (average over the 50 instances). The average solution value reached by \ac{MD-SPP-H} improves the solution found by EP-All by $\Delta_\text{avg}$= 2.2\% (average over the 50 instances). In the double-center cases, we report average improvements of the best-found solution value and average solution value reached by \ac{MD-SPP-H} of $\Delta_\text{best}$ = 3.6\% and $\Delta_\text{avg}$= 2.1\%, respectively. 

%Comparing \ac{MD-SPP-H} and SPP-All, we report on average over the 50 instances with single-center customer distribution improvements of $\Delta_\text{best}$ = 2.7\% and $\Delta_\text{avg}$=  0.9\% for the best-found and average solution value, respectively. The corresponding values for the double-center instances are $\Delta_\text{best}$ = 3.1\% and $\Delta_\text{avg}$= 1.5\%.

\begin{table}[htbp]
  \centering
  \caption{Results of the comparison of \ac{MD-SPP-H} and \ac{EP-All} and \ac{SPP-All} for various customers ranging from $n=50$ to $250$ and single-center customer distributions.}
    \renewcommand{\arraystretch}{1.1}
    \scalebox{0.6}{%\resizebox{\textwidth}{!}{
\begin{tabular}{ccrrrrrrrrrrrrrr}
\toprule
\multirow{2}[4]{*}{$n$} & \multirow{2}[4]{*}{ID} &       & \multicolumn{3}{c}{MD-SPP-H} &       & \multicolumn{4}{c}{EP-All}    &       & \multicolumn{4}{c}{SPP-All} \\
\cmidrule{4-6}\cmidrule{8-11}\cmidrule{13-16}      &       &       & \multicolumn{1}{c}{$z_\text{best}$} & \multicolumn{1}{c}{$z_\text{avg}$} & \multicolumn{1}{c}{Avg. runtime (s)} &       & \multicolumn{1}{c}{$z$} & \multicolumn{1}{c}{Runtime (s)} & \multicolumn{1}{c}{$\Delta_\text{best}$} & \multicolumn{1}{c}{$\Delta_\text{avg}$} &       & \multicolumn{1}{c}{$z$} & \multicolumn{1}{c}{Runtime (s)} & \multicolumn{1}{c}{$\Delta_\text{best}$} & \multicolumn{1}{c}{$\Delta_\text{avg}$} \\
\midrule
\multirow{10}[2]{*}{50} & \multicolumn{1}{r}{71} &       & 428.2 & 432.7 & 7.6   &       & 446.6 & 1.9   & 4.1   & 3.1   &       & 446.6 & 0.0   & 4.1   & 3.1 \\
      & \multicolumn{1}{r}{72} &       & 483.0 & 484.2 & 3.3   &       & 484.3 & 2.0   & 0.3   & 0.0   &       & 484.3 & 0.1   & 0.3   & 0.0 \\
      & \multicolumn{1}{r}{73} &       & 361.5 & 365.9 & 7.7   &       & 376.8 & 2.1   & 4.0   & 2.9   &       & 376.8 & 0.0   & 4.0   & 2.9 \\
      & \multicolumn{1}{r}{74} &       & 528.2 & 532.1 & 7.2   &       & 582.5 & 1.8   & 9.3   & 8.6   &       & 550.7 & 0.1   & 4.1   & 3.4 \\
      & \multicolumn{1}{r}{75} &       & 552.1 & 561.1 & 9.0   &       & 589.1 & 2.3   & 6.3   & 4.8   &       & 565.4 & 0.1   & 2.3   & 0.8 \\
      & \multicolumn{1}{r}{76} &       & 539.0 & 550.0 & 6.5   &       & 573.7 & 1.9   & 6.0   & 4.1   &       & 563.5 & 0.1   & 4.4   & 2.4 \\
      & \multicolumn{1}{r}{77} &       & 425.3 & 430.6 & 8.5   &       & 442.1 & 3.1   & 3.8   & 2.6   &       & 442.1 & 0.1   & 3.8   & 2.6 \\
      & \multicolumn{1}{r}{78} &       & 541.4 & 557.0 & 7.8   &       & 558.1 & 2.0   & 3.0   & 0.2   &       & 558.1 & 0.0   & 3.0   & 0.2 \\
      & \multicolumn{1}{r}{79} &       & 432.7 & 441.5 & 6.2   &       & 465.4 & 2.6   & 7.0   & 5.1   &       & 451.6 & 0.1   & 4.2   & 2.2 \\
      & \multicolumn{1}{r}{80} &       & 550.8 & 558.5 & 6.2   &       & 573.8 & 3.5   & 4.0   & 2.7   &       & 571.1 & 0.1   & 3.6   & 2.2 \\
\midrule
\multirow{10}[2]{*}{75} & \multicolumn{1}{r}{81} &       & 695.6 & 715.3 & 32.5  &       & 767.7 & 11.3  & 9.4   & 6.8   &       & 701.3 & 0.2   & 0.8   & -2.0 \\
      & \multicolumn{1}{r}{82} &       & 560.5 & 580.1 & 28.4  &       & 624.4 & 10.4  & 10.2  & 7.1   &       & 624.4 & 0.2   & 10.2  & 7.1 \\
      & \multicolumn{1}{r}{83} &       & 616.6 & 624.8 & 34.6  &       & 631.9 & 11.0  & 2.4   & 1.1   &       & 631.9 & 0.2   & 2.4   & 1.1 \\
      & \multicolumn{1}{r}{84} &       & 684.2 & 706.5 & 34.8  &       & 715.5 & 14.3  & 4.4   & 1.3   &       & 710.8 & 0.2   & 3.7   & 0.6 \\
      & \multicolumn{1}{r}{85} &       & 580.3 & 589.8 & 26.0  &       & 588.1 & 10.4  & 1.3   & -0.3  &       & 585.6 & 0.2   & 0.9   & -0.7 \\
      & \multicolumn{1}{r}{86} &       & 720.9 & 737.4 & 45.1  &       & 769.3 & 14.9  & 6.3   & 4.1   &       & 762.5 & 0.5   & 5.5   & 3.3 \\
      & \multicolumn{1}{r}{87} &       & 687.0 & 702.0 & 29.2  &       & 739.2 & 11.7  & 7.1   & 5.0   &       & 738.3 & 0.2   & 6.9   & 4.9 \\
      & \multicolumn{1}{r}{88} &       & 689.9 & 699.9 & 34.1  &       & 695.3 & 11.7  & 0.8   & -0.7  &       & 692.5 & 0.2   & 0.4   & -1.1 \\
      & \multicolumn{1}{r}{89} &       & 639.1 & 650.0 & 30.7  &       & 697.4 & 10.7  & 8.4   & 6.8   &       & 665.4 & 0.2   & 4.0   & 2.3 \\
      & \multicolumn{1}{r}{90} &       & 693.1 & 711.9 & 32.8  &       & 776.8 & 16.3  & 10.8  & 8.4   &       & 733.6 & 0.2   & 5.5   & 3.0 \\
\midrule
\multirow{10}[2]{*}{100} & \multicolumn{1}{r}{91} &       & 778.5 & 793.9 & 78.1  &       & 827.4 & 37.9  & 5.9   & 4.0   &       & 826.5 & 0.8   & 5.8   & 3.9 \\
      & \multicolumn{1}{r}{92} &       & 854.7 & 877.0 & 103.1 &       & 890.7 & 76.6  & 4.0   & 1.5   &       & 887.2 & 1.5   & 3.7   & 1.1 \\
      & \multicolumn{1}{r}{93} &       & 844.9 & 861.5 & 116.0 &       & 934.4 & 44.0  & 9.6   & 7.8   &       & 934.3 & 0.7   & 9.6   & 7.8 \\
      & \multicolumn{1}{r}{94} &       & 947.8 & 971.0 & 117.6 &       & 1015.3 & 78.2  & 6.7   & 4.4   &       & 997.8 & 1.5   & 5.0   & 2.7 \\
      & \multicolumn{1}{r}{95} &       & 738.9 & 746.0 & 54.4  &       & 752.7 & 62.0  & 1.8   & 0.9   &       & 748.9 & 1.5   & 1.3   & 0.4 \\
      & \multicolumn{1}{r}{96} &       & 809.5 & 829.7 & 113.7 &       & 866.1 & 41.3  & 6.5   & 4.2   &       & 861.8 & 0.7   & 6.1   & 3.7 \\
      & \multicolumn{1}{r}{97} &       & 831.6 & 845.5 & 120.1 &       & 851.4 & 70.6  & 2.3   & 0.7   &       & 848.0 & 1.5   & 1.9   & 0.3 \\
      & \multicolumn{1}{r}{98} &       & 785.7 & 803.5 & 131.3 &       & 840.9 & 52.5  & 6.6   & 4.5   &       & 833.1 & 1.6   & 5.7   & 3.6 \\
      & \multicolumn{1}{r}{99} &       & 662.1 & 678.2 & 95.8  &       & 679.7 & 45.3  & 2.6   & 0.2   &       & 679.1 & 0.8   & 2.5   & 0.1 \\
      & \multicolumn{1}{r}{100} &       & 758.3 & 768.1 & 95.9  &       & 797.2 & 61.3  & 4.9   & 3.7   &       & 793.9 & 0.9   & 4.5   & 3.3 \\
\midrule
\multirow{10}[2]{*}{175} & \multicolumn{1}{r}{101} &       & 1025.9 & 1049.3 & 791.2 &       & 1051.6 & 848.8 & 2.4   & 0.2   &       & 1021.8 & 9.5   & -0.4  & -2.7 \\
      & \multicolumn{1}{r}{102} &       & 1052.7 & 1065.8 & 510.4 &       & 1071.6 & 1776.6 & 1.8   & 0.5   &       & 1058.5 & 7.7   & 0.5   & -0.7 \\
      & \multicolumn{1}{r}{103} &       & 1117.8 & 1152.0 & 586.1 &       & 1071.6 & 1776.6 & -4.3  & -7.5  &       & 1058.5 & 7.7   & -5.6  & -8.8 \\
      & \multicolumn{1}{r}{104} &       & 1073.4 & 1088.8 & 591.6 &       & 1132.8 & 544.1 & 5.2   & 3.9   &       & 1119.4 & 7.1   & 4.1   & 2.7 \\
      & \multicolumn{1}{r}{105} &       & 1074.1 & 1115.2 & 793.0 &       & 1093.7 & 536.7 & 1.8   & -2.0  &       & 1060.6 & 7.4   & -1.3  & -5.1 \\
      & \multicolumn{1}{r}{106} &       & 1107.6 & 1132.5 & 604.1 &       & 1143.0 & 695.6 & 3.1   & 0.9   &       & 1138.1 & 7.0   & 2.7   & 0.5 \\
      & \multicolumn{1}{r}{107} &       & 1171.3 & 1200.9 & 795.6 &       & 1217.9 & 440.7 & 3.8   & 1.4   &       & 1217.9 & 7.4   & 3.8   & 1.4 \\
      & \multicolumn{1}{r}{108} &       & 1115.8 & 1132.4 & 693.5 &       & 1142.6 & 487.0 & 2.4   & 0.9   &       & 1126.9 & 8.9   & 1.0   & -0.5 \\
      & \multicolumn{1}{r}{109} &       & 1095.3 & 1112.3 & 646.1 &       & 1129.9 & 424.3 & 3.1   & 1.6   &       & 1121.5 & 10.1  & 2.3   & 0.8 \\
      & \multicolumn{1}{r}{110} &       & 1077.0 & 1093.2 & 616.2 &       & 1103.9 & 564.3 & 2.4   & 1.0   &       & 1103.2 & 9.3   & 2.4   & 0.9 \\
\midrule
\multirow{10}[2]{*}{250} & \multicolumn{1}{r}{111} &       & 1279.5 & 1284.1 & 1467.3 &       & 1279.9 & 3896.2 & 0.0   & -0.3  &       & 1270.8 & 43.6  & -0.7  & -1.0 \\
      & \multicolumn{1}{r}{112} &       & 1278.6 & 1300.1 & 1425.9 &       & 1300.1 & 3351.2 & 1.7   & 0.0   &       & 1264.3 & 35.5  & -1.1  & -2.8 \\
      & \multicolumn{1}{r}{113} &       & 1379.7 & 1396.1 & 1675.0 &       & 1405.4 & 2195.8 & 1.8   & 0.7   &       & 1389.2 & 37.5  & 0.7   & -0.5 \\
      & \multicolumn{1}{r}{114} &       & 1358.5 & 1363.9 & 1325.4 &       & 1369.0 & 2127.4 & 0.8   & 0.4   &       & 1364.0 & 41.6  & 0.4   & 0.0 \\
      & \multicolumn{1}{r}{115} &       & 1342.8 & 1353.3 & 1374.2 &       & 1335.8 & 2728.1 & -0.5  & -1.3  &       & 1328.9 & 47.3  & -1.0  & -1.8 \\
      & \multicolumn{1}{r}{116} &       & 1303.7 & 1351.4 & 1770.5 &       & 1358.2 & 2982.4 & 4.0   & 0.5   &       & 1335.4 & 45.0  & 2.4   & -1.2 \\
      & \multicolumn{1}{r}{117} &       & 1277.1 & 1287.5 & 1684.7 &       & 1295.4 & 2699.1 & 1.4   & 0.6   &       & 1291.2 & 65.2  & 1.1   & 0.3 \\
      & \multicolumn{1}{r}{118} &       & 1263.3 & 1272.6 & 1613.6 &       & 1277.5 & 4917.2 & 1.1   & 0.4   &       & 1271.3 & 120.2 & 0.6   & -0.1 \\
      & \multicolumn{1}{r}{119} &       & 1189.0 & 1199.6 & 1298.9 &       & 1194.7 & 2464.2 & 0.5   & -0.4  &       & 1190.6 & 33.8  & 0.1   & -0.8 \\
      & \multicolumn{1}{r}{120} &       & 1272.9 & 1303.5 & 1650.1 &       & 1313.4 & 2428.5 & 3.1   & 0.7   &       & 1312.3 & 31.6  & 3.0   & 0.7 \\
\midrule
\multicolumn{2}{c}{\textbf{Average}} &       &       &       &       &       &       &       & \textbf{3.9} & \textbf{2.2} &       &       &       & \textbf{2.7} & \textbf{0.9} \\
\bottomrule
\end{tabular}
}
\label{tab:TabBenchmarkAgatzKunduSC}%
\end{table}%

\begin{table}[htbp]
\centering
\caption{Results of the comparison of \ac{MD-SPP-H} and \ac{EP-All} and \ac{SPP-All} for various customers ranging from $n=50$ to $250$ and double-center customer distributions.}
\renewcommand{\arraystretch}{1.1}
    \scalebox{0.6}{%\resizebox{\textwidth}{!}{
\begin{tabular}{ccrrrrrrrrrrrrrr}
\toprule
\multirow{2}[4]{*}{$n$} & \multirow{2}[4]{*}{ID} &       & \multicolumn{3}{c}{MD-SPP-H} &       & \multicolumn{4}{c}{EP-All}    &       & \multicolumn{4}{c}{SPP-All} \\
\cmidrule{4-6}\cmidrule{8-11}\cmidrule{13-16}      &       &       & \multicolumn{1}{c}{$z_\text{best}$} & \multicolumn{1}{c}{$z_\text{avg}$} & \multicolumn{1}{c}{Avg. runtime (s)} &       & \multicolumn{1}{c}{$z$} & \multicolumn{1}{c}{Runtime (s)} & \multicolumn{1}{c}{$\Delta_\text{best}$} & \multicolumn{1}{c}{$\Delta_\text{avg}$} &       & \multicolumn{1}{c}{$z$} & \multicolumn{1}{c}{Runtime (s)} & \multicolumn{1}{c}{$\Delta_\text{best}$} & \multicolumn{1}{c}{$\Delta_\text{avg}$} \\
\midrule
\multirow{10}[2]{*}{50} & \multicolumn{1}{r}{71} &       & 802.0 & 805.4 & 5.1   &       & 846.3 & 1.9   & 5.2   & 4.8   &       & 846.3 & 0.4   & 5.2   & 4.8 \\
      & \multicolumn{1}{r}{72} &       & 788.5 & 797.5 & 8.0   &       & 832.7 & 2.6   & 5.3   & 4.2   &       & 832.7 & 0.4   & 5.3   & 4.2 \\
      & \multicolumn{1}{r}{73} &       & 717.3 & 730.0 & 5.1   &       & 762.2 & 2.5   & 5.9   & 4.2   &       & 762.2 & 0.6   & 5.9   & 4.2 \\
      & \multicolumn{1}{r}{74} &       & 702.0 & 705.5 & 6.3   &       & 707.5 & 4.0   & 0.8   & 0.3   &       & 707.5 & 0.3   & 0.8   & 0.3 \\
      & \multicolumn{1}{r}{75} &       & 825.3 & 833.5 & 6.8   &       & 887.4 & 2.5   & 7.0   & 6.1   &       & 883.9 & 0.7   & 6.6   & 5.7 \\
      & \multicolumn{1}{r}{76} &       & 795.5 & 803.6 & 8.6   &       & 847.2 & 2.9   & 6.1   & 5.2   &       & 847.2 & 0.2   & 6.1   & 5.2 \\
      & \multicolumn{1}{r}{77} &       & 728.0 & 734.8 & 8.4   &       & 742.9 & 2.5   & 2.0   & 1.1   &       & 736.9 & 0.7   & 1.2   & 0.3 \\
      & \multicolumn{1}{r}{78} &       & 734.9 & 748.3 & 6.5   &       & 756.9 & 3.8   & 2.9   & 1.1   &       & 756.9 & 0.3   & 2.9   & 1.1 \\
      & \multicolumn{1}{r}{79} &       & 757.5 & 772.8 & 7.8   &       & 823.9 & 2.1   & 8.1   & 6.2   &       & 823.9 & 0.3   & 8.1   & 6.2 \\
      & \multicolumn{1}{r}{80} &       & 861.7 & 865.7 & 7.0   &       & 883.8 & 2.4   & 2.5   & 2.0   &       & 883.8 & 0.5   & 2.5   & 2.0 \\
\midrule
\multirow{10}[2]{*}{75} & \multicolumn{1}{r}{81} &       & 858.3 & 867.6 & 34.3  &       & 871.4 & 13.9  & 1.5   & 0.4   &       & 864.4 & 0.5   & 0.7   & -0.4 \\
      & \multicolumn{1}{r}{82} &       & 886.1 & 902.0 & 31.9  &       & 936.4 & 17.2  & 5.4   & 3.7   &       & 936.4 & 0.7   & 5.4   & 3.7 \\
      & \multicolumn{1}{r}{83} &       & 862.6 & 872.4 & 32.0  &       & 881.0 & 17.0  & 2.1   & 1.0   &       & 878.6 & 0.5   & 1.8   & 0.7 \\
      & \multicolumn{1}{r}{84} &       & 1115.1 & 1138.5 & 33.2  &       & 1220.5 & 13.3  & 8.6   & 6.7   &       & 1220.5 & 0.5   & 8.6   & 6.7 \\
      & \multicolumn{1}{r}{85} &       & 1054.6 & 1064.2 & 30.1  &       & 1088.8 & 14.6  & 3.1   & 2.3   &       & 1088.5 & 1.1   & 3.1   & 2.2 \\
      & \multicolumn{1}{r}{86} &       & 913.2 & 933.5 & 26.0  &       & 943.8 & 9.4   & 3.2   & 1.1   &       & 937.9 & 0.5   & 2.6   & 0.5 \\
      & \multicolumn{1}{r}{87} &       & 941.7 & 947.2 & 32.2  &       & 963.0 & 12.8  & 2.2   & 1.6   &       & 963.0 & 0.9   & 2.2   & 1.6 \\
      & \multicolumn{1}{r}{88} &       & 1069.0 & 1101.5 & 27.0  &       & 1122.5 & 17.1  & 4.8   & 1.9   &       & 1119.3 & 0.7   & 4.5   & 1.6 \\
      & \multicolumn{1}{r}{89} &       & 908.5 & 931.5 & 38.0  &       & 1017.3 & 11.6  & 10.7  & 8.4   &       & 1017.3 & 1.3   & 10.7  & 8.4 \\
      & \multicolumn{1}{r}{90} &       & 836.5 & 848.5 & 30.6  &       & 865.2 & 11.0  & 3.3   & 1.9   &       & 865.2 & 0.5   & 3.3   & 1.9 \\
\midrule
\multirow{10}[2]{*}{100} & \multicolumn{1}{r}{91} &       & 1042.2 & 1055.0 & 77.6  &       & 1080.7 & 50.8  & 3.6   & 2.4   &       & 1073.6 & 2.9   & 2.9   & 1.7 \\
      & \multicolumn{1}{r}{92} &       & 988.9 & 1009.7 & 99.5  &       & 1028.4 & 40.0  & 3.8   & 1.8   &       & 1018.7 & 1.4   & 2.9   & 0.9 \\
      & \multicolumn{1}{r}{93} &       & 990.9 & 1005.9 & 99.6  &       & 1052.5 & 27.0  & 5.9   & 4.4   &       & 1049.9 & 2.1   & 5.6   & 4.2 \\
      & \multicolumn{1}{r}{94} &       & 1056.3 & 1077.0 & 100.3 &       & 1076.6 & 55.5  & 1.9   & 0.0   &       & 1074.0 & 1.4   & 1.7   & -0.3 \\
      & \multicolumn{1}{r}{95} &       & 1091.8 & 1115.7 & 80.4  &       & 1132.5 & 36.8  & 3.6   & 1.5   &       & 1130.8 & 4.1   & 3.5   & 1.3 \\
      & \multicolumn{1}{r}{96} &       & 1149.4 & 1174.8 & 108.2 &       & 1237.7 & 44.7  & 7.1   & 5.1   &       & 1182.8 & 2.9   & 2.8   & 0.7 \\
      & \multicolumn{1}{r}{97} &       & 1214.6 & 1237.0 & 91.8  &       & 1259.9 & 39.6  & 3.6   & 1.8   &       & 1259.9 & 1.4   & 3.6   & 1.8 \\
      & \multicolumn{1}{r}{98} &       & 1118.8 & 1149.0 & 81.2  &       & 1164.9 & 47.1  & 4.0   & 1.4   &       & 1160.8 & 1.4   & 3.6   & 1.0 \\
      & \multicolumn{1}{r}{99} &       & 1003.0 & 1015.1 & 52.8  &       & 1028.5 & 35.4  & 2.5   & 1.3   &       & 1025.4 & 2.1   & 2.2   & 1.0 \\
      & \multicolumn{1}{r}{100} &       & 1122.9 & 1138.0 & 80.3  &       & 1137.0 & 52.5  & 1.2   & -0.1  &       & 1113.6 & 1.4   & -0.8  & -2.2 \\
\midrule
\multirow{10}[2]{*}{175} & \multicolumn{1}{r}{101} &       & 1269.1 & 1307.3 & 658.7 &       & 1344.9 & 640.0 & 5.6   & 2.8   &       & 1344.3 & 22.0  & 5.6   & 2.7 \\
      & \multicolumn{1}{r}{102} &       & 1491.4 & 1526.0 & 433.0 &       & 1520.6 & 639.1 & 1.9   & -0.4  &       & 1502.2 & 15.1  & 0.7   & -1.6 \\
      & \multicolumn{1}{r}{103} &       & 1524.3 & 1544.4 & 854.1 &       & 1596.2 & 710.2 & 4.5   & 3.2   &       & 1592.9 & 26.2  & 4.3   & 3.0 \\
      & \multicolumn{1}{r}{104} &       & 1397.0 & 1443.7 & 780.4 &       & 1455.0 & 531.0 & 4.0   & 0.8   &       & 1445.4 & 21.1  & 3.3   & 0.1 \\
      & \multicolumn{1}{r}{105} &       & 1543.5 & 1572.0 & 695.0 &       & 1604.7 & 615.6 & 3.8   & 2.0   &       & 1602.6 & 15.2  & 3.7   & 1.9 \\
      & \multicolumn{1}{r}{106} &       & 1549.1 & 1573.3 & 758.1 &       & 1597.6 & 748.8 & 3.0   & 1.5   &       & 1597.1 & 14.3  & 3.0   & 1.5 \\
      & \multicolumn{1}{r}{107} &       & 1389.7 & 1405.2 & 676.4 &       & 1430.2 & 501.2 & 2.8   & 1.7   &       & 1426.3 & 28.5  & 2.6   & 1.5 \\
      & \multicolumn{1}{r}{108} &       & 1511.5 & 1545.6 & 688.6 &       & 1555.5 & 541.4 & 2.8   & 0.6   &       & 1548.1 & 31.1  & 2.4   & 0.2 \\
      & \multicolumn{1}{r}{109} &       & 1447.5 & 1522.3 & 699.4 &       & 1518.1 & 620.6 & 4.7   & -0.3  &       & 1516.2 & 14.6  & 4.5   & -0.4 \\
      & \multicolumn{1}{r}{110} &       & 1541.1 & 1567.8 & 903.3 &       & 1602.0 & 481.1 & 3.8   & 2.1   &       & 1583.3 & 16.9  & 2.7   & 1.0 \\
\midrule
\multirow{10}[2]{*}{250} & \multicolumn{1}{r}{111} &       & 1762.6 & 1779.9 & 1796.4 &       & 1774.7 & 2834.2 & 0.7   & -0.3  &       & 1764.5 & 68.1  & 0.1   & -0.9 \\
      & \multicolumn{1}{r}{112} &       & 1807.6 & 1829.3 & 1402.3 &       & 1819.3 & 3177.3 & 0.6   & -0.5  &       & 1802.9 & 124.0 & -0.3  & -1.5 \\
      & \multicolumn{1}{r}{113} &       & 1706.3 & 1730.3 & 1631.9 &       & 1747.2 & 2436.7 & 2.3   & 1.0   &       & 1728.6 & 69.2  & 1.3   & -0.1 \\
      & \multicolumn{1}{r}{114} &       & 1737.6 & 1752.1 & 1662.3 &       & 1772.8 & 2811.4 & 2.0   & 1.2   &       & 1758.5 & 170.4 & 1.2   & 0.4 \\
      & \multicolumn{1}{r}{115} &       & 1869.5 & 1897.6 & 1708.0 &       & 1973.5 & 3486.4 & 5.3   & 3.8   &       & 1949.4 & 127.6 & 4.1   & 2.7 \\
      & \multicolumn{1}{r}{116} &       & 1807.1 & 1832.6 & 1333.7 &       & 1808.4 & 2557.6 & 0.1   & -1.3  &       & 1808.0 & 65.5  & 0.1   & -1.4 \\
      & \multicolumn{1}{r}{117} &       & 1817.4 & 1836.3 & 1420.0 &       & 1831.7 & 2575.2 & 0.8   & -0.3  &       & 1821.2 & 94.4  & 0.2   & -0.8 \\
      & \multicolumn{1}{r}{118} &       & 1813.8 & 1823.2 & 1354.1 &       & 1841.1 & 2168.5 & 1.5   & 1.0   &       & 1814.2 & 63.0  & 0.0   & -0.5 \\
      & \multicolumn{1}{r}{119} &       & 1820.4 & 1851.3 & 1597.4 &       & 1823.4 & 4099.7 & 0.2   & -1.5  &       & 1805.1 & 69.1  & -0.8  & -2.6 \\
      & \multicolumn{1}{r}{120} &       & 1761.9 & 1787.4 & 1680.2 &       & 1827.0 & 2910.7 & 3.6   & 2.2   &       & 1790.2 & 101.8 & 1.6   & 0.2 \\
\midrule
\multicolumn{2}{c}{\textbf{Average}} &       &       &       &       &       &       &       & \textbf{3.6} & \textbf{2.1} &       &       &       & \textbf{3.1} & \textbf{1.5} \\
\bottomrule
\end{tabular}%
}
\label{tab:TabBenchmarkAgatzKunduDC}%
\end{table}%

\newpage
\section{Impact of Individual Parameters and Their Interaction on Time Savings and Delivery System Performance} \label{app:MI}

In the following, we present the supplementary results of the full factorial experiment we conducted. The results are sorted based on the customer distribution: uniform, single- and double-center. If not included in the main manuscript, we present plots summarizing the observed percentage time savings, the maximum drone flight duration of a single flight, and the total drone flight duration. We average the results over ten instances for each scenario.

%% uniform
\begin{figure}[H]
    \RawFloats
    \centering
    \includegraphics{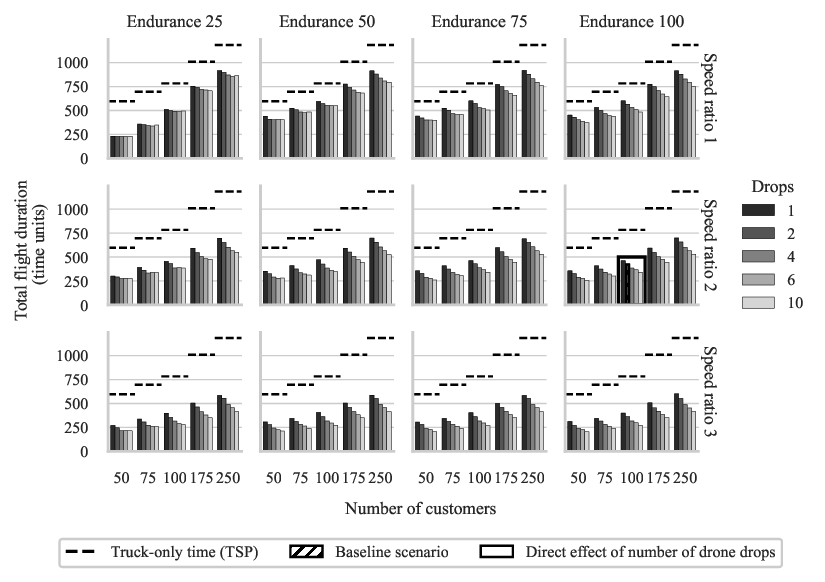}
    \caption{Comparison of average observed total drone flight duration (i.e., sum over all drone flights) over ten instances and truck-only delivery system completion time for instances with uniform customer distribution}  
    \label{fig:FigMIUniformTotalFlightDuration}
\end{figure}

%% single-center
\begin{figure}[H]
    \RawFloats
    \centering
    %[width=0.84\textwidth]
    \includegraphics{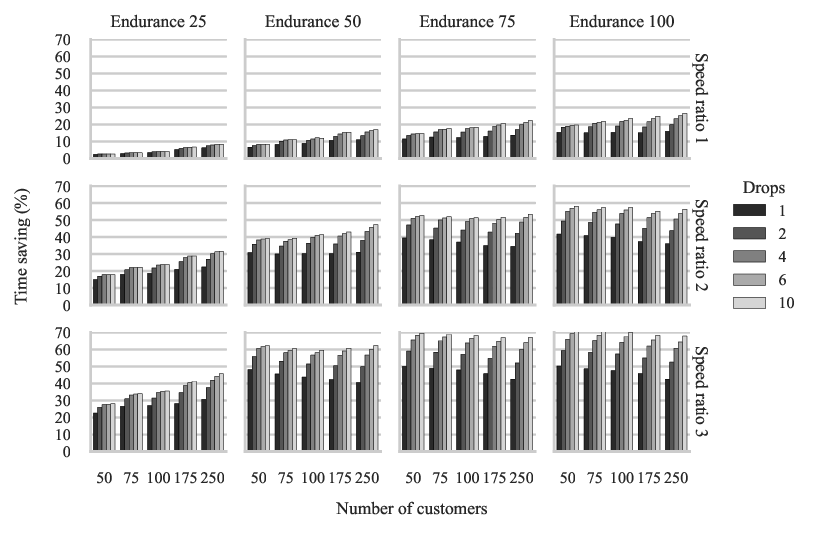}
    \caption{Average percentage time savings over ten problem instances from a truck-and-drone system compared to the truck-only alternative for varying parameter settings and single-center customer distribution.}
    \label{fig:FigMISingleTimeRatio}
\end{figure}

\begin{figure}[H]
    \RawFloats
    \centering
    \includegraphics{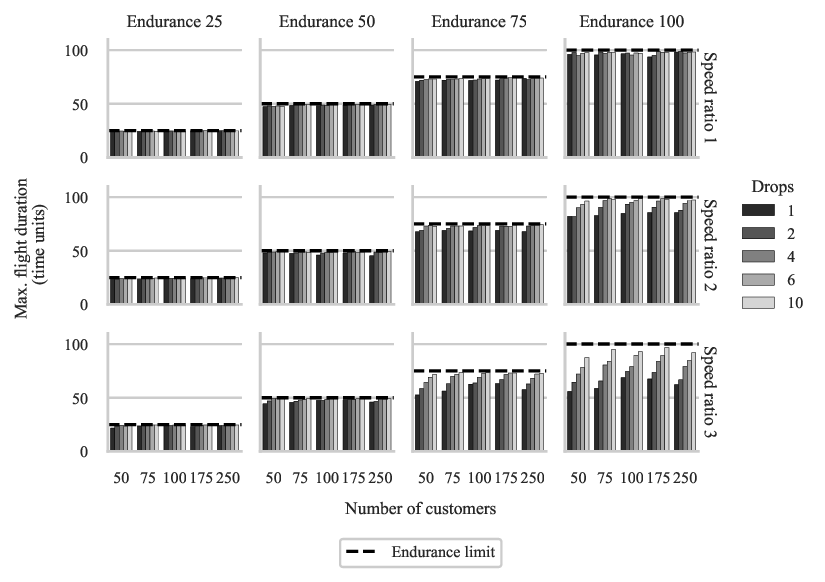}
    \caption{Comparison of average observed maximum drone flight duration over ten instances and imposed drone flight endurance limit for instances with single-center customer distribution}  
    \label{fig:FigMISingleMaxFlightDuration}
\end{figure}

%% double-center
\begin{figure}[H]
    \RawFloats
    \centering
    \includegraphics{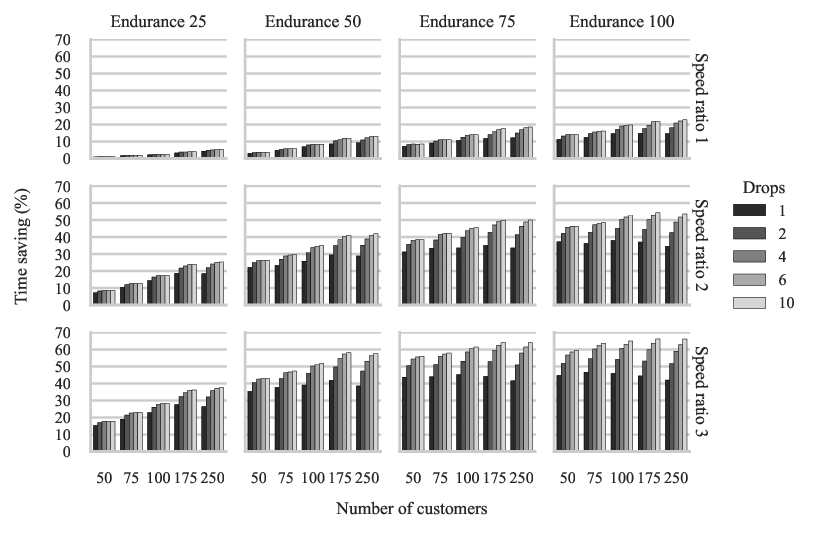}
    \caption{Average percentage time savings over ten problem instances from a truck-and-drone system compared to the truck-only alternative for varying parameter settings and double-center customer distribution.}  
    \label{fig:FigMIDoubleTimeRatio}
\end{figure}

\begin{figure}[H]
    \RawFloats
    \centering
    \includegraphics{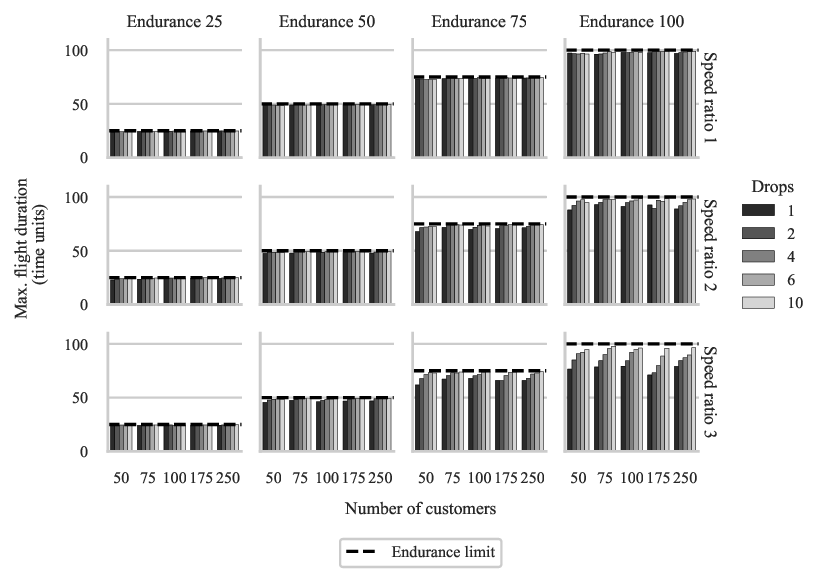}
    \caption{Comparison of average observed maximum drone flight duration over ten instances and imposed drone flight endurance limit for instances with double-center customer distribution}  
    \label{fig:FigMIDoubleMaxFlightDuration}
\end{figure}

\section{Experiments on Real-World Instances}
\label{app:realworld}

To test the applicability of the \ac{MD-SPP-H} to real-world delivery scenarios, we solve \ac{FSTSP-MD} instances generated based on the routing dataset of the \emph{2021 Amazon Last Mile Routing Research Challenge} \citep{Merchan20242021Set}. These instances feature real-world customer locations served in last-mile delivery routes across five major US cities. For our experiment, we randomly select 80 routes, with the number of customer locations per delivery route ranging from 60 to 215. Truck travel times are given as the average time needed between locations, as recorded in the dataset. The drone flight endurance and speed are set to 1800 seconds and 10 m/s, respectively, based on the real-world case study by \cite{Kang2021AnProblem}. The drone travels on Euclidean distance. We vary the number of possible drone drops from 1 to 10, assuming all customers are drone-eligible and have unit demand. For service times at customer locations, we use two settings: first, assuming no service times for both trucks and drones to compare with results from the stylized experiments presented in Section \ref{sec:managerial_insights}; and second, adopting the service times from \cite{Kang2021AnProblem} of 60 seconds for trucks and 30 seconds for drones. We summarize the experiment setting in Table \ref{tab:ExpSetting_Amazon}. In total our experiment design gives rise to 400 routing instances with and without service times. We provide the problem instances and detailed results in our GitHub repository at \url{https://github.com/schasarah/FSTSP-MD-Experiment-Results.git}.

\begin{table}[H]
    \centering
    \renewcommand{\arraystretch}{1.1}
    \caption{Summary of parameter choices for numerical experiments on real-world instances based on the dataset provided by \cite{Merchan20242021Set}.}
    \scalebox{0.8}{
    \begin{tabular}{lll}
    \toprule
    Type  & Parameter & Value \\
    \midrule
    \multirow{3}[1]{*}{General} 
          & Total number of instances  & 80 \\
          & Customer distribution & Based on real-world Amazon routes\\
          & Number of customers  & From 60 to 215 customers   \\
    \midrule
    \multirow{8}[1]{*}{Vehicle} & Drone eligibility  & 100\% \\
          & Truck travel time  &  Based on recorded avg. travel times \\
          & Drone travel distance &  Euclidean [m] \\
          & Drone speed &  10 [m/s] \\
          & Flight endurance &   1800 [s] \\
          & Number of drops  &   $\{1,2,4,6,10\}$  \\
          & Service time truck/drone  &  0s/0s and 60s/30s \\
    \midrule
    \multirow{2}[1]{*}{CPU} & Run-time limit  &  10 min \\
          & Runs  & 1 \\
    \bottomrule
    \end{tabular}%
    }
  \label{tab:ExpSetting_Amazon}%
\end{table}%

In the real-world instances, the average speed ratio between the truck and the drone across the 80 route instances is 2.14. Table \ref{tab:real-world_completion_time_savings} compares the average percentage completion time savings of the real-world instances with the average savings achieved for the uniform customer distribution in the stylized dataset of \cite{Agatz2018OptimizationDrone} with drone speed ratios of 2 and 3. The completion time savings for the real-world instances without additional service times align closely with those based on the stylized dataset without service times and similar speed ratios. We observe even slightly stronger diminishing marginal returns for an increasing number of drops for real-world instances. When additional service times are included, we observe even higher savings. 

\begin{table}[ht]
  \centering
  \renewcommand{\arraystretch}{1.2}
  \caption{Comparison of completion time savings for the real-world instances and the stylized datasets.}
  \scalebox{0.85}{
    \begin{tabular}{crrrrr}
    \toprule
    \multirow{2}[3]{*}{Drops} & \multicolumn{2}{c}{Real-world (speed ratio 2.14)} & \multicolumn{1}{c}{Stylized uniform (speed ratio 2)} & \multicolumn{1}{c}{Stylized uniform (speed ratio 3)} \\
    \cmidrule(lr){2-3} \cmidrule(lr){4-5}
          & \multicolumn{1}{c}{No service time} & \multicolumn{1}{c}{Service time} & \multicolumn{1}{c}{No service time} & \multicolumn{1}{c}{No service time} \\
    \midrule
    \multicolumn{1}{r}{1} & 37.04 & 45.26 & 33.11 & 39.56 \\
    \multicolumn{1}{r}{2} & 43.09 & 53.05 & 41.32 & 48.80 \\
    \multicolumn{1}{r}{4} & 47.90 & 56.65 & 48.71 & 56.83 \\
    \multicolumn{1}{r}{6} & 50.20 & 57.76 & 52.18 & 60.73 \\
    \multicolumn{1}{r}{10} & 52.18 & 58.78 & 55.76 & 64.71 \\
    \bottomrule
    \end{tabular}%
    }
  \label{tab:real-world_completion_time_savings}
\end{table}

\begin{comment}
    \begin{table}[ht]
\color{blue}
  \centering
  \renewcommand{\arraystretch}{1.2}
  \caption{\textcolor{blue}{Comparison of average percentage time savings and truck distance savings for the real-world instances (80 delivery routes) and the numerical experiment data.}}
  \scalebox{0.85}{
    \begin{tabular}{crrrrrr}
    \toprule
    \multirow{2}[3]{*}{\textbf{Drops}} & \multicolumn{2}{c}{\textbf{Real-world (speed ratio 2.14)}} & \multicolumn{2}{c}{\textbf{Stylized uniform (speed ratio 2)}} & \multicolumn{2}{c}{\textbf{Stylized uniform (speed ratio 3)}} \\
    \cmidrule(lr){2-3} \cmidrule(lr){4-5} \cmidrule(lr){6-7}
          & \multicolumn{1}{c}{Time} & \multicolumn{1}{c}{Truck distance} & \multicolumn{1}{c}{Time} & \multicolumn{1}{c}{Truck distance} & \multicolumn{1}{c}{Time} & \multicolumn{1}{c}{Truck distance} \\
    \midrule
    \multicolumn{1}{r}{1} & 37.04 & 16.01 & 33.11 & 34.40 & 39.56 & 40.80 \\
    \multicolumn{1}{r}{2} & 43.09 & 19.98 & 41.32 & 42.66 & 48.80 & 49.93 \\
    \multicolumn{1}{r}{4} & 47.90 & 23.49 & 48.71 & 49.65 & 56.83 & 57.80 \\
    \multicolumn{1}{r}{6} & 50.20 & 25.27 & 52.18 & 52.98 & 60.73 & 61.60 \\
    \multicolumn{1}{r}{10} & 52.18 & 26.73 & 55.76 & 56.45 & 64.71 & 65.29 \\
    \bottomrule
    \end{tabular}%
    }
  \label{tab:realworldcomparison}
\end{table}

\end{comment}
\end{appendix}
\end{document}